\documentclass[a4paper]{scrartcl}
\usepackage[utf8]{inputenc}
\usepackage[english]{babel}
\usepackage{xcolor}
\usepackage{graphicx}
\usepackage{hyperref}
\usepackage[T1]{fontenc}
\usepackage{lmodern}
\usepackage{rotating}
\usepackage{tikz}
\usepackage{tikz-cd}
\usepackage[backend=biber,sortlocale=en_US,url=false,doi=false,isbn=false,giveninits=true,style=alphabetic,maxbibnames=1000,date=year,sorting=nyt]{biblatex}
\bibliography{references.bib}
\DeclareFieldFormat*{title}{\mkbibemph{#1\isdot}}
\DeclareFieldFormat*{journaltitle}{#1\isdot}
\renewbibmacro{in:}{}
\usepackage{amsmath}
\usepackage{amsfonts}
\usepackage{amssymb}
\usepackage{amsthm}
\usepackage{mathtools}
\usepackage{stmaryrd}
\usepackage{enumerate}
\usepackage{verbatim}
\usepackage{dsfont}
\usepackage{mathabx}
\usepackage{booktabs}
\usepackage{CJKutf8}
\usepackage[normalem]{ulem}
\DeclareMathOperator\defect{def}
\DeclareMathOperator\Gal{Gal}

\DeclareMathOperator\Hom{Hom}
\DeclareMathOperator\Spec{Spec}

\DeclareMathOperator\PGL{PGL}

\DeclareMathOperator\wt{wt}

\DeclareMathOperator\avg{avg}
\DeclareMathOperator\supp{supp}

\DeclareMathOperator\LP{LP}
\DeclareMathOperator\cl{cl}
\DeclareMathOperator\Adm{Adm}

\DeclareMathOperator\GSp{GSp}

\DeclareMathOperator\rk{rk}
\DeclareMathOperator\ad{ad}
\DeclareMathOperator\codim{codim}
\DeclareMathOperator\depth{depth}
\DeclareMathOperator\Res{Res}
\DeclareMathOperator\perf{perf}
\DeclareMathOperator\length{length}
\DeclareMathOperator\indec{indec}
\def\mi{\mu,\indec}
\def\GL{{\mathrm{GL}}}

\def\dom{{\mathrm{dom}}}

\hypersetup{colorlinks=false, linkbordercolor={white}, hidelinks, pdfauthor={Felix Schremmer and Eva Viehmann},pdftitle={Affine Deligne-Lusztig varieties beyond the minute case}}
\usepackage{csquotes}
\author{Felix Schremmer and Eva Viehmann}
\date{\today}
\title{Affine Deligne-Lusztig varieties beyond the minute case}
\numberwithin{equation}{section}
\newtheorem{theorem}[equation]{Theorem}
\newtheorem*{theorem*}{Theorem}
\newtheorem{alphatheorem}{Theorem}

\newtheorem{proposition}[equation]{Proposition}
\newtheorem{lemma}[equation]{Lemma}

\theoremstyle{definition}
\newtheorem{definition}[equation]{Definition}
\newtheorem*{definition*}{Definition}

\theoremstyle{remark}
\newtheorem{example}[equation]{Example}
\newtheorem{remark}[equation]{Remark}
\def\abs#1{{\left\lvert{#1}\right\rvert}}

\let\oldqedsymbol\qedsymbol
\def\qedaddendum{}
\def\qedsymbol{\oldqedsymbol\qedaddendum}

\def\af{{\mathrm{af}}}
\def\rightqed{\pushQED{\qed}\qedhere\popQED}

\def\rm{\mathrm}

\begin{document}

\maketitle

\begin{abstract}
Affine Deligne-Lusztig varieties in the fully Hodge-Newton decomposable (or minute) case are the only larger class of ADLVs which could be described completely in the past. Instances of them play important roles in arithmetic geometry, from Harris-Taylor's proof of the local Langlands correspondence to applications in the Kudla program. We study generalizations for many of the equivalent conditions characterizing them to obtain in this way a larger class of ADLVs that still have a similarly good and computable description of their geometry. To generalize the minute condition itself, we introduce the notion of depth for a Shimura datum - the minute cases being those of depth bounded by 1, the cases we study being the ones of depth less than 2.
\end{abstract}

Central geometric questions within the Langlands program, and thus for arithmetic geometry, are to describe the geometry of Shimura varieties and their reduction modulo $p$, and in the function field case as well as in the recent context of Fargues-Scholze's geometrization of the Langlands program, to describe the geometry of moduli spaces of local $G$-shtukas. An important tool for such questions are Newton stratifications. The geometry of Newton stratifications in the special fiber of such moduli spaces is best studied via its close relation to affine Deligne-Lusztig varieties (ADLVs), certain group-theoretically defined (perfect) subschemes of affine Grassmannians or affine flag varieties which are associated with a parahoric sub-group scheme $K$, a cocharacter (which in the case of Shimura varieties is minuscule and induced from the Shimura datum) and a Frobenius-conjugacy class. 

Only very fundamental questions concerning the geometry of affine Deligne-Lusztig varieties are solved in general - in the Iwahori case, even to determine the dimension or the set of irreducible components of all ADLVs is currently not within reach. However, under certain conditions one has a very elegant description in terms of a natural decomposition of the ADLV into locally closed pieces which are isomorphic to classical Deligne-Lusztig varieties and such that the closure relations are given in terms of the Bruhat-Tits building of an associated group. First examples of this phenomenon are the theses of Kaiser \cite{Kaiser} for $\GSp_{4}$ and Vollaard \cite{Vollaard2010}, later generalized by Vollaard-Wedhorn \cite{Vollaard2011}, for unitary groups of signature $(1,n-1)$ at an inert prime. Further explicit examples are unitary groups of signature $(1,n-1)$ at a ramified prime considered by Rapoport-Terstiege-Wilson \cite{RapoportTerstiegeWilson}. Explicit descriptions of affine Deligne-Lusztig varieties and of the geometry of the associated Shimura varieties have striking applications in several areas of the Langlands program. Examples range from Harris-Taylor's proof of the local Langlands correspondence for $\GL_n$ \cite{Harris2001}, via the Kudla program through computations of intersection numbers of special cycles \cite{Kudla2011} and Xiao-Zhu's work on the Tate conjecture for Shimura varieties \cite{Xiao2017} to Wei Zhang's Arithmetic Fundamental Lemma \cite{Rapoport2013}.

In \cite{Goertz2019},  Görtz, He and Nie give a general classification of Shimura data $(G,\mu)$ at a prime $p$ such that the basic ADLV satisfies the above-mentioned condition, and give several equivalent characterizations. These pairs are called fully Hodge-Newton decomposable (according to the characterization by the property that each non-basic affine Deligne-Lusztig variety for $(G,\mu)$ satisfies a variant of Kottwitz's Hodge-Newton decomposition \cite{Kottwitz2003}). Full Hodge-Newton decomposability is a condition that also plays a role when studying adic Newton strata as defined by Caraiani-Scholze \cite{Caraiani2017}, a key geometric ingredient in the geometrization of the Langlands correspondence by Fargues-Scholze \cite{Fargues2021}. Indeed, Chen-Fargues-Shen \cite{Chen2021} prove that $(G,\mu)$ is fully Hodge-Newton decomposable if and only if the weakly admissible locus and the admissible locus in the affine Schubert cell for $\mu$ in the $B_{{\text{dR}}}^+$-affine Grassmannian for $G$ coincide.

Now that a good understanding of the fully Hodge-Newton decomposable case is available, several authors have also developed explicit descriptions in special cases that no longer satisfy these conditions, but that still promise similarly important applications. Fox and Imai \cite{FoxImai2021} and Fox, Howard and Imai \cite{FHI2023} describe the irreducible components of affine Deligne-Lusztig varieties for the groups GU$(2,n-2)$ in the unramified case in terms of vector bundles over classical Deligne-Lusztig varieties. Trentin \cite{Trentin2023} studies ADLV for unitary groups of signature $(2,4)$ at a ramified prime. She proves that all irreducible components are closures of vector bundles (of rank at most 2) over generalized classical Deligne-Lusztig varieties. Shimada \cite{Shimada2024_pct} considers ADLV for $\GL_n$ and classifies and studies cases of coweights $\mu$ which are of so-called positive Coxeter type. In these cases, the $J$-stratification of Chen-Viehmann \cite{Chen2018} induces a decomposition of the ADLV into locally closed pieces that are products of an affine space with a classical Deligne-Lusztig variety of Coxeter type.

The goal of the present paper is to find a uniform group-theoretic explanation for these and to classify the pairs $(G,\mu)$ to which many of the various conditions that are equivalent to full Hodge-Newton decomposability can be extended in a suitably modified way. Particularly simple and at the same time of independent interest is the generalization of the minute criterion. It is based on the notion of depth, which should be seen as a numerical measure of the expected complexity of a local Shimura datum.

For the sake of simplicity, we define it here only for split groups, for the general case see Definition \ref{defdepth}.
\begin{definition*}
Let $G$ be a split reductive group over $\mathbb F_q$, let $T$ be a maximal torus and let $\mu\in X_\ast(T)$ be a dominant cocharacter. Then the depth of $(G,\mu)$ is
\begin{align*}
\depth(G,\mu) := \max_{\omega}\,\langle \mu,\omega\rangle\in\mathbb Q,
\end{align*}
where the maximum is taken over all fundamental weights.
\end{definition*}
A pair $(G,\mu)$ is fully Hodge-Newton decomposable if and only if it satisfies the minute criterion of \cite[Def.~3.2 and Thm.~3.3]{Goertz2019} which with the above definition (and its generalization to non-split groups) reads $\depth(G,\mu)\leq 1$. Using the notion of depth, one can also rephrase a number of other results on the complexity of affine Deligne-Lusztig varieties. The condition $\depth(G,\mu)=0$ is equivalent to $\mu$ being central. In this case, for any Frobenius stable compact open subgroup $K\subseteq G(\breve F)$ containing a fixed Frobenius stable Iwahori subgroup,  $X^K_{\mu}(\mu(t))$ is the only non-empty affine Deligne-Lusztig variety for the coweight $\mu$ and it is by a direct calculation isomorphic to $G(F)/(G(F)\cap K)$. By \cite{Goertz2019} 5.5, $\depth(G,\mu)<1$ if and only if $(G,\mu)$ is of so-called Harris-Taylor type. In this case one has an explicit classification, reducing essentially to the cases considered by Harris and Taylor in their proof of the local Langlands correspondence for $\GL_n$ \cite{Harris2001}. In particular, all affine Deligne-Lusztig varieties for $G$, $\mu$ and any $[b]\in B(G,\mu)$ are then still 0-dimensional. In a slightly different direction, Lau-Nicole-Vasiu \cite{Lau2013} provide explicit bounds on the isomorphism cutoff and the isogeny cutoff of $p$-divisible groups which can be reformulated in terms of the depth of the associated unitary Shimura datum.

In this paper, we study the pairs $(G,\mu)$ such that $\depth(G,\mu)<2$. Our first main result, Theorem \ref{prop:classification}, gives a complete classification of all such cases in terms of the associated Dynkin diagrams, after some initial reduction steps. It turns out that almost all new cases (i.e., which are not fully Hodge-Newton decomposable) are of Dynkin type $A$, with two exceptions of type $C_3$ and $D_5$. In particular, we obtain four infinite families of minuscule pairs $(G,\mu)$ (that is, Shimura data) of type $A$.

In Theorem \ref{thm:EOGeometry}, we study the geometric properties of the associated affine Deligne-Lusztig varieties for unramified $G$. For the Iwahori ADLVs associated to Ekedahl-Oort strata in a hyperspecial ADLV, we obtain the following description. 
\begin{alphatheorem}[{Cf.\ Theorem~\ref{thm:EOGeometry}}]\label{thm:introEO}
    Let $\mu\in X_\ast(T)$ be a dominant cocharacter such that $\depth(G,\mu)<2$, and let $x\in\Adm(\mu)^K$, where $K$ is hyperspecial and $\sigma$-stable.
    \begin{enumerate}[(a)]
    \item The set
    $        B(G)_x := \{[b]\in B(G)\mid X_x(b)\neq\emptyset\}$
    has the form $$B(G)_x = \{[b]\in B(G)\mid [b_{x,\min}]\leq [b]\leq [b_{x,\max}]\}$$ for uniquely determined and explicitly described elements $[b_{x,\min}], [b_{x,\max}]\in B(G)$.
    \item For $[b]\in B(G)_x$, we have an explicit formula for $\dim X_x(b)$, as detailed in Theorem \ref{thm:EOGeometry}.
    \item For any $[b]\in B(G)_x$, the ADLV $X_x(b)$ is equidimensional and the group $J_b(F)$ acts transitively on the set of irreducible components. Up to perfection, each irreducible component of $X_x(b)$ is a direct product of copies of $\mathbb A^1$ and $\mathbb G_m$ with an irreducible component of a classical Deligne-Lusztig variety.
    \end{enumerate}
\end{alphatheorem}
The main step in the proof is to show that the elements $x\in \Adm(\mu)^K$ are all of geometric Coxeter type in the sense of \cite{Nie2025}, and then use the results of loc.~cit. to conclude. For this, we prove that almost all occurring elements $x\in\Adm(\mu)^K$ are already of positive Coxeter type in the sense of Schremmer, Shimada and Yu \cite{Schremmer2023_coxeter}. Among the cases of depth less than 2 there are two exceptions (for groups of type $C_3$ resp.~$D_5$ and in each case one particular $x\in \widetilde W$) which are not of positive Coxeter type, compare Examples \ref{ex:39} and \ref{ex:310}. There, the condition of geometric Coxeter type is verified by direct computations.

We also study Iwahori ADLV in affine flag varieties and their Kottwitz-Rapoport strata. In contrast to the fully Hodge-Newton decomposable case, the shape of the results varies with the parahoric level. Also in the Iwahori case, bounding the depth by $2$ allows us to describe the foundational geometric properties of all Kottwitz-Rapoport strata, such as nonemptiness and dimension. This is most interesting for the infinite families of type $A$, as the two sporadic cases of types $C_3$ resp.\ $D_5$ can be handled through individual, finite computations.
\begin{alphatheorem}\label{thm:introKR}
	Let $G$ be a quasi-split group such that every irreducible component of the Dynkin diagram has type $A$.
	Let $\mu\in X_\ast(T)$ be a dominant cocharacter such that $\depth(G,\mu)<2$, and let $x\in\Adm(\mu)$.
	\begin{enumerate}[(a)]
		\item The set
		$        B(G)_x := \{[b]\in B(G)\mid X_x(b)\neq\emptyset\}$
		has the form $$B(G)_x = \{[b]\in B(G)\mid [b_{x,\min}]\leq [b]\leq [b_{x,\max}]\}$$ for uniquely determined and explicitly described elements $[b_{x,\min}], [b_{x,\max}]\in B(G)$.
		\item For $[b]\in B(G)_x$, we have an explicit formula for $\dim X_x(b)$, as detailed in Lemma~\ref{lem:genericClassDistance}.
		\item For any $[b]\in B(G)_x$, the ADLV $X_x(b)$ is equidimensional.
	\end{enumerate}
\end{alphatheorem}

For ADLV $X(\mu,b,K)$ associated to general $\sigma$-stable parahoric level $K$, there is no general dimension formula known yet, and even for special cases of $(G,\mu)$, the literature is rather sparse. We give a recipe to obtain dimension formulas for most ADLV $X(\mu,b,K)$ where $\depth(G,\mu)<2$ in Section \ref{sec:5}. We explicitly formulate these dimension formulas for the infinite families $(G,\mu)$ where $G$ is split of type $A$ and $K$ is an Iwahori subgroup. Before, such formulas for the dimension of $X(\mu,b,K)$ only existed for the case that $K$ is hyperspecial, in the fully Hodge-Newton decomposable case where all non-basic such ADLVs are zero-dimensional, or for superregular $\mu$.

\begin{alphatheorem}[{Cf.\ Theorem~\ref{thm:depth2GLn}}]\label{thm:introUnionADLV}
	Let $G = \GL_{m+1}$ and let $\mu\in X_\ast(T)\cong \mathbb Z^{m+1}$ with the usual conventions. We define a constant $D\in \frac 12 \mathbb Z$ in the following cases:
	\begin{enumerate}[(i)]
		\item If $\mu = (2,0,\dotsc,0)$, then $D:=m$.
		\item If $m\geq 3$, $\mu = (1,1,0,\dotsc,0)$ and $J\cap \{s_0,s_2\}=\emptyset$, then $D:=m$.
		\item If $m\geq 2$, $\mu = (2,0,\dotsc,0,-1)$ and $J\cap \{s_0, s_1\}=\emptyset$, then $D:=\frac{5m-2}2$.
		\item If $m\geq 3$, $\mu = (1,1,0,\dotsc,0,-1)$ and $J\cap \{s_0, s_m\}=\emptyset$, then $D:=\frac{5m-10}2$.
		\item If $m=4$, $\mu = (2,1,0,0,0)$ and $J=\emptyset$, then $D:=7/2$.
	\end{enumerate}
	For every $[b]\in B(G,\mu)$ such that $(G,\mu,[b])$ is Hodge-Newton indecomposable, we get
	\begin{align*}
		\dim X(\mu,b,I) = D -\langle \nu(b),\rho\rangle - \frac 12\defect(b).
	\end{align*}
	In case (v), the ADLV $X(\mu,b,I)$ is equidimensional.
\end{alphatheorem}

In Sections \ref{sec:4} and \ref{sec:5} we study two new properties used to prove Theorems \ref{thm:introKR} and \ref{thm:introUnionADLV}. The first property, called (CL1BC), is concerned with the variance of the generic Newton point if one replaces $x$ by a lower Bruhat cover. We prove in Lemma \ref{lem:CL1BC_cordial} that (CL1BC) implies the conclusions of Theorem~\ref{thm:introKR}. Then, we show in Theorem \ref{thm:depth2L1BC} that for restrictions of scalars of $\PGL_n$, $\mu$ of depth less than 2 and $[b]\in B(G,\mu)$ Hodge-Newton indecomposable that the property (CL1BC) is satisfied. The second property is the (ING) property of Lemma \ref{lem:defing} and Definition \ref{def:ing} which is concerned with the maximal Hodge-Newton indecomposable Newton stratum in $K\Adm(\mu)K$. In Theorem \ref{thm:INGoverview} we show that (ING) is satisfied in many cases of depth less than 2. In Section \ref{sec:purity} we generalize purity of the Newton stratification and the foliation structure on Newton strata to our context. The combination of (ING) and (CL1BC) then yields dimension formulas for $X(\mu,b,K)$ as described in Proposition~\ref{prop:ING_L1BC}. \\ 

So far, we discussed conditions that are concerned either with the group-theoretic datum or with affine Deligne-Lusztig varieties, i.e.~the special fiber of associated moduli spaces of local shtukas. In \cite{Chen2021}, Chen, Fargues and Shen prove that full Hodge-Newton decomposability is equivalent to the condition that the admissible and the weakly admissible locus in the adic flag variety associated with $(G,\mu)$ coincide. In \cite{Chen2022}, Chen and Tong discuss a more general condition that they call \emph{weakly fully Hodge-Newton decomposable}, and show that it is equivalent to request that the weak Harder-Narasimhan decomposition is a coarsening of the Newton stratification of the Schubert variety for $\mu$ in the $B_{\rm{dR}}^+$-affine Grassmannian corresponding to $G$. Our condition that $\depth(G,\mu)<2$ implies that $(G,\mu)$ is also weakly fully Hodge-Newton decomposable, and one can easily find examples of depth 2 or higher that are no longer weakly fully Hodge-Newton decomposable. On the other hand, all $(G,\mu)$ such that the basic class in $B(G,\mu)$ is superbasic are weakly fully Hodge-Newton decomposable, which provides examples of arbitrarily large depth that are weakly fully Hodge-Newton decomposable.\\

\emph{Acknowledgement.} {Our collaboration was initiated due to discussions between the second author and Xuhua He at a workshop in Oberwolfach. We thank Xuhua He with whom we had many helpful discussions and who contributed several ideas to this project. We are grateful to the Mathematical Research Institute in Oberwolfach for its hospitality. We thank Miaofen Chen, Michael Rapoport, Ryosuke Shimada and Qingchao Yu for helpful discussions and Johannes Funk for providing some initial examples. The first author was partially supported by the New Cornerstone Foundation through the
New Cornerstone Investigator grant, and by Hong Kong RGC grant 14300122, both
awarded to Prof.\ Xuhua He. The second author was partially supported by the ERC via Consolidator grant 770936 NewtonStrat and by the DFG via Germany’s Excellence Cluster EXC 2044-390685587 Mathematics
M\"unster: Dynamics-Geometry-Structure, through CRC 1442 Geometry: Deformations and
Rigidity and by a Leibniz prize.} 

\section{Notation and preliminaries}\label{sec:notation}

Let $F$ be a non-archimedian local field. In the arithmetic case, $F$ is a finite extension of $\mathbb Q_p$, whereas $F=\mathbb F_{q}(\!(\varepsilon)\!)$ for some power $q$ of $p$ in the function field context. The completion of the maximal unramified extension of $F$ is denoted by $\breve F$. We write $\Gamma_0\subset \Gamma$ for the absolute Galois groups of $\breve F$ resp.\ $F$. The Galois group $\Gal(\breve F/F) = \Gamma/\Gamma_0$ is an infinite cylic group generated by the Frobenius $\sigma$. We pick a uniformizer inside the ring of integers $\mathcal O_{F}$ that in both cases is denoted  $\varepsilon$ and which is then also a uniformizer of $\mathcal{O}_{\breve F}$.

Let $G$ be a connected reductive group over $F$. Throughout the paper, we assume that $G$ is quasi-split. Possible generalizations to non-quasi-split groups are discussed in Section \ref{sec:depthnonqs} and at the beginning of Section \ref{sec:4}. We choose a maximal $\breve F$-split torus $T'$ of $G$ defined over $F$, and denote its centralizer by $T$. Then $T$ is a maximal torus defined over $F$. Let $\mathcal A=\mathcal A(G_{\breve F}, T')$ be the apartment in the Bruhat-Tits building of $G_{\breve F}$ defined by $T'$. In this apartment, we choose a $\sigma$-stable alcove $\mathfrak a$, whose stabilizer is a $\sigma$-stable Iwahori subgroup which we denote by $I\subseteq G(\breve F)$.

We choose a $\sigma$-stable special vertex $\mathfrak x\in\mathcal A$ adjacent to $\mathfrak a$. Relative to this basepoint, there is a unique Weyl chamber $\mathcal C^-$ containing $\mathfrak a$. We denote its opposite chamber by $\mathcal C$. It defines a Borel subgroup $B\subset G$ which is then also $\sigma$-stable.

Denote the based relative root datum of $(G_{\breve F},T)$ by $(\Phi, \Phi^{\pm}, \Delta)$ and the Weyl group by $W_0 = N_{G_{\breve F}}(T)(\breve F)/T(\breve F)$. Write $\Phi_\af = \Phi\times\mathbb Z$ for the set of affine roots, $\Phi_\af^{\pm}$ for the positive/negative affine roots (as defined by $\mathfrak a$) and $\Delta_{\af}$ for the set of simple affine roots. Thus $\Delta_{\af}$ consists of all affine roots of the form $(-\alpha,0)$ where $\alpha\in\Delta$ and all $(\theta,1)$ whenever $\theta$ is the highest root of an irreducible component of $\Phi$.

The Iwahori-Weyl group $\widetilde W = N_{G_{\breve F}}(T)(\breve F)/(I\cap T(\breve F))$ has a semi-direct product decomposition $\widetilde W \cong W_0 \ltimes X_\ast(T)_{\Gamma_0}$. Accordingly, we write elements $x\in\widetilde W$ by $x = w\varepsilon^{\mu}$ with $w\in W_0$ and $\mu\in X_\ast(T)_{\Gamma_0}$. We will also consider the action of $W_0$ on $X_*(T)$. To distinguish the two we will for this denote cocharacters by $\mu$ and write $w\mu$ for some $w\in W_0$ for the $w$-conjugate of $\mu$. 

We have the Iwahori-Bruhat decomposition
\begin{align*}
G(\breve F) = \bigsqcup_{x\in\widetilde W} I x I.
\end{align*}
There is a partial order $\leq$ defined on $\widetilde W$ describing the closure relations of Iwahori double cosets with respect to the topology on $G(\breve F)$ coming from the valuation on $\breve F$. This partial order is known as \emph{Bruhat order}. For a combinatorial description compare \cite{Schremmer2024_bruhat}. For a cocharacter $\mu\in X_\ast(T)_{\Gamma_0}$, we denote the \emph{admissible set} by 
\begin{align*}
    \Adm(\mu) := \{x\in\widetilde W\mid \exists u\in W:~x\leq\varepsilon^{u\mu}\},
\end{align*}
a finite subset of $\widetilde W$. An explicit combinatorial criterion to test if $x\in\Adm(\mu)$ can be found in \cite[Proposition~4.12]{Schremmer2024_bruhat}.

The group $\widetilde W$ acts transitively on the set of alcoves in $\mathcal A$, and the length of a minimal gallery between $\mathfrak a$ and $x\mathfrak a$ in $\mathcal A$ is called \emph{length} of the element $x\in\widetilde W$ and denoted $\ell(x)$. The group of length zero elements $\Omega\subseteq \widetilde W$ is isomorphic to the $\Gamma_0$-coinvariants of the Borovoi fundamental group $\pi_1(G)_{\Gamma_0} = X_\ast(T)_{\Gamma_0} / \mathbb Z\Phi^\vee$.

For an element $b\in G(\breve F)$, we define its $\sigma$-conjugacy class
\begin{align*}
[b] := \{g^{-1}b\sigma(g)\mid g\in G(\breve F)\}
\end{align*}
and write $B(G)$ for the set of $\sigma$-conjugacy classes. Kottwitz's classification of $B(G)$ in \cite{Kottwitz1985, Kottwitz1997} shows that a $\sigma$-conjugacy class of $b\in G(\breve F)$ is uniquely determined by two invariants, its (dominant) \emph{Newton point} $\nu(b)\in X_\ast(T)_{\Gamma_0}^{\dom}\otimes\mathbb Q$ and its \emph{Kottwitz point} $\kappa(b)\in \pi_1(G)_{\Gamma}$. The set $B(G)$ carries the partial order given by $[b]\leq[b']$ if $\kappa(b)=\kappa(b')$ and $\nu(b)\leq \nu(b')$ with respect to the order induced by the Bruhat order, i.e. $\nu(b')-\nu(b)$ is a non-negative rational linear combination of positive coroots. For a cocharacter $\mu\in X_\ast(T)_{\Gamma_0}$, we denote the set of \emph{neutrally acceptable elements} by $B(G,\mu) := \{[b]\in B(G)\mid [b]\leq [\mu(\varepsilon)]\}$.

Associated with any $\sigma$-stable parahoric $K\supseteq I$, we have the corresponding partial affine flag variety. If $F$ is of equal characteristic, this is an ind-scheme over the residue field of $\mathcal O_F$ with geometric points given by $G(\breve F)/K$. If $F$ is of mixed characteristic, we follow the conventions of Zhu \cite{Zhu17} and Bhatt-Scholze \cite{Bhatt17}. They define the partial affine flag variety as an ind-perfect ind-scheme over the residue field of $\mathcal O_F$. Its geometric points are again given by $G(\breve F)/K$.

In this paper, we study various kinds of affine Deligne-Lusztig varities (ADLV) contained in partial flag varieties. First, we associate to each $x\in \widetilde W$ and $b\in G(\breve F)$ the \emph{single ADLV} $X_x(b)$, which is the reduced sub-(perfect-)scheme of the affine flag variety with geometric points
\begin{align*}
	X_x(b) := \{g\in G(\breve F)/I\mid g^{-1}b\sigma(g)\in IxI\}\subseteq G(\breve F)/I.
\end{align*}
Given any $\sigma$-stable parahoric $K\supseteq I$, a dominant cocharacter $\mu\in X_\ast(T)_{\Gamma_0}$ and $b\in G(\breve F)$, we moreover define the \emph{union ADLV} as the reduced sub-(perfect-)scheme of the partial affine flag variety with geometric points
\begin{align*}
	X(\mu,b,K) := \{g\in G(\breve F)/K\mid g^{-1} b\sigma(g)\in K\Adm(\mu)K\}\subseteq G(\breve F)/K.
\end{align*}
When $K$ is a fixed hyperspecial subgroup, we follow the common notation and write $X_\mu(b)$ for $X(\mu,b,K)$. We remark that $[b]\in B(G,\mu)$ if and only if $X(\mu,b,K)\neq\emptyset$, cf.\ \cite{He2016}.

Given an element $[b]\in B(G,\mu)$, we can express the difference $\mu-\nu(b)\in X_\ast(T)_{\Gamma_0}\otimes\mathbb Q$ as a $\mathbb Q_{\geq 0}$-linear combination of simple coroots. If every simple coroot of $G$ occurs with a non-zero coefficient, we say that $(G,\mu,[b])$ is \emph{Hodge-Newton indecomposable}, otherwise we call the triple \emph{Hodge-Newton decomposable}, cf.\ \cite{Goertz2019}. If $(G,\mu,[b])$ is Hodge-Newton decomposable, then by \cite[Theorem~4.17]{Goertz2019}, there is a proper and $\sigma$-stable standard Levi subgroup $M\subseteq G$ such that $X(\mu,b,K)$ is isomorphic to a union of ADLV $X(\mu_i, b_i, K_i)$ for $M$, for $i$ in a finite index, cocharacters $\mu_i\in X_\ast(T)$, elements $b_i\in M(\breve F)$ and $\sigma$-stable parahoric subgroups $K_i\supseteq M(\breve F)\cap I$ for all $i$. We remark that each $\mu_i$ lies in the $W$-orbit of $\mu$ in $X_\ast(T)$.

\section{The notion of depth}

In this section we introduce the notion of depth, provide foundational properties and classify all pairs $(G,\mu)$ of depth less than 2.
\subsection{Definition and basic properties}
\begin{definition}\label{defdepth}
Let $\mu\in X_\ast(T)_{\mathbb Q}$ be a rational cocharacter.
\begin{enumerate}[(a)]
\item
For a $\sigma$-orbit of simple roots $\mathcal O\subseteq \Delta$, we define the weight $\omega_{\mathcal O}\in\mathbb Q\Phi$ via the identity
\begin{align*}
\langle \alpha^\vee,\omega_{\mathcal O}\rangle = \begin{cases}1,&\alpha\in\mathcal O,\\
0,&\alpha\in\Delta\setminus\mathcal O.\end{cases}
\end{align*}

\item
We set
\begin{align*}
\depth(G,\mu) := \max_{\mathcal O}\langle \mu_{\dom},\omega_{\mathcal O}\rangle\in\mathbb Q,
\end{align*}
where the maximum is taken over all $\sigma$-orbits $\mathcal O\subseteq \Delta$ and where $\mu_{\dom}$ is the dominant representative in the Weyl group orbit of $\mu$. 

This also defines a notion of depth for any conjugacy class of rational cocharacters $\{\mu\}$ of $G$ by considering any (or the dominant) representative in $X_*(T)$.
\item Similarly, we define for any  (not necessarily dominant) $\mu\in X_*(T)_{\mathbb Q}$
$$\depth'(G,\mu):=\max_{\mathcal O}\langle \mu,\omega_{\mathcal O}\rangle\in\mathbb Q.$$
\end{enumerate}
\end{definition}

\begin{remark}\label{remdepthred}
Let $\mu$ be a rational cocharacter of $G$. From the definition we obtain immediately the following reductions.
\begin{enumerate}[(a)]
\item The depth of a dominant cocharacter $\mu$ only depends on its image in $X_*(T)_{\Gamma}$. In particular, we have $$\depth(G,\mu)=\depth(G,\bar\mu)$$ where $\bar\mu$ is the $\sigma$-average of $\mu$.
\item Let $\mu_{\ad}$ be the character of $G_{\ad}$ induced by $\mu$. Then $\depth(G,\mu)=\depth(G_{\ad},\mu_{\ad})$.
      \item Assume that $G=G_1\times\dotsm \times G_n$ is a $\sigma$-stable decomposition, and let $\mu=(\mu_1,\dotsc,\mu_n)$ be the corresponding decomposition of $\mu$. Since each orbit $\mathcal O$ consists of roots of one of the factors $G_i$, we obtain $$\depth(G,\mu)=\max_i \depth(G_i,\mu_i).$$ 
\item Famously, the cocharacters $\mu$ with $\depth(G,\mu)\leq 1$ are known as the \emph{fully Hodge-Newton decomposable} ones. This is equivalent to the requirement that $(G,\mu,[b])$ is Hodge-Newton decomposable for each non-basic $[b]\in B(G,\mu)$. This is moreover equivalent to far-reaching consequences for the geometry of ADLV, as developed in \cite{Goertz2019}.
\end{enumerate}
\end{remark}

\begin{lemma}\label{lempropdepth}
 Let $T\subset B$ be a maximal torus and a Borel subgroup of $G$ and let $\mu\in X_*(T)_{\mathbb Q}$.
\begin{enumerate}[(a)]
    \item The value $\depth(G,\mu)$ only depends on $G$ and on the conjugacy class of the cocharacter $\mu$, but not on the choice of a ($\sigma$-stable) Borel subgroup.
    \item $\depth(G,\mu)\geq 0$.
    \item Write $\mu\in X_*(T)_{\Gamma}\subseteq X_*(T)_{\Gamma,\mathbb Q}$ as $\mu=\delta+\sum_i c_i\alpha_i^{\vee}$ with $\delta\in X_*(T)_{\Gamma,\mathbb Q}$ central, the sum running over all $\Gamma$-orbits of simple roots and $c_i\in\mathbb Q$. Let $\mathcal O_i$ be the orbit of $\alpha_i$. Then $\langle\mu,\omega_{\mathcal O_i}\rangle=c_i$ and hence $\depth'(G,\mu)=\max_i c_i$. If $\mu$ is dominant, the same holds for $\depth(G,\mu)$. 
    \item If $\mu,\tilde\mu\in X_*(T)$ are such that $\tilde\mu-\mu$ is a non-negative rational linear combination of positive coroots, then $\depth'(G,\mu)\leq \depth'(G,\tilde\mu).$
    \item $\depth'(G,\mu)\leq \depth'(G,\mu_{\dom})=\depth(G,\mu)$.
    \item Let $\mu\leq \tilde\mu$ both be dominant. Then $\depth(G,\mu)\leq\depth(G,\tilde\mu)$.
\end{enumerate}  
\end{lemma}
\begin{proof}
    (a)--(c) follow directly from the definition, (c) implies (d), and (d) implies (e) and (f).
\end{proof}

\subsection{Compatibility with Weil restrictions of scalars}\label{sec:weilRestrict}

For an integer $d\geq 1$, let $F^{(d)}\subseteq \breve F$ denote the uniquely determined unramified extension of $F$ of degree $d$. The Frobenius of $F^{(d)}$ is given by $\sigma^d \in \Gal(\breve F/F^{(d)})$. For a quasi-split group $G'$ over $F^{(d)}$, the Weil restriction of scalars $G = \Res_{F^{(d)}/F} G'$ is a 
quasi-split group over $F$. Choosing a torus $T'$ and a Borel subgroup $B'$ of $G'$ as above, we may choose $T = \Res_{F^{(d)}/F} T'$ and $B = \Res_{F^{(d)}/F} B'$.

We can identify $G_{F^{(d)}}$ with the $d$-fold direct product of $G'$ in such a way that the Frobenius $\sigma$ acts on $G_{F^{(d)}}\cong (G')^d$ via $\sigma((g_1,\dotsc,g_d)) = (\sigma^d(g_d), g_1,\dotsc,g_{d-1})$. We do the same for the tori and Borel subgroups. We similarly identify the root system $\Phi_G$ of $G$ with the $d$-fold direct sum of the root system $\Phi_{G'}$ of $G'$ using the same convention, and similarly for the Iwahori-Weyl groups $\widetilde W_G$ and $\widetilde W_{G'}$.

We have a natural map $B(G)\rightarrow B(G')$ of sets sending any $\sigma$-conjugacy class $[b] = [(b_1,\dotsc,b_d)]\in B(G)$ to the $\sigma^d$-conjugacy class $[b_d\cdots b_1]\in B(G')$. This is a bijection of posets, cf.\ \cite[Section~4]{He2018_acceptable}.

\begin{lemma}\label{lem:depthRestriction}
Let $\mu_1,\dotsc,\mu_d\in X_\ast(T')_{\Gamma_0}$ be dominant cocharacters, set $\mu = (\mu_1,\dotsc,\mu_d)\in X_\ast(T)_{\Gamma_0}$ and $\mu' = \mu_1+\cdots+\mu_d$.
\begin{enumerate}[(a)]
\item We have $\depth(G,\mu) = \depth(G',\mu')$.
\item An element $x = (x_1,\dotsc,x_d)\in \widetilde W_G$ is $\sigma$-fundamental if and only if $x' := x_d\cdots x_1\in \widetilde W_{G'}$ is $\sigma^d$-fundamental and $\ell(x') = \ell(x_1)+\cdots+\ell(x_d)$.
\item The natural map $B(G)\rightarrow B(G')$ restricts to a bijection $B(G,\mu)\xrightarrow\sim B(G',\mu')$, cf.\ \cite[Section~4]{He2018_acceptable}.
\end{enumerate}
\end{lemma}
\begin{proof}
\begin{enumerate}[(a)]
\item The simple roots $\Delta_G$ in $\Phi_G$ are of the form 
\begin{align*}
\alpha^{(i)} := (\underbrace{0,\dotsc,0}_{i-1\text{ times}},\alpha,\underbrace{0,\dotsc,0}_{d-i\text{ times}})
\end{align*}
for simple roots $\alpha\in \Delta_{G'}\subseteq \Phi_{G'}$ and $i\in\{1,\dotsc,d\}$. Given a $\sigma^d$-orbit $\mathcal O'\subseteq \Delta_{G'}$, the set 
\begin{align*}
\mathcal O := \{\alpha^{(i)}\mid \alpha\in\mathcal O',~i\in\{1,\dotsc,d\}\}
\end{align*}
is a $\sigma$-orbit in $\Delta_G$, and each $\sigma$-orbit arises in this way. Now observe that $\langle \mu,\omega_{\mathcal O}\rangle = \langle \mu',\omega_{\mathcal O'}\rangle$.
\item For $m\geq 1$, we write the $m$-fold $\sigma$-twisted power
\begin{align*}
x^{\sigma,m} := x \sigma(x)\cdots \sigma^{m-1}(x)\in\widetilde W_G
\end{align*}
as $x^{\sigma,m} = (x^{(m)}_1,\dotsc,x^{(m)}_d)$. Then,
\begin{align*}
x^{(m)}_i = x_i x_{i-1}\cdots x_{i-m+1}
\end{align*}
holds if $m\leq i$. Otherwise, we can write $m = i+ad + r$ for uniquely determined integers $a,r\geq 0$ with $r<d$, and obtain
\begin{align*}
x^{(m)}_i = x_i x_{i-1}\cdots x_1 \sigma^{d}(x') \sigma^{2d}(x')\cdots \sigma^{ad}(x') \sigma^{(a+1)d}(x_d x_{d-1}\cdots x_{d-r+1}).
\end{align*}
We see that $\ell(x^{\sigma,m}) = m\ell(x)$ holds if and only if all these products expressing $x^{(m)}_i$ are length additive for $m\geq 1$ and $i=1,\dotsc,d$. This is equivalent to the stated conditions.
\item This follows from \cite[Section~4]{He2018_acceptable}.
\qedhere\end{enumerate}
\end{proof}
\subsection{The depth with respect to Levi subgroups}
 \begin{theorem}
     \label{lemdepthprop}
Assume that $G$ is quasi-split, let $T$ be a maximal torus of $G$ and let $\mu\in X_*(T)_{\mathbb Q}$. Let $M\supseteq T$ be a $\sigma$-stable semi-standard Levi subgroup. 
\begin{enumerate}[(a)]
    \item\label{lemdp1} We have $\depth(M,\mu)\leq \depth(G,\mu)$.
    \item\label{lemdp2} If equality holds in (\ref{lemdp1}) and the group $G$ is relatively quasi-simple, then
    \begin{enumerate}[1.]
        \item $\pi_M(\mu)\in X_*(T)_{\mathbb Q,\Gamma}$, the $\sigma$-average of $\avg_M(\mu)$ as in \cite[Definition~3.2]{Chai2000}, is central in $G$. That is, if $M$ is a standard Levi for a $\sigma$-stable Borel subgroup $B$, then $\langle \mu, \omega_{\mathcal O_{\alpha}}\rangle=0$ for every simple root of $(G,B)$ that is not in $M$.
        \item If we choose a $\sigma$-stable Borel subgroup for $M$, for the $\sigma$-orbits $\mathcal O$ of simple roots in all but at most one of the $\sigma$-orbits of connected components of the Dynkin diagram of $M$, we have that $\langle\mu,\omega_{\mathcal O}\rangle=0$. 
    \end{enumerate}
\end{enumerate}
\end{theorem}
Here, $\avg_M(\mu)$ denotes the average over all elements of the $W_M$-orbit of $\mu$. It is the unique rational cocharacter of $T$ that is central in $M$ such that $\mu - \avg_M(\mu)$ is a rational linear combination of coroots of $M$.

The requirement that $G$ is relatively quasi-simple means that the Frobenius $\sigma$ acts transitively on the set of connected components of the Dynkin diagram (for any $\sigma$-stable Borel).

\begin{proof}We may assume $\mu$ is $\sigma$-stable and that $G$ is  adjoint and relatively quasi-simple by Remark~\ref{remdepthred}. Using induction, we may further assume that $M$ is the Levi factor containing $T$ of a maximal parabolic subgroup.

We first prove (\ref{lemdp1}). By Lemma \ref{lempropdepth}, we may compute the depth with respect to any $\sigma$-stable Borel subgroup $B$. We claim that we can choose $B$ such that the following three properties hold:
\begin{enumerate}[(i)]
    \item The Levi subgroup $M$ is standard with respect to $B$,
    \item the cocharacter $\mu$ is dominant with respect to $B\cap M$ and
    \item the rational cocharacter $\avg_M(\mu)$ is dominant with respect to $B$.
\end{enumerate}
Once properties (i) and (iii) are established, we can also get (ii) by considering a suitable $W_M$-conjugate of $B$.

Let $B'$ be a $\sigma$-stable Borel satisfying (i). If $\avg_M(\mu)$ is central in $G$, $B'$ also satisfies (iii) and we are done. Otherwise, let $P$ be the parabolic subgroup of $G$ corresponding to $\avg_M(\mu)$ and let $B\subseteq P$ be the Borel subgroup of $G$ with $B'\cap M=B\cap M$. Then $\avg_M(\mu)$ is $B$-dominant.

Assume from now on that properties (i)--(iii) are satisfied. As in Lemma \ref{lempropdepth}, (c), we express $\mu\in X_*(T)_{\mathbb Q,\Gamma}$ as $\mu=\avg_M(\mu)+\sum_{\mathcal O_j} c_{j,M}\alpha_j^{\vee}$ where the sum runs over all $\Gamma$-orbits of simple roots in $M$ and where $\alpha_j\in \mathcal O_j$ is a chosen representative. Since $\mu$ is $B\cap M$-dominant, $c_{\alpha,M}\in\mathbb Q_{\geq 0}$. Then $\depth(M,\mu)=\max_j c_{j,M}$. A similar consideration applies to $G$, which is adjoint, leading to $\mu=\sum c'_{j,G}\alpha_j^{\vee}$. Since $\avg_M(\mu)$ is $B$-dominant, it is a non-negative rational linear combination of the simple coroots of $G$. In particular, $c_{j,M}\leq c'_{j,G}$ for every orbit $\mathcal O_j$ of simple roots $T$ in $M$. By Lemma \ref{lempropdepth}, $\depth(M,\mu)\leq \depth'(G,\mu)\leq \depth(G,\mu)$.

We now prove (\ref{lemdp2}), so assume that $\depth(G,\mu) = \depth(M,\mu)$ and that $G$ is relatively quasi-simple. If $\mu$ itself is central in $G$, then clearly 1., 2.~and equality in (\ref{lemdp1}) are satisfied. We assume that this is not the case. We denote the $\sigma$-orbit of simple roots of $G$ that are not in $M$ by $\mathcal O'$.

    Assume that 1. is not satisfied. Then $\pi_M( \mu)$ is a non-zero dominant rational cocharacter for $G$, so a positive linear combination of all simple coroots. In the above argument, this shows that $c_{j,M} < c'_{j,G}$ for all $j$. Hence, $\depth(M,\mu)< \depth'(G,\mu)\leq\depth(G,\mu)$. This proves necessity of 1. In particular, we then have $\langle\mu,\omega_{\mathcal O'}\rangle=0$. Moreover, we get $\langle \mu,\omega_{\mathcal O,G}\rangle = \langle \mu,\omega_{\mathcal O,M}\rangle$ for all $\sigma$-orbits $\mathcal O\subseteq \Delta_M$. We denote the common value by $\langle \mu,\omega_{\mathcal O}\rangle$, and hence $\depth'(G,\mu)=\depth(M,\mu)$. 
    
    We assume that 2. is not satisfied. By definition, it is enough to find an element $\mu'$ in the Weyl group orbit of $\mu$ such that $\depth'(G,\mu')>\depth(M,\mu)=\depth'(G,\mu)$. Let $D_0,\dotsc,D_l$ be the $\sigma$-connected components of the Dynkin diagram of $M$. In each $D_i$ we choose a $\sigma$-orbit of simple roots $\mathcal O_i$ such that $\langle\mu,\omega_{\mathcal O_i}\rangle\geq 0$ is maximal among the $\sigma$-orbits in that component. After renumbering the $D_i$ we may assume that $\langle\mu,\omega_{\mathcal O_0}\rangle=\depth(M,\mu)> 0$. Remove the union of all $\mathcal O_i$ from the Dynkin diagram of $G$ and let $D'$ be the $\sigma$-connected component of $\mathcal O'$ of this complement. Let $D''$ be the complement of $\mathcal O'$ in $D'$. Let $w'_0$ and $w''_0$ be the longest elements in the Weyl groups for $D'$, resp. $D''$. We claim that $\mu'=w'_0w_0''(\mu)$ satisfies $\depth'(G,\mu')>\depth'(G,\mu)$. This then shows that $\depth(G,\mu)>\depth(M,\mu)$.
    
    To prove this claim we write again $\mu=\avg_G(\mu) + \sum c_\alpha\alpha^{\vee}\in X_*(T)_{\mathbb Q,\Gamma}$ with  $\alpha$ running through a set of representatives of the $\Gamma$-orbits of simple roots of $G$. Then $\mu'-\avg_G(\mu)=\sum c_\alpha w_0'w''_0(\alpha^{\vee})$. We view the sum as a linear combination of the simple coroots $\alpha^{\vee}$ as above and  want to compute the coefficient corresponding to the orbit $-w_0'\mathcal O'\subseteq D'$. For a simple root $\alpha$ that is not in $D'\cup \mathcal O_0\cup\cdots \cup \mathcal O_l$, we have $w_0'w''_0(\alpha^{\vee})=\alpha^{\vee}$, thus the corresponding summand does not contribute to the coefficient of interest. Under $w_0'w''_0$, simple roots in $D''$ are mapped to simple roots in $D'$ that are not in $-w_0'\mathcal O'$. For $\alpha'\in\mathcal O'$ we have coefficient $c_{\alpha'}=0$ by 1. 
    
 It remains to compute the contribution of the summands $c_{\alpha_i} w_0'w''_0(\alpha_i^{\vee})$ for $i=0,\dotsc,l$ and $\alpha_i\in\mathcal O_i$. First notice that $w_0'w''_0(\alpha_i^{\vee})$ and $w''_0(\alpha_i^{\vee})$ are again positive coroots. Indeed, writing these coroots as linear combinations of the simple coroots, the coefficient of $\alpha_i^{\vee}$ is still 1 as $w''_0, w_0'w''_0\in W_{D'}$. We have $\langle w_0'' \alpha_i^\vee, \alpha\rangle\geq 0$ for all simple roots $\alpha$ in $D''$, and (by $\sigma$-connectedness of the Dynkin diagram of $G$ and $\sigma$-stability of $\mu$) we have a strict inequality for at least one such root. Thus, $w_0'' \alpha_i^\vee - \alpha_i^\vee$ is a linear combination of simple coroots of $D''$, and $w_0''\alpha_i^{\vee}$ defines a non-zero $D''$-dominant cocharacter. Thus, for every $\sigma$-orbit of simple roots containing a simple root in the connected component of $D''$ adjacent to $\alpha_i$, the corresponding coefficient in $w_0'' \alpha_i^\vee - \alpha_i^\vee\in X_*(T)_{\mathbb Q,\Gamma}$ is strictly positive. The coefficient corresponding to $\mathcal O'$ in $w_0'' \alpha_i^\vee$ is 0. Thus $\sum_{\tilde\alpha\in\mathcal O'}\langle w_0''\alpha_i^{\vee},\tilde\alpha\rangle<0$. Hence $$\sum_{\tilde\alpha\in-w_0'\mathcal O'}\langle w_0'w_0''\alpha_i^{\vee},\tilde\alpha\rangle>0,$$  which implies that the coefficient $d_i$ of $(\alpha')^{\vee}$ in the decomposition of $w_0'w_0''(\alpha_i^{\vee})$ is strictly positive. Note that the $d_i$ are also integral.
    
    We obtain that \begin{align*}
        \langle \sum_\alpha c_\alpha w'_0w''_0(\alpha^{\vee}),\omega_{-w_0'\mathcal O'}\rangle = \sum_{i=0}^l d_ic_{\alpha_i}.
    \end{align*}We assumed that $c_{0}=\depth(M,\mu)>0$, and have that $d_{0}\geq 1$. Thus as soon as $c_i>0$ for some $i\in \{1,\dotsc, l\}$ (i.e. 2. is violated), we see that this sum is strictly greater than $\depth(M,\mu)$, which proves the claim.
\end{proof}

\begin{example} One can improve on (\ref{lemdp2}) of Theorem \ref{lemdepthprop} and give necessary and sufficient conditions for equality by an explicit consideration of all possible Dynkin diagrams and automorphisms, and using an argument as in the proof of (\ref{lemdp2}). We give some examples in the split cases.

     For groups of type $A_n$, the necessary conditions of (\ref{lemdp2}) are also sufficient. Indeed, using a similar argument as in the above proof, it is enough to consider $G=\PGL_n$, with $T$ the diagonal torus and $B$ the upper triangular Borel. Using induction we can reduce to the case that $M$ is the standard Levi subgroup of a maximal standard parabolic subgroup. Then $M$ corresponds to a simple root $\alpha'=\alpha_{j_0}$ and the two conditions imply that the $M\cap B$-dominant element $\mu\in X_*(T)=\mathbb Z^n/\mathbb{Z}\cdot (1,\dotsc,1)$ has a representative of the form $(\mu_1,\dotsc,\mu_{j_0},0,\dotsc,0)$ or $(0,\dotsc, 0,\mu_{j_0+1},\dotsc,\mu_n)$ (by (b)) for $\mu_j\in\mathbb Z$ which are in decreasing order and satisfy $\sum\mu_j=0$ (by (a)). We consider the first case, the second is analogous. Let $i_0\leq j_0$ be such that $\mu_1,\dotsc, \mu_{i_0}\geq 0$ and $\mu_{i_0+1},\dotsc, \mu_{j_0}<0$. Then $\depth(M,\mu)=\sum_{i=1}^{i_0}\mu_i$ and $\mu_{G-\dom}=(\mu_1,\dotsc, \mu_{i_0},0,\dotsc,0,\mu_{i_0+1},\dotsc,\mu_{j_0})$. Hence $\depth(G,\mu)=\depth(M,\mu)$.

     For groups of type $B_n$ with $n>2$, $\alpha'$ has to be a long root and the depth has to be realized by a long root in the component of the Dynkin diagram of $M$ containing the short root. 
     
     For groups of type $C_n$ with $n\geq 2$, $\alpha'$ has to be a short root and the depth has to be realized by the long simple root.

     For type $D_n$ with $n\geq 4$ label the simple roots such that the branch point of the Dynkin diagram corresponds to $\alpha_{n-2}$ and the two horns to $\alpha_{n-1}$ and $\alpha_n$. Then there are two cases with equality. The first is that $\alpha'=\alpha_{n-2}$ and the depth is realized by one of the two horns. The second is that $\alpha'=\alpha_{j_0}$ for some $j_0<n-2$, the depth is realized by $\alpha_{n-2}$ and we have $c_{n-2}=2c_{n-1}=2c_n$.

     For $G_2$ and any proper Levi subgroup, no non-central $\mu$ satisfies equality in (\ref{lemdp1}).
\end{example}

\subsection{Classification}\label{secclassification}

In this section we classify all groups $G$ and all dominant cocharacters $\mu$ such that $\depth(G,\mu)<2$. Recall that a pair $(G,\mu)$ is \emph{fully Hodge-Newton decomposable} iff $\depth(G,\mu)\leq 1$, cf.\ \cite[Def.~3.2 and Theorem 3.3]{Goertz2019}. Given the classification of fully Hodge-Newton decomposable pairs in  \cite[Theorem~3.5]{Goertz2019}, it thus remains to classify the pairs with $1<\depth(G,\mu)<2$. Just as in \cite[Section~3.3]{Goertz2019} and using Remark \ref{remdepthred}, this task is easily reduced to adjoint groups $G$ which are quasi-simple over $F$, so the Frobenius $\sigma$ acts transitively on the set of irreducible components of the root system. Such a group $G$ will be isogenous to the restriction of scalars of a group $G'$ over an unramified extension $F^{(d)}$ of $F$ of some degree $d$ such that $G'$ is absolutely quasi-simple, i.e.\ such that its root system is irreducible. By Lemma~\ref{lem:depthRestriction}, it thus suffices to classify pairs with $1<\depth(G,\mu)<2$ for groups $G$ which are adjoint and quasi-simple.
\begin{theorem}\label{prop:classification}
Assume that $G$ is adjoint and absolutely quasi-simple over $F$, and let $\mu \in X_\ast(T)_{\Gamma_0}$ be a dominant cocharacter. Then $1<\depth(G,\mu)<2$ if and only if the following two conditions are both satisfied.
\begin{enumerate}[(i)]
\item The group $G$ is split over $F$.
\item Under Bourbaki's conventions for root systems and the enumeration of fundamental coweights, the Cartan type of $G$ and the coweight $\mu$ or some conjugate of it under Dynkin diagram automorphisms occur in the following table.

\begin{center}\begin{tabular}{ll}
Cartan Type&Coweight $\mu\in  X_\ast(T)_{\Gamma_0}$\\\hline
$A_n$, $n\geq 1$ & $2\omega_1^\vee+\omega_n^\vee$\\
$A_n$, $n\geq 2$ & $2\omega_1^\vee$\\
$A_n$, $n\geq 2$ & $\omega_2^\vee+\omega_n^\vee$\\
$A_n$, $n\geq 4$ & $\omega_2^\vee$\\
$A_4$ & $\omega_1^\vee+\omega_2^\vee$ \\
$A_5,~A_6,~A_7$ & $\omega_3^\vee$\\
$C_3$ & $\omega_3^\vee$\\
$D_5$ & $\omega_5^\vee$
\end{tabular}
\end{center}
\end{enumerate}
\end{theorem}
\begin{proof}
Given any fixed adjoint and absolutely quasi-simple group $G$ over $F$, that is, any choice of a finite Dynkin diagram together with a distinguished automorphism $\sigma$, it is straightforward to enumerate all coweights $\mu$ with $1<\depth(G,\mu)<2$. Using a computer one can thus check this proposition for all root systems of rank $\leq 20$. From now on we assume that $\rk \Phi\geq 21$. In particular, the root system is of classical type.

Let $M\subseteq G$ be a $\sigma$-stable standard Levi subgroup. By Theorem \ref{lemdepthprop},  $\depth(M,\mu)\leq \depth(G,\mu)$. In particular, every restriction of $\mu$ to a Levi subgroup for a sub-root system generated by a $\sigma$-stable set of at most 20 of the simple roots must occur in the above list or in the list of $\depth\leq 1$ cases in \cite[Theorem~3.5]{Goertz2019}.

Consider first the case where the Dynkin diagram of $G$ is of type $D_n$ for some $n\geq 21$, so the Frobenius action will always fix $\alpha_1,\dotsc,\alpha_{n-2}$. By studying the various $\sigma$-stable Levi subgroups of ranks $\leq 20$, we see that $\depth(G,\mu)<2$ is only possible for $\mu \in \{0,\omega_1^\vee,2\omega_1^\vee,\omega_2^\vee\}$. The first two cases are known to satisfy $\depth(G,\mu)\leq 1$. The highest root $\theta$ is known to satisfy $\theta^\vee = \omega_2^\vee$ (by studying the well-known affine Dynkin diagram), so $\depth(G,\omega_2^\vee) = \depth(G,\theta^\vee)=2$ by expressing $\theta^\vee$ as a sum of simple coroots. We note that 
\begin{align*}
\depth(G,2\omega_1^\vee) = \depth(G,\omega_2^\vee+\alpha_1^\vee) \geq \depth(G,\omega_2^\vee)=2,
\end{align*}
showing that indeed none of these cases have depth strictly between $1$ and $2$.

Consider next the case where the Dynkin diagram of $G$ is of type $A_{n-1}$ for some $n\geq 22$, so the Frobenius action can be trivial or given by $\alpha_i\mapsto\alpha_{n-i}$.
Studying $\sigma$-stable Levi subgroups of rank $\leq 20$, we see that $\mu$ must be of the form $\mu = \mu_1+\mu_2$ with
\begin{align*}
\mu_1 \in\{0,\omega_1^\vee,\omega_2^\vee,2\omega_1^\vee\},\quad 
\mu_2 \in\{0,\omega_{n-1}^\vee,\omega_{n-2}^\vee,2\omega_{n-1}^\vee\}
\end{align*}

Realizing our roots $\alpha_i =\alpha_i^\vee = e_i-e_{i+1}\in \mathbb Z^{n}$ as usual, we get the fundamental (co)weights and weights
\begin{align*}
\omega_i = \omega_i^\vee = (\underbrace{1,\dotsc,1}_{i\text{ times}},\underbrace{0,\dotsc,0}_{n+1-i\text{ times}}) - \frac i{n}(1,\dotsc,1) = (\underbrace{1-i/n,\dotsc,1-i/n}_{i\text{ times}}, \underbrace{-i/n,\dotsc,-i/n}_{n+1-i\text{ times}}).
\end{align*}
Thus for $i,j\in\{1,\dotsc,n-1\}$, we get
\begin{align*}
\langle \omega_i^\vee,\omega_j\rangle = \begin{cases}j(1-i/n),&j\leq i,\\
i(1-j/n),&i\leq j.\end{cases}
\end{align*}
For split $G$ we compute
$$\langle 2\omega_1^\vee+\omega_{n-1}^\vee,\omega_j\rangle = 2(1-j/n) + j(1-(n-1)/n) = 2-\frac jn,$$
thus $\depth(G,2\omega_1^\vee+\omega_{n-1}^\vee) = 2-\frac 1n<2$, and consequently also $2\omega_1^{\vee}$ has depth $<2$. A similar calculation shows that $\omega_2^{\vee}<\omega_2^{\vee}+\omega_{n-1}^{\vee}$ have depth $<2$. A comparison with \cite[Theorem~3.5]{Goertz2019} shows that they all have $\depth>1$. On the other hand,
$$\langle \omega_2^\vee+\omega_{n-2}^\vee,\omega_2\rangle = 2(1-2/n) + 2(1-(n-2)/n) = 2,$$
hence $\depth(G,\omega_2^\vee+\omega_{n-2}^\vee) \geq 2.$ Similarly, one sees that the coweights
\begin{align*}
\omega_2^\vee+\omega_{n-2}^\vee,~
2\omega_1^\vee+\omega_{n-2}^\vee,~
\omega_2^\vee+2\omega_{n-1}^\vee,~
2\omega_1^\vee+2\omega_{n-1}^\vee
\end{align*}
all have $\depth\geq 2$ if $G$ is split. If the Frobenius acts via $\alpha_i\mapsto \alpha_{n-i}$, an analogous explicit calculation shows that $\depth(G,\mu)\leq 1$ or $\depth(G,\mu)\geq 2$ for all dominant coweights $\mu$.

In the remaining two Cartan types $B_n$ and $C_n$ for $n\geq 21$, a simple folding argument reduces to types $D$ and $A$, respectively, and shows that no dominant coweight $\mu$ can satisfy $1<\depth(G,\mu)<2$.
\end{proof}
\begin{remark}
Considering the classification, one makes the following remarkable observation for all relatively quasi-split groups $G$ and cocharacters $\mu$ of $\depth(G,\mu)<2$: The set
\begin{align*}
    \{[b]\in B(G,\mu)\mid (G,\mu,[b])\text{ is HN-indecomposable}\}
\end{align*}
is totally ordered with respect to the natural order on $B(G)$.
\end{remark}

\subsection{Non-quasi-split groups}\label{sec:depthnonqs}

The depth can also be defined without assuming $G$ to be quasi-split, but it requires a careful choice of conventions.

Let $G$ be any connected and reductive group over $F$.
As in \cite{Goertz2019}, we consider the unique quasi-split inner form of $G$, which we denote $G_{\text{qs}}$.  The Bruhat-Tits buildings of $G$ and $G_{\text{qs}}$ are the same, but the Frobenius actions are not. Let $T, \mathcal A$ be as in Section~\ref{sec:notation}. We write $\sigma_0$ for the Frobenius action of $G_{\text{qs}}$ on the standard apartment, and choose $\mathfrak x\in \mathcal A$ to be a $\sigma_0$-stable special vertex. Identify $\mathcal A\cong X_\ast(T)_{\Gamma_0}\otimes\mathbb R$ by sending $\mathfrak x$ to $0$. Then one defines the root systems and associated notation as in Section~\ref{sec:notation}. The purpose of this section is to propose the following definition of depth in this context.

\begin{definition}\label{defdepthnqs}
Let $\mu\in X_\ast(T)_{\mathbb Q}$ be a rational cocharacter.
\begin{enumerate}[(a)]
\item
For a $\sigma_0$-orbit of simple roots $\mathcal O\subseteq \Delta$, we define the weight $\omega_{\mathcal O}\in\mathbb Q\Phi$ via the identity
\begin{align*}
\langle \alpha^\vee,\omega_{\mathcal O}\rangle = \begin{cases}1,&\alpha\in\mathcal O,\\
0,&\alpha\in\Delta\setminus\mathcal O.\end{cases}
\end{align*}
\item
We set
\begin{align*}
\depth(G,\mu) := \max_{\mathcal O}\Bigl(\langle \mu_{\dom},\omega_{\mathcal O}\rangle+\mathrm{frac}(\langle \sigma(0),\omega_{\mathcal O}\rangle)\Bigr)\in\mathbb Q,
\end{align*}
where the maximum is taken over all $\sigma_0$-orbits $\mathcal O\subseteq \Delta$ and where $\mu_{\dom}$ is the dominant representative in the Weyl group orbit of $\mu$. Here, $\mathrm{frac} : \mathbb R\rightarrow [0,1)$ denotes the fractional part of a real number. 
\end{enumerate}
\end{definition}
This definition is made analogously to \cite[Definition~3.2]{Goertz2019}, the minute condition for full Hodge-Newton decomposability is also in this case given by $\depth (G,\mu)\leq 1$. 

There are many examples of groups $G$ that are not quasi split together with cocharacters $\mu\in X_\ast(T)$ where $1<\depth(G,\mu)<2$. Indeed, we notice by definition that $\depth(G_{\text{qs}},\mu)\leq \depth(G,\mu)<\depth(G_{\text{qs}},\mu)+1$. Therefore, whenever $(G_{\text{qs}},\mu)$ occurs in the classification of fully Hodge-Newton indecomposable cases from \cite[Theorem~D]{Goertz2019}, but $(G,\mu)$ does not, this yields such an example.

It appears to be very challenging to give a satisfactory description of the geometry of the ADLV in all non-quasi-split cases of depth between $1$ and $2$. Our key properties such as (L1BC) and geometric Coxeter type remain true e.g.\ for all inner forms of $\GL_3$ and $\mu = (1,0,-1)$, but fail for other non-quasi-split cases.

\section{Ekedahl-Oort strata}

In this section, we study the geometry of affine Deligne-Lusztig varieties $X_\mu(b)$ in the affine Grassmannian of $G$. For this, assume that $G$ has a reductive model over $\mathcal{O}_{F}$, also denoted $G$. Let $K=G(\mathcal O_{\breve F})$, a hyperspecial and $\sigma$-stable subgroup of $G(\breve F)$. Notice that this assumption implies that $G$ is unramified.

\subsection{Geometric properties}

We summarize some of the main known results about the geometry of $X_{\mu}(b)$ for arbitrary $G,\mu,b$.
\begin{theorem}
Let $\mu$ be a dominant cocharacter and $[b]\in B(G)$.
\begin{enumerate}[(a)]
\item  The affine Deligne-Lusztig variety $X_\mu(b)$ is non-empty if and only if $[b]$ lies in the neutrally acceptable set $B(G,\mu)$.

Conjectured by Kottwitz-Rapoport, proved by \cite{Rapoport1996, Gashi2010, He2014}.
\item If $X_\mu(b)\neq\emptyset$, it is equidimensional of dimension
\begin{align*}
\dim X_\mu(b) = \frac 12\left(\langle \mu - \nu(b),2\rho\rangle - \defect(b)\right).
\end{align*}
Here, $\defect(b)$ denotes the defect of $b$, which is defined as $\mathrm{rk}_F(G) - \mathrm{rk}_F(J_b)$, cf.\ \cite{Chai2000, Kottwitz2006}. Conjectured by Rapoport, proved by \cite{Goertz2006, Viehmann2006, Hamacher2015, Takaya2022}.
\item  The number of $J_b(F)$-orbits of irreducible components of $X_\mu(b)$ is equal to the sum of the dimensions of the weight spaces $M_\mu(\lambda)$ of the unique irreducible representation $M_\mu$ of the Langlands dual group $\hat G$ with highest weight $\mu$. Here, $\lambda \in X_\ast(T)$ runs over all maximal elements in the set of cocharacters  $\lambda$ such that the $\sigma$-average of $\lambda$ is at most $\nu(b)$ and $\kappa(\lambda) = \kappa(b)$.

Conjectured by Chen-Zhu \cite{Hamacher2018}, proved by \cite{Zhou2020, Nie2022}.
\end{enumerate}
\end{theorem}
These results mark major milestones in the study of affine Deligne-Lusztig varieties. For the purpose of this work, we examine the geometry of $X_\mu(b)$ more closely in the cases where $\depth(G,\mu)<2$. Under this assumption, we obtain information regarding the geometry of a certain natural stratification of $X_\mu(b)$, which is analogous to the Ekedahl-Oort stratification of integral models of Shimura varieties. We assume that our chosen Iwahori subgroup $I$ is contained in $K$, which we may always achieve with a suitable choice of $I$. We write $\widetilde W^K$ for the set of all minimal length elements in $\widetilde W$ with respect to the right-multiplication by $\widetilde W_K := \{y\in \widetilde W\mid \dot y \in K\}$.

\begin{theorem}[{\cite[Theorem~1.1]{Viehmann2014}}]
Let $\mu\in X_\ast(T)$ be a dominant cocharacter.
For each $g\in K\mu(\varepsilon)K\subseteq G(\breve F)$, there exists a unique element $x\in\widetilde W^{K} \cap W_0 \varepsilon^\mu W_0$ such that the $(K,\sigma)$-conjugacy class of $g$ intersects with $IxI$.
\end{theorem}
Recall that affine Schubert cells are locally closed. Thus for each $x\in\widetilde W^K\cap W_0\varepsilon^\mu W_0$ and for each fixed $b$, we obtain a locally closed subscheme $\mathrm{EO}_x \subset X_\mu(b)$ whose  $\overline{\mathbb F}_p$-valued points are those $gK\in X_\mu(b)\subseteq G(\breve F)/K$ such that $g^{-1}b\sigma(g)$ lies in the $(K,\sigma)$-orbit of a point in $IxI$. It is called the EO-stratum associated with $x$. The affine Deligne-Lusztig variety $X_\mu(b)$ is the union of these mutually disjoint EO-strata.

\begin{theorem}
Let $\mu\in X_\ast(T)$ be a dominant cocharacter.
\begin{enumerate}[(a)] 
    \item For $x\in\widetilde W^K$, we have $x\in W_0\varepsilon^{\mu'}W_0$ for a dominant $\mu'\leq \mu$ if and only if $x\in \Adm(\mu)$ \cite[Theorem~6.10]{He2017}.
    \item
 For $x\in \Adm(\mu)\cap \widetilde W^K$, the map
    \begin{align*}
        X_x(b)\rightarrow \mathrm{EO}_x,\quad gI\mapsto gK
    \end{align*}
    is surjective with finite fibres \cite[Theorem~6.21]{He2017}. 
\end{enumerate}
\end{theorem}
Here $\mathrm{EO}_x$ denotes the Ekedahl-Oort stratum for $x$ in $X_{\mu'}(b)$ where $\mu'\leq\mu$ is chosen such that $x\in W_0\varepsilon^{\mu'}W_0$. Equivalently, it is equal to the analogously defined EO-stratum in $X_{\leq \mu}(b)$. The theorem implies that for all $x\in \Adm(\mu)^K := \Adm(\mu)\cap \widetilde W^K$, the geometry of the ADLV $X_x(b)$ in the affine flag variety is closely related to the geometry of the corresponding EO-stratum $\mathrm{EO}_x$ in $X_{\leq \mu}(b)$.

The main result of this section is the following.
\begin{theorem}\label{thm:EOGeometry}
    Let $\mu\in X_\ast(T)$ be a dominant cocharacter such that $\depth(G,\mu)<2$, and let $x\in\Adm(\mu)^K$, where $K$ is a $\sigma$-stable hyperspecial subgroup of $G$.
    \begin{enumerate}[(a)]
    \item The set
    $B(G)_x := \{[b]\in B(G)\mid X_x(b)\neq\emptyset\}$
    has the form $$B(G)_x = \{[b]\in B(G)\mid [b_{x,\min}]\leq [b]\leq [b_{x,\max}]\}$$ for uniquely determined and explicitly described elements $[b_{x,\min}], [b_{x,\max}]\in B(G)$.
    \item For any $[b]\in B(G)_x$, the ADLV $X_x(b)$ is equidimensional and the group $J_b(F)$ acts transitively on the set of irreducible components. Up to universal homeomorphism (in equal characteristic) or perfection (in mixed characteristic), $X_x(b)$ is isomorphic to the disjoint union of products of a classical Deligne-Lusztig varieties with copies of $\mathbb A^1$ and $\mathbb G_m$.
    \item If $G$ is absolutely quasi-simple and $1<\depth(G,\mu)<2$, then the classical Deligne-Lusztig variety in (b) can be chosen to be of Coxeter type. Moreover, for all $[b]\in B(G)_x$, we have
    \begin{align*}
    \dim X_x(b) = \frac 12\left(\ell(x) + \ell_{R,\sigma}(\cl(x)) - \langle \nu(b),2\rho\rangle - \defect(b)\right).
    \end{align*}
    Here, $\cl : \widetilde W\rightarrow W_0$ is the projection function sending $w\varepsilon^\lambda\in\widetilde W$ to $w\in W_0$, and $\ell_{R,\sigma} : W_0\rightarrow \mathbb Z_{\geq 0}$ denotes the $\sigma$-twisted notion of reflection length from \cite[Definition~4.4]{Schremmer2023_coxeter}.

    For comparison, in the cases of $\depth(G,\mu)\leq 1$, we get $\dim X_x(b)=\ell(x)$ if $[b]$ is basic and $\dim X_x(b)=0$ otherwise.
    \end{enumerate}
\end{theorem}
\begin{remark}
The geometric properties discussed in Theorem~\ref{thm:EOGeometry} (b) and (c) are invariant under finite morphisms, hence immediately transfer to the corresponding EO-strata of $X_{\leq \mu}(b)$.
\end{remark}
In order to prove Theorem~\ref{thm:EOGeometry}, we first need to recall a concept introduced by Shimada, Yu and the first author.
\begin{definition}[{\cite[Section~2.2]{Schremmer2022_newton}}]
Let $x = w\varepsilon^\mu\in \widetilde W$.\begin{enumerate}[(a)]
\item We define the \emph{length functional} of $x$ to be the function $\ell(x,\cdot) : \Phi\rightarrow\mathbb Z$, sending a root $\alpha\in\Phi$ to
\begin{align*}
\ell(x,\alpha) = \langle \mu,\alpha\rangle + \begin{cases}0,&\alpha\text{ and }w\alpha\text{ are both negative or both positive}\\
1,&\alpha\in\Phi^+\text{ and }w\alpha\in\Phi^-,\\
-1,&\alpha\in\Phi^-\text{ and }w\alpha\in \Phi^+.
\end{cases}
\end{align*}
\item An element $v\in W$ is called \emph{length positive} for $x$ if $\ell(x,v\alpha)\geq 0$ for all $\alpha\in \Phi^+$.
\end{enumerate}
\end{definition}
The notion of length positive elements is a natural construction, used in results regarding the Bruhat order on $\widetilde W$, the generic $\sigma$-conjugacy classes of Iwahori double cosets $IxI$ and Kazhdan-Lusztig cells for affine Coxeter groups.

Recall that an element $w\in W_0$ is called a \emph{partial $\sigma$-Coxeter} element if for some (equivalently any) reduced word $w = s_{\alpha_1}\cdots s_{\alpha_{\ell(w)}}$, the simple roots $\alpha_1,\dotsc,\alpha_{\ell(w)}\in \Delta$ lie in pairwise distinct $\sigma$-orbits.

\begin{definition}[{\cite[Definition~3.4]{Schremmer2023_coxeter}}]
Let $x = w\varepsilon^\mu\in\widetilde W$ and $v\in \LP(x)$. If $v^{-1}\sigma(wv)\in W_0$ is a partial $\sigma$-Coxeter element, then we call $x$ a \emph{positive Coxeter type element} and $(x,v)$ a \emph{positive Coxeter type pair}.
\end{definition}
This definition generalizes the notion of \emph{finite Coxeter type} as introduced by He-Nie-Yu \cite{He2022}. We have a good understanding of the geometry of ADLV $X_x(b)$ where $x$ has positive Coxeter type, which we summarize as below:
\begin{theorem}[{\cite[Theorem~5.7]{Schremmer2023_coxeter}}]\label{thm:ssyPositiveCoxeterType}
Let $(x=w\varepsilon^\mu,v)$ be a positive Coxeter pair with support $J=\supp_\sigma(v^{-1}\sigma(wv))$.
\begin{enumerate}[(a)]
\item The set $B(G)_x := \{[b]\in B(G)\mid X_x(b)\neq\emptyset\}$ can be written as
\begin{align*}
    B(G)_x = \{[b]\in B(G)\mid [b_{x,\min}]\leq [b]\leq [b_{x,\max}]\}.
\end{align*}
Here, $[b_{x,\min}]$ is the uniquely determined smallest $\sigma$-conjugacy class $[b]\in B(G)$ satisfying
\begin{enumerate}[(1)]
\item $\kappa(b)\equiv \mu$ in $\pi_1(G)_{\Gamma}$ and
\item $\nu(b)-\avg_\sigma(v^{-1}\mu)\in \mathbb Q \Phi_J^\vee$ in $X_\ast(T)\otimes\mathbb Q$, where $\avg_\sigma$ denotes the $\sigma$-average.
\end{enumerate}
Similarly, $[b_{x,\max}]$ is the uniquely determined largest $\sigma$-conjugacy class $[b]\in B(G)$ satisfying (1), (2) and
\begin{enumerate}[(3)]
\item
$\nu(b)\geq \avg_\sigma(v^{-1}\mu - \wt(v\Rightarrow\sigma(wv)))$, where $\wt(\cdot\Rightarrow\cdot)$ denotes the weight function of the quantum Bruhat graph.
\end{enumerate}

\item Let $[b]\in B(G)_x$. The affine Deligne-Lusztig variety $X_x(b)$ is equidimensional of dimension
\begin{align*}
\dim X_x(b) = \frac 12\Bigl(\ell(x)+\ell_{R,\sigma}(w) - \langle \nu(b),2\rho\rangle - \mathrm{def}(b)\Bigr).
\end{align*}
\item Let $[b]\in B(G)_x$. The natural group action of $J_b(F) = \{g\in G(\breve F)\mid g^{-1}b\sigma(g) = b\}$ on $X_x(b)$ by left multiplication induces a transitive action on the set of irreducible components.
\end{enumerate}
\end{theorem}
There are a handful more consequences of the positive Coxeter property, such as a precise understanding of the Deligne-Lusztig reduction trees \cite[Theorem 5.7 (c), (d)]{Schremmer2023_coxeter} and a description of $J_b(F)$-stabilizers of irreducible components \cite[Proposition~6.7]{Schremmer2023_coxeter}.

In particular, all elements of positive Coxeter type are also of \emph{geometric Coxeter type} as defined by Nie-Schremmer-Yu \cite{Nie2025}. The definition of geometric Coxeter type is a bit technical, and formulated in terms of reduction trees. It has the following consequence:
\begin{theorem}[{\cite[Theorem~6.9]{Nie2025}}]\label{thm:gctConsequences}
Let $x\in\widetilde W$ be of positive Coxeter type and let $[b]\in B(G)_x$. Then one can compare $X_x(b)$ to the direct product
\begin{align}
(\mathbb G_m)^{\ell_1}\times \mathbb A^{\ell_2}\times C\times D,\label{eq:geoCoxGeometry}
\end{align}
where $\ell_1,\ell_2\geq 0$ are integers depending on $x$ and $[b]$, where $C$ is a classical Deligne-Lusztig variety of Coxeter type and where $D$ is a discrete set of points. More precisely, if $F$ has equal characteristic, there is a universal homeomorphism from $X_x(b)$ to \eqref{eq:geoCoxGeometry} (in the category of schemes). In mixed characteristic, there is an isomorphism in the category of perfect schemes from $X_x(b)$ to \eqref{eq:geoCoxGeometry}.
\end{theorem}

In the next section, we will prove the following result.
\begin{theorem}\label{thm:depth2PositiveCoxeterType}
    Assume that every connected component of the Dynkin diagram of $G$ is of type $A$. Let $\mu$ be a dominant cocharacter with $\depth(G,\mu)<2$. Then every element $x\in \Adm(\mu)\cap \widetilde W^K$ is of positive Coxeter type.
\end{theorem}
The combination of Theorems \ref{thm:ssyPositiveCoxeterType} and \ref{thm:depth2PositiveCoxeterType} can be used to prove Theorem~\ref{thm:EOGeometry} in most cases. Before we give the proof, we consider two exceptional cases, labeled $(C_3, \omega_3^\vee)$ and $(D_5, \omega_5^\vee)$ in Theorem~\ref{prop:classification}, in which a more explicit argument is needed.
\begin{example}\label{ex:39}
Let $G$ be the symplectic group $G = \mathrm{Sp}_6$ and set $\mu = \omega_3^\vee$. Then, the admissible set $\Adm(\mu)\cap \widetilde W^K$ contains eight elements, which are easily enumerated. We use a notational shortcut expressing elements $w$ with reduced word $w = s_{i_1}\cdots s_{i_{\ell(w)}}$ as $w = s_{i_1 i_2\cdots i_{\ell(w)}}$. Denote the unique length zero element in $\Adm(\mu)$ by $\tau$. The following seven elements are of positive Coxeter type.
 \begin{align*}
    s_{323123} \varepsilon^{-\mu} = \tau&\quad\text{ with }\quad v=s_{12312},
    \\s_{3123}\varepsilon^{-\mu} = \tau s_{10}& \quad \text{ with }\quad v=s_{2312312},
    \\s_{323}\varepsilon^{-\mu} = \tau s_{210}&\quad \text{ with }\quad v=s_{3123121},
    \\s_{123}\varepsilon^{-\mu} = \tau s_{010}&\quad \text{ with }\quad v=s_{323123},
    \\s_{23}\varepsilon^{-\mu} = \tau s_{2010}&\quad \text{ with }\quad v=s_{32312321},
    \\s_3\varepsilon^{-\mu} = \tau s_{12010}&\quad \text{ with }\quad v=s_{32312312},
    \\\varepsilon^{-\mu} = \tau s_{012010}&\quad \text{ with }\quad v = s_{323123}
 \end{align*} The unique element that is not of positive Coxeter type is 
 \begin{align*}
     x = s_{23123} \varepsilon^{-\mu} = \tau s_0 \in\widetilde W.
 \end{align*} 
    Note that its classical part $w = s_2 s_3 s_1 s_2 s_3\in W_0$ is not conjugate to any partial Coxeter element. We have $\ell(x)=1$, so $x$ is of minimal length in its $\sigma$-conjugacy class of $\widetilde W$. Writing $x = s_3\tau $, where $\tau\in \Adm(\mu)$ is the unique length zero element, one checks that $x$ satisfies the \emph{minimal Coxeter type} condition from \cite[Definition~5.3]{Nie2025}. So the claims of Theorem~\ref{thm:EOGeometry} for that given element $x\in\widetilde W$ follow from \cite[Theorems 3.5 and 4.8]{He2014} as well as a straightforward computation, together with Theorem~\ref{thm:gctConsequences}. Indeed, we get $\ell_R(w)=3$ and $[b] = [\dot x]$ satisfies $\nu(b)=0$ and $\mathrm{def}(b)=2$. Hence, we get
    \begin{align*}
    \dim X_x(b) &= \ell(x)-\langle \nu(b),2\rho\rangle = 1 = \frac 12\left(1+3-0-2\right) \\&= \frac 12\left(\ell(x)+\ell_R(w) -\langle\nu(b),2\rho\rangle-\mathrm{def}(b)\right).
    \end{align*}
\end{example}
\begin{example}\label{ex:310}
    Let now $G$ be the special orthogonal group $G = \mathrm{SO}_{10}$ and set $\mu = \omega_5^\vee$. Then the admissible set $\Adm(\mu)\cap \widetilde W^K$ contains $16$ elements, which are again easily enumerated with a computer. Of those, $15$ are of positive Coxeter type. The unique exception is given by
    \begin{align*}
        x = s_2 s_3 s_5 s_1 s_2 s_3 s_4 \varepsilon^{-\omega_4^\vee}\in\widetilde W.
    \end{align*}
    The fact that $x$ has positive Coxeter type can be verified using the definition by constructing a reduction tree, as explained in \cite{Nie2025}.
    Then the claims of Theorem~\ref{thm:EOGeometry} can be verified for that given element $x$ with a bit of effort, e.g.\ using said reduction tree together with the Deligne-Lusztig reduction method from Görtz-He \cite{Goertz2010b}. The explicit description of irreducible components comes from Theorem~\ref{thm:gctConsequences}.
\end{example}
\begin{proof}[Proof of Theorem~\ref{thm:EOGeometry} using Theorem~\ref{thm:depth2PositiveCoxeterType}]
Let $(G,\mu)$ be as in the theorem. The claims can be verified individually for each $\sigma$-connected component of the Dynkin diagram, so let us assume without loss of generality that $G$ is relatively quasi-simple.

If $\depth(G,\mu)\leq 1$, then parts (a) and (b) as well as the \enquote{For comparison} dimension formula in (c) follow from \cite[Theorem B]{Goertz2019}.

Let us hence assume that $1<\depth(G,\mu)<2$. Then we are in one of the cases described in Theorem~\ref{prop:classification}. If $G$ is of split type $A$, then $x$ has positive Coxeter type by Theorem~\ref{thm:depth2PositiveCoxeterType}. It also has geometric Coxeter type by comparing \cite[Theorem 5.7 (c)\&(d), Section 5.3]{Schremmer2023_coxeter} with the Definition of geometric Coxeter type in \cite{Nie2025}. We get all claims from Theorems \ref{thm:ssyPositiveCoxeterType} and \ref{thm:gctConsequences}. If $G$ is not of split type $A$, then we must be in the situation of Example \ref{ex:39} or \ref{ex:310}, and in either case we get the desired claims.
\end{proof}

\subsection{Proof of Theorem~\ref{thm:depth2PositiveCoxeterType}}

In this section we prove Theorem~\ref{thm:depth2PositiveCoxeterType}. We begin with two technical lemmas, the first being a reduction result that allows us to pass to smaller groups in many cases.
\begin{lemma}\label{lem:positiveCoxeterPAlcove}
    Let $\mu$ be a dominant cocharacter, consider an element $x=w\varepsilon^{\mu_x}\in\Adm(\mu)\cap \widetilde W^K$ and some $v\in W_0$. Set $J := \supp_{\sigma}(v^{-1}\sigma(wv))$. Assume that $v\in W_0^J$ and that $\ell(x,v\alpha)\geq 0$ for all $\alpha\in\Phi^+\setminus \Phi^+_J$.

    Denote the $\sigma$-stable Levi subgroup of $G$ defined by $J$ by $M = M_J$, and let $\tilde x = \sigma^{-1}(v)^{-1} xv\in \widetilde W_M$.
    \begin{enumerate}[(a)]
    \item The cocharacter $\mu' := (vw_0(J))^{-1} \mu_x$ is dominant for $M\cap B$ and we have $\tilde x \in \Adm_M(\mu') \cap \widetilde W_M^{K\cap M}$.

    \item The element $x\in\widetilde W$ has positive Coxeter type for $G$ if and only if $\tilde x\in\widetilde W_M$ has positive Coxeter type for $M$.
    \end{enumerate}
\end{lemma}
\begin{proof}
    (a) The condition $v\in W_0^J$ is equivalent to $\ell(v,\alpha)=0$ for all $\alpha\in\Phi_J$. Thus $\ell(\tilde x,\alpha) = \ell(x,v\alpha)$ for all $\alpha\in \Phi_J$ by \cite[Lemma~2.12]{Schremmer2022_newton}.
    
    The condition $x\in\widetilde W^K$ is equivalent to $\ell(x,\alpha)\leq 0$ for all $\alpha\in\Phi^+$. Given $\alpha\in\Phi^+_J$, we thus see $\ell(\tilde x,\alpha) = \ell(x,v\alpha)\leq 0$ since $v\alpha\in\Phi^+$. This implies $\tilde x\in\widetilde W_M^{K\cap M}$. 
    
    In particular, the longest element $w_0(J)\in W_{0,M}$ is length positive for $\tilde x\in\widetilde W_M$. We also compute for $\alpha\in J$ that
    \begin{align*}
        \langle \mu',\alpha\rangle &= \langle \mu_x, vw_0(J)\alpha\rangle = \ell(x, vw_0(J)\alpha) - \Phi^+(vw_0(J)\alpha) + \Phi^+(wvw_0(J)\alpha)\\&\geq \ell(x, vw_0(J)\alpha)\geq 0
    \end{align*}
    since $w_0(J)\alpha\in\Phi_J^-$ so $vw_0(J)\alpha\in\Phi^-$. This shows the dominance claim. The fact $\tilde x\in\Adm_M(\mu')$ is now easy to see through elementary arguments, or by appealing to \cite[Proposition~4.12]{Schremmer2024_bruhat}.

    Assertion (b) is \cite[Lemma~5.14]{Schremmer2023_coxeter}.
\end{proof}

\begin{lemma}\label{lem:admDecomposition}
    Let $\mu_1, \mu_2\in X_*(T)^{\dom}$ and let $x = w\varepsilon^{\mu_x}\in\widetilde W$ for some $w\in W_0$ and $\mu_x\in X_*(T)$.
    \begin{enumerate}[(a)]
    \item We have $x\in \Adm(\mu_1)^K$ if and only if $x\in\widetilde W^K$ and $\mu_{\dom} = w_0\mu\leq \mu_1$.
    \item Every element $x\in\Adm(\mu_1+\mu_2)^K$ allows a decomposition $x = x_1 x_2$ such that $\ell(x) = \ell(x_1) + \ell(x_2)$ and $x_i\in \Adm(\mu_i)^K$ for $i=1,2$.
    \end{enumerate}
\end{lemma}
\begin{proof}
\begin{enumerate}[(a)]
\item Certainly the condition $x\in\widetilde W^K$ is necessary for $x\in \Adm(\mu_1)^K$, so let us assume this from now on. Then $\mu_{\dom} = w_0 \mu$ and $w_0\in \LP(x)$. By \cite[Proposition~4.12]{Schremmer2024_bruhat}, we get $x\in \Adm(\mu_1)$ if and only if $w_0\mu \leq \mu_1$.
\item 
By \cite[Theorem~5.1]{He2016}, we find elements $\tilde x_1\in \Adm(\mu_1)$ and $\tilde x_2\in\Adm(\mu_2)$ with $x = \tilde x_1 \tilde x_2$. Pick reduced words
\begin{align*}
    \tilde x_i = \tau_i s_{\alpha^{(i)}_1}\cdots s_{\alpha^{(i)}_{\ell(x_i)}}\in\widetilde W
\end{align*}
with $\tau_i$ of length $0$. Composing these words for $\tilde x_1$ and $\tilde x_2$ yields a word for $x$ in $\widetilde W$. By the deletion property of Coxeter groups, we can find a reduced subword
\begin{align*}
x = \tau_1 s_{\alpha^{(1)}_{i_1}}\cdots s_{\alpha^{(1)}_{i_m}} \tau_2 s_{\alpha^{(2)}_{j_1}}\cdots s_{\alpha^{(2)}_{j_{\ell(x)-m}}}.
\end{align*}
Put $x_1' = \tau_1 s_{\alpha^{(1)}_{i_1}}\cdots s_{\alpha^{(1)}_{i_m}}$ and $x_2' = \tau_2 s_{\alpha^{(2)}_{j_1}}\cdots s_{\alpha^{(2)}_{j_{\ell(x)-m}}}$. Then $x = x_1' x_2'$ is a length additive product and $x_i'\leq x_i\in \Adm(\mu_i)$. Write $x_1' = x_1 u$ with $x_1\in\widetilde W^K$ and $u\in W_0$. Then $x = x_1 u x_2'$ is a length additive product. Set $x_2 = u x_2'$. We get $x_1\in\widetilde W^K$ and $x_1\leq x_1'\leq \tilde x_1\in\Adm(\mu_1)$, proving $x_1\in \Adm(\mu_1)^K$. The condition $x\in\widetilde W^K$ implies $x_2', x_2\in \widetilde W^K$. Hence $x_2\in\Adm(\mu_2)\iff x_2'\in\Adm(\mu_2)$ by (a). Now recall $x_2'\leq \tilde x_2\in\Adm(\mu_2)$.
\qedhere\end{enumerate}
\end{proof}

For the remainder of this subsection, we assume that $G$ is adjoint, every connected  component of the Dynkin diagram of $G$ is type $A$ and that $\depth(G,\mu)<2$ for a dominant cocharacter $\mu$. The classification in Theorem \ref{prop:classification} and the reduction steps at the beginning of Section \ref{secclassification} show that this implies that $G$ is isogeneous to a direct product of restrictions of scalars of groups $\mathrm{PGL}_{n}$ for varying $n$. Our goal is to prove that every element $x\in\Adm(\mu)\cap \widetilde W^K$ has positive Coxeter type. For this, we can easily reduce to the case where $G$ is isomorphic to a Weil restriction of scalars of a fixed group $G' = \mathrm{PGL}_n$ along an unramified field extension $F^{(f)}/F$.

A related notion in this context is that of \emph{Shimura varieties of Coxeter type} \cite{Goertz2015, Goertz2024}. If $\depth(G,\mu)\leq 1$, then we say that $(G,\mu,K)$ is of \emph{Coxeter type} if each $x\in \Adm(\mu)\cap \widetilde W^{K}$ such that $IxI$ is contained in a basic $\sigma$-conjugacy class is a partial Coxeter element of $\widetilde W$. This condition suggests a particularly beautiful structure on the associated Rapoport-Zink spaces. In the reduction in the proof of Theorem \ref{thm:depth2PositiveCoxeterType} we are also led to such cases, and need the following lemma.
\begin{lemma}\label{lem:pctNonsplitGLn}
    Assume that $G$ is adjoint with connected Dynkin diagram of type $A_n$ and such that the Frobenius action satisfies $\sigma(\alpha) = -w_0\alpha$ for all $\alpha\in\Phi$. Let $\mu = \omega_1^\vee$ and $x\in \Adm(\mu)\cap \widetilde W^K$. Then at least one of the following claims is true:
    \begin{enumerate}[(i)]
    \item\label{316pr1} The element $x$ is of positive Coxeter type or
    \item\label{316pr2} There is a $v\in \LP(x)$ such that $J := \supp_\sigma(v^{-1}\sigma(wv))\subsetneq \Delta$.
    \end{enumerate}
\end{lemma}
\begin{proof}
    From \cite[Theorem~3.3]{Goertz2019}, we see that at least one of the following statements is true:
    \begin{enumerate}[(a)]
    \item\label{316pra} The double coset $IxI$ is contained in a basic $\sigma$-conjugacy class or
    \item\label{316prb} the element $x\in\widetilde W$ is straight.
    \end{enumerate}
    In Case (\ref{316pra}), we can apply \cite[Theorem~1.4]{Goertz2024} to see that $x$ is, in fact, a partial $\sigma$-Coxeter element inside $\widetilde W$. In particular, $x$ must be of minimal length in its $\sigma$-conjugacy class in $\widetilde W$. This is also well-known in Case (\ref{316prb}).

    Assume that (\ref{316pr2}) is not satisfied. This means that $x$ is not a $(J, v, \sigma)$-alcove element for any $v\in W_0$ and $J\subsetneq\Delta$, cf.\ \cite[Section~3]{Schremmer2023_coxeter}. We apply \cite[Lemma~5.15]{Schremmer2023_coxeter} to write $x = \tau u$ such that $\tau \in\Omega$ has length zero, $u\in W_{\af}$ and $\tilde K = \supp_{\sigma\circ\tau}(u)\subsetneq \Delta_{\af}$ is maximal among all $\sigma\circ\tau$-stable proper subsets of $\Delta_{\af}$. By \cite[Theorem~1.4]{Goertz2024}, $u$ is actually a $\sigma\circ\tau$-Coxeter element of $\tilde K$.

    Let us write the simple reflections of $W_0$ as $s_1,\dotsc,s_n$ in the usual fashion for $G=\PGL_{n+1}$, and let $s_0$ denote the affine simple reflection in $\widetilde W$ that is not in $W_0$. The condition $x\in \Adm(\mu)^K$ forces $u\in (\widetilde W)^K$ and $\tau = \tau_{\mu} = \varepsilon^\mu s_1 \cdots s_n$. The $\sigma\circ\tau$-orbits of simple affine reflections are given by $\{s_0, s_n\},\{s_1,s_{n-1}\},\ldots$. So the only possibility for $u\in \widetilde W^K$ is given by $u = s_{\lfloor n/2\rfloor-1}\cdots s_1 s_0$. Then $x = \tau u = s_{\lfloor n/2\rfloor}\cdots s_2 s_1 \tau$ has classical part
    \begin{align*}
        \cl(x) = s_{\lfloor n/2\rfloor}\cdots s_2 s_1 s_1\cdots s_n = s_{\lfloor n/2\rfloor+1}\cdots s_{n-1} s_n.
    \end{align*}
    Observe that this is a $\sigma$-Coxeter element in $W_0$. Then (\ref{316pr1}) follows from \cite[Theorem~5.12]{Schremmer2023_coxeter}.  
\end{proof}
To handle the remaining cases, we need some assertions for the symmetric group $S_n$.

\begin{lemma}\label{lem:symmetricGroupPartialCoxeter}
Let $w\in S_n$. Then there exists $v\in S_{n-1}\subset S_n$ such that $v^{-1} w v\in S_n$ is a partial Coxeter element of $S_n$.
\end{lemma}
\begin{proof}
We write $w$ as a product of disjoint cycles of lengths $\geq 1$. We may order the cycles and the elements within them such that $n$ is the last entry of the last such cycle. Let $v\in S_{n-1}$ be the permutation ordering the other entries of the cycles by size. Then $v$ is as claimed.
\end{proof}
\begin{lemma}\label{lem:pctDevisage}
Let $n\geq 1$ and $i,j\in\{0,\dotsc,n\}, k\in \{1,\dotsc,n+1\}$. Define the elements $w_1, w_2, w_3\in S_{n+1}$ as follows:
\begin{align*}
w_1 = s_1s_2\cdots s_i,\quad w_2 = s_1s_2\cdots s_j,\quad w_3 = s_n s_{n-1}\cdots s_k\in S_n.
\end{align*}
Then there exist elements $v_1, v_2, v_3\in S_{n+1}$, subject to the requirements
\begin{align*}
v_1^{-1}(m)>v_1^{-1}(i+1)&\text{ for all } m>i+1,\\
v_2^{-1}(m)>v_2^{-1}(j+1)&\text{ for all } m>j+1,\\
v_3^{-1}(m)<v_3^{-1}(k)&\text{ for all } m<k,
\end{align*}
such that at least one of the following conditions holds true:
\begin{enumerate}[(i)]
\item The joint support
\begin{align*}
\supp(v_3^{-1} w_1 v_1)\cup \supp(v_1^{-1} w_2 v_2)\cup \supp(v_2^{-1} w_3 v_3)\subseteq \{s_1,\dotsc,s_n\}
\end{align*}
is not the full set $\{s_1,\dotsc,s_n\}$ or
\item The three elements
\begin{align*}
v_3^{-1} w_1 v_1,~ v_1^{-1} w_2 v_2, ~v_2^{-1} w_3 v_3\in S_{n+1}
\end{align*}
are partial Coxeter elements of disjoint support.
\end{enumerate}
\end{lemma}
\begin{remark}
We will see later in this section that, in fact, Condition (ii) can always be achieved. For now, it suffices to prove the lemma as stated (which is significantly easier).
\end{remark}
\begin{proof}[Proof of Lemma~\ref{lem:pctDevisage}]
Written as permutations of $\{1,\dotsc,n+1\}$, we identify $w_1, w_2, w_3$ as the cycles
\begin{align*}
w_1 = (1,2,3,\dotsc,i+1),\quad w_2 = (1,2,3,\dotsc,j+1),\quad w_3 = (n,n-1,\dotsc,k).
\end{align*}
Now we distinguish a number of cases:
\begin{itemize}
\item \textbf{Case $k>\max(i,j)+1$}: In this case, we simply choose $v_1=v_2=v_3=1$ and see that $s_{k-1}$ is not in the support of any $w_\bullet$.
\item \textbf{Case $i\leq j-1$ and $k\leq j$}: In this case, choose $v_1=v_3$ as the transposition $(j+1, n+1)$ interchanging $j+1$ and $n+1$, and let $v_2 \in S_n$ with $v_2^{-1}(j)=n+1$ and $v_2^{-1}(j+1) = 1$. The required inequalities for the $v_\bullet$ are easily checked. Moreover, Condition (i) follows since the three elements
\begin{align}\label{eqast}
w_1' = v_3^{-1} w_1 v_1,~ w_2' = v_1^{-1} w_2 v_2, ~w_3' = v_2^{-1} w_3 v_3
\end{align}
all have $n+1$ as fixed point.
\item \textbf{Case $j\leq i-2$ and $k\leq i$}: In this case, choose $v_3 = (i+1, n+1)$, $v_2 = (i, n+1)$ and $v_1\in S_{n+1}$ is chosen such that $v_1^{-1}(i)=n+1, v_1^{-1}(i+1) = 1$. Then again, $w_1', w_2', w_3'$ defined as in \eqref{eqast} have $n+1$ as fixed point.
\item \textbf{Case $i\leq j = k-1$}: For an element $u\in S_i = \langle s_1,\dotsc,s_{i-1}\rangle$ to be determined later, we choose $v_1 = s_i\cdots s_1 u, v_2 = s_j \cdots s_{i+1} u$ and $v_3 = u$. The required inequalities for the $v_\bullet$ are immediately verified. With the notation of (\ref{eqast}), we get $w_1' = 1, w_3' = s_{i+1}\cdots s_j s_n\cdots s_k$ and
\begin{align*}
w_2' = u^{-1} s_1\cdots s_i s_1\cdots s_i u.
\end{align*}
Observing that these three elements have pairwise disjoint supports, it suffices to choose $u$ such that $w_2'$ is a partial Coxeter element. This is possible by Lemma~\ref{lem:symmetricGroupPartialCoxeter}.
\item \textbf{Case $j<i = k-1$}: For an element $u\in S_j = \langle s_1,\dotsc,s_{j-1}\rangle$ to be determined later, we choose $v_1 = s_i\cdots s_{j+1} u$ and $v_2 = v_3 = s_j\cdots s_1 u$. The required inequalities for the $v_\bullet$ are immediately verified. With the notation of \eqref{eqast}, we get
\begin{align*}
w_1' = u^{-1} s_1\cdots s_j s_1\cdots s_j u,\quad w_2' = s_{j+1}\cdots s_i,\quad w_3' = s_n\cdots s_k.
\end{align*}
Observing that these three elements have pairwise disjoint supports, it suffices to choose $u$ such that $w_1'$ is a partial Coxeter element. This is possible by Lemma~\ref{lem:symmetricGroupPartialCoxeter}.
\end{itemize}
The above covers the cases $k\geq \max(i,j)+1$ and $k=\max(i,j)+1$ unconditionally, as well as $k\leq\max(i,j)$ whenever $i\leq j-1$ or $j\leq i-2$. So let us assume that $i\in \{j, j+1\}$ and $k\leq\max(i,j) = i$. For an element $u\in S_{i-1}$ to be determined later, let $v_1 = s_1\cdots s_j u,~v_2 = u,~ v_3 = s_k\cdots s_i u$. With the notation of \eqref{eqast}, we get $w_2'=1,~ w_3' = s_n\cdots s_{i+1}$ and
\begin{align*}
w_1' = u^{-1} s_i\cdots s_k s_1\cdots s_j s_1\cdots s_i u.
\end{align*}
Note that $w_1'\in S_i$. By Lemma~\ref{lem:symmetricGroupPartialCoxeter}, we see that we can choose $u$ such that $w_1'$ is a partial Coxeter element in $S_i$. The required inequalities for $v_1, v_2, v_3$ are immediately satisfied. This finishes the proof.
\end{proof}

For the next lemma, we call an element $x = w\varepsilon^{\mu_x}$ in the extended affine Weyl group $\widetilde W$ of some reductive group $G$ to be \emph{Levi restrictable} if there exists $v\in W_0$ such that $J := \supp_\sigma(v^{-1}\sigma(wv))\subsetneq \Delta$, and moreover $\ell(x,v\alpha)\geq 0$ for all $\alpha\in \Phi^+\setminus \Phi_J$.

\begin{lemma}\label{lem:pctRes}
   For each $d\geq 1$, fix an unramified extension $F_d$ of $F$. Let $G'$ be a Chevalley group (defined and split over $\mathbb Z$) and let $G_d = \Res_{F_d / F} G'$ be the restriction of scalars for $d\geq 1$. We write elements in the Iwahori-Weyl group $\widetilde W_{G_d}$ of $G_d$ as $d$-tuples of elements of $\widetilde W_{G'}$ as in Section~\ref{sec:weilRestrict}.
   
   Let $2\leq k\leq d<\infty$ and consider $x = (x_1,\dotsc, x_d)\in\widetilde W_d$. We set \begin{align*}x' := (x_1,\dotsc,x_{k-2},x_{k-1}x_k, x_{k+1},\dotsc,x_d)\in\widetilde W_{G_{d-1}}.\end{align*}
   \begin{enumerate}[(a)]
   	\item 
   	If $x$ is Levi restrictable, then so is $x'$. In case $\ell(x_k)=0$, this is an equivalence.
   	\item If $x$ has positive Coxter type, then so does $x'$. In case $\ell(x_k)=0$, this is an equivalence.
   \end{enumerate}
\end{lemma}
\begin{proof}
 We first prove (a) and begin with the special case $x_k=1$. Then, the equivalence follows directly from studying the definition of being Levi restrictable:
		
		Unfolding the definition of $x$ being Levi restrictable, it means that there is a proper subset $J\subsetneq \Delta_{G'}$ of the set of simple roots of $G'$ and elements $v_1,\dotsc,v_d\in W_{0,G'}$ in the finite Weyl group of $G'$ such that
		\begin{itemize}
		\item For every positive root $\alpha$ of $G'$, not contained in the span of $J$, and for all $i=1,\dotsc,d$, we have $\ell(x_i, v_i\alpha)\geq 0$ and
		\item for all $i=1,\dotsc,d-1$, we have
		\begin{align*}
		v_i^{-1}\cl(x_{i-1}) v_{i-1}\in W_J,\qquad v_d^{-1} \cl(x_1) v_1\in W_J.
		\end{align*}
		\end{itemize}
		Now if we are given $J,(v_1,\dotsc,v_d)$ satisfying these properties, the condition $x_k=1$ shows $v_k^{-1} v_{k-1}\in W_J$. Therefore, we see that $x'$ is Levi restrictable by unfolding the definition as above, and setting $J' = J$ and $(v_1',\dotsc,v_{d-1}') = (v_1,\dotsc, v_{k-1}, v_{k+1},\dotsc, v_d)$.
		
		Conversely, if $x'$ is Levi restrictable, we get $J'$ and $(v_1',\dotsc, v_{d-1}')$ satisfying the analogous conditions. Since $x_k=1$, we can set \begin{align*}
			J = J'\text{ and }(v_1,\dotsc, v_d) = (v_1',\dotsc, v_{k-1}', v_{k-1}',v_k',\dotsc, v_{d-1}').
		\end{align*}
		
		This settles the claim in case $x_k=1$. Next, we consider the case $\ell(x_k)=0$. Then $x$ and
		\begin{align*}
		\tilde x = (x_1,\dotsc,x_{k-2}, x_{k-1} x_k,1,x_{k+1},\dotsc,x_d)\in \widetilde W_{G_d}
		\end{align*}
		are related through $\sigma$-conjugation of a length zero element. Note that the condition of being Levi restrictable is invariant under this operation, and note that $x' = (\tilde x)'$. This settles all claims in case $\ell(x_k)=0$.
		
		Finally, we tackle the case $\ell(x_k)>0$ through induction on $\ell(x_k)$. Let $r_a$ be a simple affine reflection such that $\ell(r_a x_k)=\ell(x_k)-1$, and consider
		\begin{align*}
		\tilde x = (x_1,\dotsc, x_{k-2}, x_{k-1} r_a,r_a x_k, x_{k+1},\dotsc, x_d)\in\widetilde W_{G_d}.
		\end{align*}
		Then $\ell(\tilde x)\leq \ell(x)$ and $\tilde x$ arises from $x$ by $\sigma$-conjugation by a simple affine reflection. If $x$ is Levi restrictable, then so is $\tilde x$ by \cite[Lemma~7.1]{He2015b}.

The proof of (b) is a very similar argument as for the first case above. 
Compatibility with $\sigma$-conjugation by simple affine reflections, as above, is found in \cite[Lemmas~5.1~and~5.2]{Schremmer2023_coxeter}.
	\qedhere
\end{proof}

\begin{proposition}\label{prop:pctRes}
Let $F'/F$ be an unramified extension of local fields. For $n\geq 1$ write $G'$ for the split group $G' = \PGL_n$ over $F'$. Let $G = \Res_{F'/F}\PGL_n$ and consider a dominant cocharacter with $\mu\in X_\ast(T)$ with $\depth(G,\mu)<2$. Let $x=w\varepsilon^\mu\in\Adm(\mu)^K$. Then at least one of the following conditions is true.
    \begin{enumerate}[(i)]
    \item The element $x$ is of positive Coxeter type.
    \item There is a $v\in \LP(x)$ such that $J := \supp_\sigma(v^{-1}\sigma(wv))\subsetneq \Delta$.
    \end{enumerate}
\end{proposition}
\begin{proof}
We use induction on $n$.

Let $[F':F] = d$ and decompose $\Phi$ into its irreducible components $\Phi = \Phi_1\sqcup\cdots\sqcup \Phi_d$ with $\sigma(\Phi_k) = \Phi_{k+1}$ and $\sigma^d=\mathrm{id}$. Denote the restrictions of $\mu$ to the individual root lattices $\mathbb Z \Phi_k$ by $\mu_1,\dotsc,\mu_d$. Denote the restriction of $x$ to the individual extended affine Weyl groups by $x_k \in W_k\ltimes \Hom(\mathbb Z\Phi_k,\mathbb Z)$.

By Lemma \ref{lem:pctRes} we can easily reduce to the case where $\mu$ is non-central in every connected component of the Dynkin diagram. Moreover, if $\mu_k\in \Hom(\mathbb Z\Phi_k,\mathbb Z)$ is a sum of two non-central dominant coweights $\mu_k = \mu_k^{(1)} + \mu_k^{(2)}$, proving the proposition for $\Res_{F^{(d+1)}/F}\PGL_n$ (where $F^{(d+1)}/F$ is unramified of degree $d+1$) and the cocharacter $(\mu_1,\dotsc,\mu_{k-1},\mu_k^{(1)},\mu_k^{(2)},\mu_{k+1},\dotsc,\mu_d)$ implies the claim for $(G,\mu)$ (using Lemma~\ref{lem:admDecomposition} and Lemma \ref{lem:pctRes}). In other words, we may assume that each $\mu_k$ is non-zero and minuscule.

Suppose that there exists a dominant coweight $0\neq \mu'\in X_\ast(T)_{\Gamma_0}$ such that $\depth(G,\mu+\mu')<2$. Suppose moreover that the statement has been proved for $G$ and $\mu+\mu'$. Then any element $x\in \Adm(\mu)^{\mathbb S}$ gives rise to an element $x' = x\varepsilon^{w_0\mu'}\in \Adm(\mu+\mu')^{\mathbb S}$ with $\LP(x')\subseteq \LP(x)$. Hence if $x'$ has positive Coxeter type, so does $x$.

Assume from now on that any non-zero dominant coweight $\mu'\in X_\ast(T)_{\Gamma_0}$ satisfies $\depth(G,\mu+\mu')\geq 2$. Moreover, if there exists a positive root $\alpha$ with $\mu + \alpha^\vee$ being dominant and of depth $<2$, it suffices to prove the claim for $(G,\mu+\alpha^\vee)$, as $\Adm(\mu)\subset \Adm(\mu+\alpha^\vee)$.

After all these reductions, by Theorem~\ref{prop:classification}, it remains to study the following cases.

\noindent Case 1: $n=1$ with $d=3$ and $\mu = (\omega_1,\omega_1,\omega_1)$\\
Case 2: $n=4$ with $d=2$ and $\mu = (\omega_1,\omega_2)$\\
Case 3: $n=5,6,7$ with $d=1$ and $\mu = \omega_3$\\
Case 4: $n\geq 2$ with $d=3$ and $\mu = (\omega_1,\omega_n, \omega_n)$\\

The first three cases can be easily settled with a computer. Let us study Case 4. Writing the simple roots of $\Phi_1$ as $\alpha_1,\dotsc,\alpha_n$ and the simple reflections as $s_1,\dotsc,s_n$, observe that
\begin{align*}
{}^K\Adm(\mu_1) = \{\varepsilon^{\omega_1}s_1\cdots s_j\mid j=0,\dotsc,n\},\quad
{}^K\Adm(\mu_3) = \{\varepsilon^{\omega_n}s_n\cdots s_j\mid j=1,\dotsc,n+1\}.
\end{align*}
We calculate
\begin{align*}
&\LP(\varepsilon^{\omega_1}s_1\cdots s_j) = \LP(\tau_1 s_n\cdots s_{j+1}) = \LP(s_n\cdots s_{j+1}) \\&\quad= \{v\in S_n \mid \ell(s_n\cdots s_{j+1} v) = \ell(s_n\cdots s_{j+1})+\ell(v)\}
\\&\quad=\{v\in S_n\mid \forall \alpha\in \Phi_1^+:~ s_n\cdots s_{j+1}\alpha\in \Phi_1^-\text{ implies that } v^{-1}\alpha\in \Phi_1^+\}
\\&\quad = \{v\in S_n\mid v^{-1}(m)>v^{-1}(j+1) \text{ for all } m>j+1\}.
\end{align*}
An entirely similar calculation shows
\begin{align*}
&\LP(\varepsilon^{\omega_n}s_n\cdots s_j)= \{v\in S_n\mid v^{-1}(m)<v^{-1}(j)\text{ for all } m<j\}.
\end{align*}
Note that $x^{-1}\in {}^K\Adm(-w_0 \mu)$. So up to inverting $x$ and $\sigma$, it suffices to find elements $v = (v_1, v_2, v_3)\in \LP(x)$ such that $(x,v)$ is a pair of positive Coxeter type, or such that $J = \supp_\sigma(v^{-1}\sigma(\cl(x) v))$ is not all of $\Delta$. This is precisely the statement of Lemma~\ref{lem:pctDevisage}.
\end{proof}
\begin{proof}[Proof of Theorem~\ref{thm:depth2PositiveCoxeterType}]
Let $\mu$ be dominant with $\depth(G,\mu)<2$, and let $x=w\varepsilon^\mu\in \Adm(\mu)^K$. It suffices to prove the positive Coxeter type condition for each $\sigma$-connected component of the Dynkin diagram. So let us assume that $G$ is relatively quasi-simple. By inspecting the classification from Theorem~\ref{prop:classification} and \cite[Theorem~3.5]{Goertz2019}, we only have to study a finite number of families. In each case, we see that $x$ is of positive Coxeter type or that there is $v\in \LP(x)$ with $J = \supp(v^{-1} \sigma(wv))\neq\Delta$, cf.\ Lemma~\ref{lem:pctNonsplitGLn} and Proposition~\ref{prop:pctRes}.

Let $v^J\in W^J\cap vW_J$ be the right coset representative. Then applying Lemma~\ref{lem:positiveCoxeterPAlcove} to $(x, v^J)$ produces a proper Levi subgroup $M$ together with cocharacter $\mu'\in W_0 \mu$ as well as an element $\tilde x\in \Adm_M(\mu')^{K\cap M}$. By Theorem~\ref{lemdepthprop}, we get $\depth(M, \mu')\leq \depth(G,\mu)<2$, and every connected component of the Dynkin diagram of $M$ has type $A$. In an inductive manner, we may assume that we have proved already that $\tilde x$ has positive Coxeter type for $M$. Then the claim follows from Lemma \ref{lem:positiveCoxeterPAlcove} (b).
\end{proof}
\begin{remark}
    Shimada \cite{Shimada2024_pct} considers the general linear group $G = \GL_n$ and classifies cocharacters $\mu$ of positive Coxeter type, which means that the statement of Theorem~\ref{thm:depth2PositiveCoxeterType} holds true. His list consists of our cases for $G = \GL_n$, the infinite family of dominant cocharacters for $G=\GL_2$, as well as a finite number of additional cases. In this sense, our Theorem~\ref{thm:depth2PositiveCoxeterType} specialized to simple split groups is equivalent to \cite[Theorem~A]{Shimada2024_pct}. We remark that our proof method handles more cases, while only requiring ad hoc treatments for two individual cases.
\end{remark}

\section{Geometry of single ADLV in flag varieties}\label{sec:4}

We continue with the same assumptions as in Section~\ref{sec:notation}, i.e.\ unless specified otherwise, $G$ can be any quasi-split connected and reductive group over any non-archimedian local field $F$. Most of the statements also hold without the assumption of a quasi-split group, either by adapting the proofs presented here using the setup in Section~\ref{sec:depthnonqs}, or through a formal comparison to the quasi-split case as in \cite[Section~2]{Goertz2015_nonemptiness}.

In this section, we study \enquote{single} affine Deligne-Lusztig varieties $X_x(b)\subseteq X(\mu,b)$ for $x\in \Adm(\mu)$ where $\mu$ is a dominant cocharacter and $[b]\in B(G)$. For this, we will use the Deligne-Lusztig reduction method, which was introduced by Görtz-He \cite[Corollary 2.5.3]{Goertz2010b} in the affine case as well as purity of the Newton stratification, which we review and discuss in Section~\ref{sec:purity}.

\subsection{The Deligne-Lusztig reduction method}

\begin{proposition}\label{prop:DLreduction}
Let $x\in \widetilde W$ and let $a\in \Delta_\af$.
If $\mathrm{char}(F)>0$, then the following two statements hold for any $b\in G(\breve F)$.
\begin{enumerate}[(a)]
\item If $\ell(r_a x r_{\sigma a})=\ell(x)$, then there exists a $\mathbb J_b(F)$-equivariant universal homeomorphism $X_x(b)\rightarrow X_{r_a x r_{\sigma a}}(b)$.
\item If $\ell(r_a x r_{\sigma a})=\ell(x)-2$, then there exists a decomposition $X_x(b)=X_1\sqcup X_2$ such that
\begin{itemize}
\item $X_1$ is open and there exists a $\mathbb J_b(F)$-equivariant morphism $X_1\rightarrow X_{r_a x}(b)$, which is  the composition of a Zariski-locally trivial $\mathbb G_m$-bundle and a universal homeo\-morphism. 
\item $X_2$ is closed and there exists a $\mathbb J_b(F)$-equivariant morphism $X_2\rightarrow X_{r_a x r_{\sigma a}}(b)$, which is the composition of a Zariski-locally trivial $\mathbb A^1$-bundle and a universal homeomorphism. 
\end{itemize}
\end{enumerate}
If $\mathrm{char}(F)=0$, then the analogues of the above statements hold after replacing $\mathbb A^1$ and $\mathbb G_m$ by $\mathbb A^{1,\perf}$ and $\mathbb G_m^{\perf}$, respectively.
\end{proposition}
Using Proposition \ref{prop:DLreduction} (b), we are able to reduce questions about the ADLV $X_x(b)$ to ADLV $X_{y}(b)$ for elements $y\in\widetilde W$ of smaller length. However, for a given element $x\in\widetilde W$, it is not always possible to find a simple affine root $a\in\Delta_{\af}$ with $\ell(r_a x r_{\sigma a}) = \ell(x)-2$. In this case, one can try to apply Proposition \ref{prop:DLreduction} (a) iteratively to replace $x$ by an element $x'$ such that $X_x(b)$ and $X_{x'}(b)$ are universally homeomorphic and such that we can then use part (b) of the proposition for $X_{x'}(b)$. The next proposition, which summarizes results of \cite{He2014b} and \cite{He2014}, shows that this strategy can always be made to work.

\begin{proposition}\label{prop:DLreductionAlgo}
Let $x\in \widetilde W$.
\begin{enumerate}[(a)]
\item There is a sequence
\begin{align*}
    x = x_1, \dotsc, x_m\in\widetilde W
\end{align*}
satisfying the following three conditions, cf.\ \cite[Theorem~A]{He2014b}:
\begin{itemize}
    \item We have $\ell(x_1)\geq \cdots \geq \ell(x_m)$.
    \item For each $i=1,\dotsc,m-1$, there exists a simple affine root $a_i$ with $x_{i+1} = r_{a_i} x_i r_{\sigma a_i}$.
    \item The element $x_m\in\widetilde W$ has minimal length in its $\widetilde W$-$\sigma$-conjugacy class $\{y^{-1} x_m \sigma(y)\mid y\in\widetilde W\}$.
\end{itemize}
\item Assume that $x$ has minimal length in its $\widetilde W$-$\sigma$-conjugacy class. Pick a representative $\dot x\in N_G(T)(\breve F)$ and let $[\dot x]\subseteq G(\breve F)$ denote its $\sigma$-conjugacy class. Then $IxI\subseteq [\dot x]$, cf.\ \cite[Theorem~3.5]{He2014}. We have $\dim X_x(\dot x) = \ell(x) - \langle \nu(\dot x),2\rho\rangle$ and the group $J_{\dot x}(F)$ acts transively on the irreducible components of $X_{x}(\dot x)$, cf.\ \cite[Theorem~4.8]{He2014}.
\end{enumerate}
\end{proposition}

\begin{remark}
    Suppose that $\ell(x) = \ell(r_a x r_{a})$  for an element $x\in\widetilde W$ and a simple affine root $a\in\Delta_{\af}$. Let moreover $\mu$ be a dominant cocharacter. We claim that then $x\in \Adm(\mu)$ if and ony if $r_a x r_a\in \Adm(\mu)$. Assume that the group $G$ is residually split, and our choice of Borel pair is such that $r_a=r_{\sigma(a)}$ for all $a$. Then the above claim implies that the algorithm to compute $\dim X_x(b)$ by combining Propositions \ref{prop:DLreduction} and \ref{prop:DLreductionAlgo} never leaves the admissible set $\Adm(\mu)$, provided the starting element $x$ lies in this set.

    To prove the claim it suffices to prove the implication $x\in \Adm(\mu)\implies r_a x r_a\in\Adm(\mu)$, and we may assume that $x\neq r_a x r_a$. Given $x\in \Adm(\mu)$, we have $x\leq \varepsilon^{u\mu}$ for some $u\in W_0$. We have $r_a x < x < xr_a$ or conversely $r_a x > x > x r_a$. The arguments in either case are very similar, so let us only consider the case $r_a x < x < xr_a$. If $\varepsilon^{u\mu} < \varepsilon^{u\mu} r_a$, we get $xr_a < \varepsilon^{u\mu} r_a$ and hence $r_a x r_a < r_a \varepsilon^{u\mu} r_a = \varepsilon^{\cl(r_a)u\mu}\in \Adm(\mu)$.

    So let us consider the case $\varepsilon^{u\mu} r_a < \varepsilon^{u\mu}$. Then we get $xr_a \leq \varepsilon^{u\mu}$ by \cite[Proposition~2.2.7]{Bjorner2005}. Hence $r_a x r_a < xr_a \leq \varepsilon^{u\mu}\in \Adm(\mu)$. The proof is complete.
\end{remark}
\subsection{The Length 1 Bruhat Cover property}

\begin{lemma}\label{lem:L1BC_def}
    Let $[b]\in B(G)$ be a $\sigma$-conjugacy class. Then the following are equivalent:
    \begin{enumerate}[(i)]
    \item For every $\sigma$-fundamental element $x\in \widetilde W$ with $[\dot x] = [b]$, and every simple affine reflection $s\in S_{\af}$ such that $\ell(sx)<\ell(x)$, we have
    \begin{align*}
        \length([b_{sx,\max}], [b]) = 1.
    \end{align*}
    \item For every element $x\in\widetilde W$ with generic $\sigma$-conjugacy class $[b_{x,\max}] = [b]$, and every lower Bruhat cover $x'\lessdot x$, we have
    \begin{align*}
        \length([b_{x',\max}], [b]) \leq 1.
    \end{align*}
    \end{enumerate}
\end{lemma}

\begin{definition}
    We say that $[b]\in B(G)$ satisfies the \emph{(open) length 1 Bruhat covers property} (L1BC) if it satisfies the equivalent conditions of the lemma. We say that it satisfies the \emph{closed length 1 Bruhat covers property} (CL1BC) if every $[b']\leq [b]$ satisfies (L1BC).
\end{definition}
\begin{proof}[Proof of Lemma \ref{lem:L1BC_def}]
    For (i) $\implies$ (ii), let $x\in\widetilde W$ with $[b_{x,\max}] = [b]$ and $x'\lessdot x$. Using a result of Viehmann\footnote{The result is applicable to our situation without further technical assumptions on $G$, as explained around \cite[Theorem~4.1]{Schremmer2022_newton}.} \cite[Corollary~5.6]{Viehmann2014}, we find a fundamental element $y\leq x$ such that $[b_{y,\max}] = [b_{x,\max}] = [b]$. If $y\leq x'$, we get $[b_{y,\max}]\leq [b_{x',\max}] \leq [b_{x,\max}]$ by Viehmann's result, so we are done.

    Otherwise, pick a reduced word
    \begin{align*}
        x = \tau s_1\cdots s_{\ell(x)},
    \end{align*}
    where $\ell(\tau)=0$ and where $s_1,\dotsc,s_{\ell(x)}$ are simple affine reflections. Since $y\leq x$, we obtain a reduced word for $y$ of the form
    \begin{align*}
        y = \tau s_{i_1} s_{i_2}\cdots s_{i_{\ell(y)}}
    \end{align*}
    with $1\leq i_1<i_2<\cdots <i_{\ell(y)}\leq \ell(x)$. Similarly, since $x'\lessdot x$, we get a reduced word
    \begin{align*}
        x' = \tau s_1\cdots s_{j-1} s_{j+1}\cdots s_{\ell(x)}
    \end{align*}
    for some $1\leq j\leq \ell(x)$. Since $y\not\leq x'$, we see that $j\in \{i_1,\dotsc,i_{\ell(y)}\}$. So let us write $j = i_k$ with $1\leq k\leq \ell(y)$.

    Since $y$ is fundamental, so is the cyclic shift
    \begin{align*}
        y' = \sigma^{-1}(s_{i_k} s_{i_{k+1}}\cdots s_{i_{\ell(y)}})\tau s_{i_1}\cdots s_{i_{k-1}},
    \end{align*}
    and $[b_{y',\max}] = [b_{y,\max}] = [b]$. By assumption (i), the element
    \begin{align*}
        y'' = \sigma^{-1}(s_{i_{k+1}}\cdots s_{i_{\ell(y)}})\tau s_{i_1}\cdots s_{i_{k-1}}
    \end{align*}
    satisfies $\length([b_{y'',\max}],[b])=1$. Observe that $y''\leq \sigma^{-1}(s_{j+1}\cdots s_{\ell(x)}) \ast \tau s_1\cdots s_{j-1}$ by the usual properties of the Demazure product. Using Viehmann's aforementioned result as well as \cite[Lemma~2.3]{He2024_demazure}, we get
    \begin{align*}
        [b_{y'',\max}] &\leq [\sigma^{-1}(b_{s_{j+1}\cdots s_{\ell(x)}) \ast \tau s_1\cdots s_{j-1},\max}] = [b_{\tau s_1\cdots s_{j-1}\ast s_{j+1}\cdots s_{\ell(x)},\max}] = [b_{x',\max}]\\&\leq [b_{x,\max}] = [b].
    \end{align*}
    We conclude $\length([b_{x',\max}], [b])\leq 1$, establishing (ii).

    The implication (ii) $\implies$ (i) is easy. Since $x$ is $\sigma$-fundamental, we get $[\dot x] = [b_{x,\max}] = [b]$. Moreover, $sx\lessdot x$, so that (ii) implies $\length([b_{sx,\max}], [b])\leq 1$. It is clear that $[b_{sx,\max}]\neq [b_{x,\max}]$, since we have $\ell(sx) <\ell(x) = \langle \nu(b),2\rho\rangle$.
\end{proof}

The property (L1BC) is very well-behaved with respect to Weil restrictions of scalars.
\begin{lemma}\label{lem:L1BC_res}
As in Section~\ref{sec:weilRestrict} let $G$ be a restriction of scalars of a reductive group $G'$, let $[b]\in B(G)$ be a $\sigma$-conjugacy class and let $[b']\in B(G')$ be the corresponding $\sigma^d$-conjugacy class. Then $[b]$ satisfies (L1BC) (for $G$) if and only if $[b']$ satisfies (L1BC) (for $G'$).
\end{lemma}
\begin{proof}
    We may identify $\widetilde W_G$ with the $d$-fold direct product $\widetilde W_{G'}\times\cdots\times \widetilde W_{G'}$. Now an element $x = (x_1, \dotsc,x_d)\in \widetilde W_G$ is $\sigma$-fundamental if and only if the following conditions are both satisfied:
    \begin{itemize}
        \item $\ell(x_1\cdots x_n) = \ell(x_1)+\cdots + \ell(x_n)$ and
        \item the element $x_d\cdots x_1\in\widetilde W_{G'}$ is $\sigma^d$-fundamental.
    \end{itemize}
    In this case, $[\dot x] = [b]$ if and only if the $\sigma^d$-conjugacy class of $[x_d\cdots x_1]$ in $B(G')$ is equal to $[b']$.

    First assume that $[b]$ satisfies (L1BC). We verify condition (i) of Lemma~\ref{lem:L1BC_def} for $[b']$. So pick an element $y\in\widetilde W_{G'}$ which is $\sigma^d$-fundamental and satisfies $[\dot y] = [b']$. Let moreover $s'\in\widetilde W_{G'}$ be a simple affine reflection with $s'y<y$ in $\widetilde W_{G'}$. Then $x = (y,1,\dotsc,1)\in\widetilde W_G$ is $\sigma$-fundamental and satisfies that $sx<x$, where $s = (s',1,\dotsc,1)\in\widetilde W_G$ is the corresponding simple affine reflection. Since $[b]$ satisfies (L1BC) by assumption, we get $\length([b_{sx,\max}, [b])=1$. Now observe that the image of $[b_{sx,\max}]$ under the natural isomorphism $B(G)\cong B(G')$ is exactly $[b_{s'y,\max}]$. We get the conclusion.

    Finally assume that $[b']$ satisfies (L1BC). We verify condition (i) of Lemma~\ref{lem:L1BC_def} for $[b']$. So let $x = (x_1,\dotsc,x_d)\in\widetilde W_G$ be $\sigma$-fundamental with $[\dot x] = [b]$, and let $s\in\widetilde W_G$ be a simple affine reflection with $sx<x$. We have to show that $\length([b_{sx,\max}], [b])=1$. This situation does not change if we replace $(x,s)$ by $(\sigma(x),\sigma(s))$. By this, we may and do assume without loss of generality that $s$ has the form $s = (1,\dotsc,1,s')$ for a simple affine reflection $s'\in\widetilde W_{G'}$. Then $s'x_d<x_d$. We see that $s'y<y$, where $y\in\widetilde W_{G'}$ is the length-additive product $y = x_d\cdots x_1$. By the above, $y$ is $\sigma^d$-fundamental with $[\dot y] = [b']$. Since $[b']$ satisfies (L1BC), we get $\length([b_{s'y,\max}], [b'])=1$. Now observe that the image of $[b_{sx,\max}]$ under the natural isomorphism $B(G)\cong B(G')$ is exactly $[b_{s'y,\max}]$. This finishes the proof.
\end{proof}
\subsection{Geometric consequences}
The relevance of the property (CL1BC) for the study of affine Deligne-Lusztig varieties comes due to the following lemma which is analogous to a result of Mili\'cevi\'c-Viehmann. Their (and our) proof uses purity of the Newton stratification, as discussed in Section~\ref{sec:purity}.

\begin{lemma}\label{lem:CL1BC_cordial}
    Let $x\in\widetilde W$ such that $[b_{x,\max}]\in B(G)$ satisfies (CL1BC). Then $x$ satisfies the equivalent conditions of \cite[Theorem~3.16]{Milicevic2020}. More explicitly, let $[b_{x,\min}]\in B(G)_x$ denote the unique minimal element \cite[Theorem~1.1]{Viehmann2021}. Then
    \begin{align*}
        B(G)_x = \{[b]\in B(G)\mid [b_{x,\min}]\leq [b]\leq [b_{x,\max}]\}.
    \end{align*}
    For every $[b]\in B(G)_x$, the ADLV $X_x(b)$ is equidimensional. We have
    \begin{align*}
        \dim X_x(b) - \dim X_x(b_{x,\max}) = \frac 12\left(\langle \nu(b_{x,\max}) - \nu(b),2\rho\rangle + \defect(b_{x,\max}) - \defect(b)\right).
    \end{align*}
\end{lemma}
\begin{proof}
    We use induction on $\ell(x)$. If $x$ has minimal length in its $\sigma$-conjugacy class in $\widetilde W$, there is nothing to prove as $B(G)_x = \{[b_{x,\max}]\}$. Assume now that this is not the case

    By Proposition~\ref{prop:DLreductionAlgo}, we find a length decreasing sequence of elements
\begin{align*}
x = x_1, x_2,\dotsc,x_r \in \widetilde W
\end{align*}
such that each $x_i$ with $i>1$ has the form $x_i = s_i x_{i-1} \sigma(s_i)$ for a simple affine reflection $s_i\in S_{\af}$, and such that $\ell(x_1) = \ell(x_2)=\cdots = \ell(x_{r-1}) > \ell(x_r)$. The generic $\sigma$-conjugacy classes of $x_1,\dotsc,x_{r-1}$ as well as $x_r' := s_r x_{r-1}$ all agree. We have $x_r \lessdot x_r'$, hence $\length([b_{x_r,\max}], [b_{x_r',\max}])\leq 1$ by assumption on $[b_{x,\max}] = [b_{x_r',\max}]$. By the inductive assumption, the lemma's statement holds for both $x_r$ and $x_r'$. Now using Proposition~\ref{prop:DLreduction}, we find
\begin{align*}
    \dim X_x(b) = 1+\max(\dim X_{x_r}(b), \dim X_{x_r'}(b))
\end{align*}
for all $[b]\in B(G)$, where we set $\dim \emptyset=-\infty$.
By Theorem \ref{thm:purity}, the Newton stratification satisfies topological strong purity. Thus by \cite[Lemma 5.12]{Viehmann2015} (compare also Remark \ref{rem:viehmann2015lemma} (a)), it is enough to show that the codimension of $IxI\cap [b]$ inside $IxI$, for all $[b]\in B(G)_x$, is at least $\length([b], [b_{x,\max}])$.

The Deligne-Lusztig reduction splits $IxI$ into an open part $\mathcal U$ and a closed part $\mathcal C$ of codimension $1$. If $[b]\cap \mathcal U\neq\emptyset$, then the codimension of this intersection inside $\mathcal U$ is equal to $\codim([b]\cap Ix_r'I~\subseteq~Ix_r'I)$. By the inductive assumption, this is at least $\length([b], [b_{x,\max}])$. Similarly, if $[b]\cap \mathcal C\neq\emptyset$, then the codimension of this intersection inside $\mathcal C$ is equal to $\codim([b]\cap Ix_rI~\subseteq Ix_rI)$. By the inductive assumption, this codimension is at least $\length([b], [b_{x_r,\max}])$. Thus the codimension of $[b]\cap \mathcal C$ inside $IxI$ is
\begin{align*}
    &\codim([b]\cap \mathcal C~\subseteq \mathcal C) + \codim(\mathcal C\subseteq IxI) = \codim([b]\cap Ix_rI~\subseteq Ix_rI) + 1 \\\geq& \length([b], [b_{x_r,\max}])+1 \underset{\text{assump.}}\geq\length([b], [b_{x_r,\max}]) + \length([b_{x_r,\max}], [b_{x,\max}]) \\=& \length([b], [b_{x,\max}]).
\end{align*}
This completes the induction and the proof.
\end{proof}

\subsection{ (CL1BC) and the depth is less than 2}
The main result of this section is the following:

\begin{theorem}\label{thm:depth2L1BC}
    Let $G$ be an unramified restriction of scalars of $\PGL_{n+1}$, and let $\mu$ be a dominant cocharacter with $\depth(G,\mu)<2$. Let $[b]\in B(G,\mu)$ such that $(G,\mu,[b])$ is Hodge-Newton indecomposable. Then $[b]$ satisfies (CL1BC).
\end{theorem}
For the proof, we need the following preparation.

\begin{lemma}\label{lem:genericClassDistance}
Let $x = w\varepsilon^\mu\in\widetilde W$, and $[b_{x,\max}]\in B(G)$ its generic $\sigma$-conjugacy class. Then
\begin{align*}
\ell(x) - \langle \lambda(b_{x,\max}),2\rho\rangle = \min_{v\in \LP(x)} d(v\Rightarrow \sigma(wv)).
\end{align*}
Denote this quantity by $d$. If $[b_{x,\max}]$ satisfies (CL1BC) and $[b]\in B(G)_x$, then
\begin{align*}
\dim X_x(b) = \frac 12\left(\ell(x)+d - \langle \nu(b),2\rho\rangle - \defect(b)\right).
\end{align*}
\end{lemma}
\begin{proof}
Apply \cite[Theorem~4.2]{Schremmer2022_newton} and \cite[Corollary~4.5]{Schremmer2022_newton} to get the first identity. The dimension formula follows from Lemma~\ref{lem:CL1BC_cordial} together with the formulas
\begin{align*}
	&\langle \nu(b_{x,\max}),2\rho\rangle - \defect(b_{x,\max}) = \langle \lambda(b_{x,\max}),2\rho\rangle
	\\&\dim X_x(b_{x,\max}) = \ell(x)-\langle \nu(b_{x,\max}),2\rho\rangle
\end{align*}
from \cite[Proposition~3.9]{Schremmer2022_newton} resp.\ \cite[Theorem~2.23]{He2015}.
\end{proof}

\begin{proof}[Proof of Theorem~\ref{thm:depth2L1BC}]
It is clear from the definition that we only have to verify that $[b]$ satisfies (L1BC). This is trivially satisfied for basic classes, so we may exclude the fully Hodge-Newton decomposable cases.
Using Lemma~\ref{lem:L1BC_res}, we may and do assume that $G = \PGL_{n+1}$. 

Following Theorem~\ref{prop:classification}, we have to study the following cases (up to a finite list of further cases, which we computer verified, compare Appendix \ref{app:a4}):
\begin{align*}
\mu = 2\omega_1^\vee+\omega_n^\vee, \omega_2^\vee + \omega_n^\vee,2\omega_1^\vee,\omega_2^\vee.
\end{align*}
Since $\Adm(\omega_2^\vee+\omega_n^\vee)\subseteq \Adm(2\omega_1^\vee+\omega_n^\vee)$, and similarly for the other two cases, it suffices to study the cases of $\mu$ being $2\omega_1^\vee+\omega_n^\vee$ or $2\omega_1^\vee$.

If $\mu = 2\omega_1^\vee+\omega_n^\vee$, it follows that $\nu(b_{x,\max})\leq \omega_1^\vee$ from Hodge-Newton indecomposability. Let $x\in\widetilde W$ be fundamental with $[\dot x] = [b]$. It follows that $x\in \Adm(\omega_1^\vee)$, so $x$ is Lubin-Tate admissible. Now any $\dot x\lessdot x$ being Lubin-Tate admissible as well, it follows that $x'$ is fundamental, and the claim follows.

Consider now the case $\mu = 2\omega_1^\vee$. For this proof, we define the function
\begin{align*}
d : \widetilde W\rightarrow\mathbb Z_{\geq 0},
\end{align*}
sending $x$ to the number described in Lemma~\ref{lem:genericClassDistance}.

Let $x\in\widetilde W$ be fundamental with $[\dot x] = [b]$, and $x'\lessdot x$ be a lower Bruhat cover such that $x' x^{-1}$ is a simple affine reflection in $\widetilde W$. Choose $v\in \LP(x)$ such that $d(v\Rightarrow \cl(x)v)=d(x)$. Consider the element $\tilde x = x\cdot \varepsilon^{v\omega_n^\vee}\in\widetilde W$. Then $\LP(\tilde x) \subseteq \LP(x)$, and similarly for $\tilde x' = x'\cdot \varepsilon^{v\omega_n^\vee}$ and $x'$.
These inclusions imply $d(x)\leq d(\tilde x)$ and $d(x')\leq d(\tilde x')$. From $v\in \LP(\tilde x)$, we get $d(x)=d(\tilde x)$.

Since $\tilde x\in\Adm(2\omega_1^\vee+\omega_n^\vee)$ and $\tilde x'\lessdot \tilde x$, we get $\length([b_{x',\max}], [b_{x,\max}])= 1$ (one easily verifies $\lambda(b_{\tilde x,\max}) = \lambda(b_{x,\max})+\omega_n^\vee$ by choice of $v$, implying Hodge-Newton indecomposability). Equivalently, this means $d(\tilde x')= d(\tilde x)$. We conclude
$d(x') \leq d(\tilde x') = d(\tilde x) = d(x).$ Therefore, \begin{align*}
    \length([b_{x',\max}], [b_{x,\max}]) &= \langle \lambda(b_{x,\max})-\lambda(b_{x',\max}),2\rho\rangle \\&= d(x')-\ell(x')-d(x)+\ell(x) = d(x')-d(x)+1\leq 1.
\end{align*}
Since $x$ is fundamental, we get $\length([b_{x',\max}], [b_{x,\max}])\geq 1$ and the conclusion follows.
\end{proof}
\begin{proof}[Proof of Theorem~\ref{thm:introKR}]
	Using the Hodge-Newton decomposition together with Theorem~\ref{lemdepthprop}, we can reduce to the case where $[b_{x,\max}]\in B(G,\mu)$ is Hodge-Newton indecomposable. Then it satisfies (CL1BC) by Theorem~\ref{thm:depth2L1BC}. We get the claims from Lemmas \ref{lem:CL1BC_cordial} and \ref{lem:genericClassDistance}.
\end{proof}

\section{Irreducible Components}\label{sec:5}

\subsection{Geometry of the Newton stratification}\label{sec:purity}

We provide some statements related to purity and the foliation structure of  Newton strata. They are generalizations to our context of results that have been successfully applied to ADLVs in the function field case and to Rapoport-Zink spaces in the past, compare \cite{Viehmann2015} for a general overview. However, they are not yet available in the generality that we need for our purpose. We state them in greater generality than what we strictly need in order to be of use also for future work on the Newton stratification.

\begin{definition}
Let $G$ be a reductive group over $F$. Let $S$ be a scheme over $\mathbb F_p$ whose underlying topological space is connected and locally Noetherian. For every $[b]\in B(G)$ let $S_{[b]}$ be a subscheme such that $S$ is the disjoint union of the $S_{[b]}$ and for every $[b]$, the subscheme $\bigcup_{[b']\leq[b]}S_{[b']}$ is closed. Then we say that this decomposition of $S$ satisfies topological strong purity if the following condition holds. Let $[b]\in B(G)$ such that the corresponding stratum is non-empty. Let $[b']< [b]$ such that there is no $[b'']\in B(G)$ with $[b']<[b'']< [b]$. Let $X$ be the closure in $S$ of an irreducible component of $S_{[b]}$. Let $X^{[b']}$ be the complement in $X$ of $X\cap \bigcup_{[b'']\leq[b' ]}S_{[b'']}$. Then for every such $X$ and $[b']$ we have that $X^{[b']}$ is empty or pure of codimension at most 1.
\end{definition}

\begin{theorem}\label{thm:purity}
Let $F$ be a local field and let $G$ be a reductive group over $F$. Let $S$ be a scheme, resp.~a perfect scheme over $\mathbb F_p$ if $F$ is of equal resp.~of mixed characteristic. Assume that Zariski locally, $S$ is isomorphic to the spectrum of (the perfection of) a Noetherian ring. Let $g\in LG(S)$, and for every algebraically closed field $k$ of characteristic $p$ and every $[b]\in B(G)$ let  
$$S_{[b]}(k)=\{x\in S(k)\mid g_x\in [b]\}.$$
\begin{enumerate}[(a)]
\item\label{thm:purityass1} There is a locally closed and reduced resp.~perfect subscheme $S_{[b]}$ of $S$ whose set of $k$-valued points for every $k$ agrees with $S_{[b]}(k)$.
\item\label{thm:purityass2} The induced decomposition of $S$ satisfies topological strong purity. 
\end{enumerate}
\end{theorem}
\begin{proof}
\eqref{thm:purityass1} follows from the fact that for every $[b]\in B(G)$, the union $\coprod_{[b']\leq [b]}S_{[b']}$ is closed by \cite{Rapoport1996}. For \eqref{thm:purityass2}, first notice that the statements are local in $S$, thus we may assume that $S$ is (the perfection of) a Noetherian affine scheme. By \cite[Prop.~2.2]{Hamacher2017}, it is enough to prove the statement for $G=\GL_n$. For Noetherian $S$, this is shown by Yang \cite[Thm.~1.1]{Yang2011} in the arithmetic case and by Viehmann \cite[Thm.~1]{Viehmann_grothconj} in the function field case. The proof of  \cite[Thm.~1]{Viehmann_grothconj} can easily be adapted to also prove the assertion in the case that $S=\rm{ Spec}~R^{\rm{perf}}$ for a Noetherian $\mathbb F_p$-algebra $R$, using that for truncated Witt vectors, we have that every $g\in W_n(R^{\rm{perf}})$ is indeed contained in $W_n(R')$ for some finite and hence Noetherian extension of $R$.
\end{proof}

\begin{remark}\label{rem:viehmann2015lemma}
\begin{enumerate}[(a)]

\item The main application of topological strong purity (both for us and in other contexts) is summarized in \cite[Lemma 5.12]{Viehmann2015}. It considers the case that $S$ is irreducible and that $S=\coprod S_{[b]}$ is a decomposition satisfying topological strong purity. Let $[b_{\eta}]\in B(G)$ be such that $S_{[b_{\eta}]}$ contains the generic point of $S$. Let $[b_0]$ be such that $S_{[b_0]}$ is non-empty and that the codimension of every irreducible component of $S_{[b_0]}$ in $S$ is at least equal to $\length([b_0],[b_{\eta}])$. Let $[b']\in B(G)$ with $[b_0]\leq[b']\leq [b_{\eta}]$. Then $S_{[b']}$ is non-empty, its closure contains $S_{[b_0]}$, and every irreducible component of $S_{[b_0]}$ is of codimension equal to $\length([b_0],[b_{\eta}])$ in $S$.

\item In slightly greater generality, one still has the following application (compare \cite[Cor.~7.7(a)]{Hartl2011}). Let $S$ be a scheme with a Newton stratification $S=\bigcup_{[b]}S_{[b]}$. Fix an irreducible component $Y$ of some $S_{[b_0]}$. Then a chain $[b_0]<[b_1]<\dotsm<[b_l]$ of elements of $B(G)$ is called realizable in $S$ at $Y$ if there are irreducible subschemes $S_i\subset S_{[b_i]}$ with $S_0=Y$ and $S_{i-1}\subseteq \overline{S_i}$ for all $i$. We call $l$ the length of the above chain. Then topological strong purity of the stratification implies that the codimension of $Y$ in $S$ is equal to the maximal length of a realizable chain in $S$ at $Y$.

\end{enumerate}
\end{remark}

To relate (co)dimensions of Newton strata $S_{[b]}$ to dimensions of affine Deligne-Lusztig varieties, let $[b]\in B(G)$ and let $K$ be a $\sigma$-stable parahoric subgroup. For the rest of this section let $S$ be a  finite union of $K$-double cosets in $LG$. Let $X_S(b)$ be the associated affine Deligne-Lusztig variety. It is the closed, reduced resp.~perfect subscheme of $LG/K$ with $\overline {\mathbb F}_p$-valued points 
$$X_S(b)(\overline {\mathbb F}_p)=\{g\in LG/K(\overline {\mathbb F}_p)\mid g^{-1}b\sigma(g)\in S(\overline {\mathbb F}_p)\}.$$
Of course, we are most interested in the cases that $K$ is hyperspecial and $S=\overline{K\mu K}$ for some $\mu$ or that $K=I$ is an Iwahori subgroup and $S=\overline{I xI}$ for some $x\in \widetilde W$ or $S=\coprod_{x\in \Adm(\mu)}IxI$. Let $S_{[b]}\subseteq S$ be the Newton stratum for $[b]$. We use the usual convention for dimensions of admissible subschemes of $LG$ (that are otherwise infinite-dimensional). For fixed parahoric $K$, they are normalized in such a way that a $K$-invariant subscheme $X=XK\subseteq LG$ has dimension equal to the dimension of $X/K\subseteq LG/K$ in the associated affine flag variety. In particular, for $K=I$, we have $\dim IxI=\ell(x)$ and in the context above, $\dim S_{[b]}=\dim S-\codim S_{[b]}$ if $S$ is irreducible.

In the function field case, we denote by $S_{h}^{\wedge}$ or $S_{[b],h}^{\wedge}$ (for $g\in S_{[b]}$) and $X_S(b)_g^{\wedge}$ (for $g\in X_S(b)$) the completions in $h$ resp.~$g$. In the arithmetic case this denotes the perfection of the completion of a deperfection as introduced by Takaya in \cite[Section 1.2]{Takaya2022}. From now on, we refer to Takaya's construction as completion (of perfect schemes).

\begin{lemma}\label{lem:inducedperfcompl}
Let $X,Z$ be perfect $\overline{\mathbb F}_q$-schemes locally of perfect finite type, let $\alpha:X\rightarrow Z$ be a morphism and let $g\in X$. Then $\alpha$ induces a morphism $$\hat\alpha: X_g^{\wedge}\rightarrow Z_{\alpha(g)}^{\wedge}$$ of the completions of the perfect schemes at $g$ resp.~$\alpha(g)$.   
\end{lemma}
\begin{proof}
We may assume that $X,Z$ are affine, that is $X=\Spec B^{\perf}$ and $Z=\Spec A^{\perf}$ for certain finitely generated $\overline{\mathbb F}_q$-algebras $A$ and $B$. Thus $\alpha$ corresponds to a morphism of algebras $f:A^{\perf}\rightarrow B^{\perf}$. Possibly replacing $B$ by its inverse image under a suitable finite power of Frobenius in $B^{\perf}$ (containing all images of a finite system of generators of $A$), we may assume that $f$ restricts to a morphism $f:A\rightarrow B$. Recall that perfection is an isomorphism of the underlying topological spaces. Let $\mathfrak p\subset B$ be the prime ideal corresponding to $g$ and let $\mathfrak p'=f^{-1}(\mathfrak p)\subset A$. Then $f$ induces morphisms $A_{\mathfrak p'}^{\wedge}\rightarrow B_{\mathfrak p}^{\wedge}$ and $(A_{\mathfrak p'}^{\wedge})^{\perf}\rightarrow (B_{\mathfrak p}^{\wedge})^{\perf}$. On spectra, the latter yields the desired morphism $\hat \alpha$.
\end{proof}

We also need the following construction which is a variant of the description of the universal deformation of a bounded local $G$-shtuka of \cite[Thm.~5.6]{Hartl2011}. For $S$ as above and $b\in S(\overline{\mathbb F}_q)$ let $Y=\Spec R$ resp.~$Y=\Spec R^{\perf}$ be the spectrum of the complete (perfect) local ring of  $K\backslash S$ at $Kb$. Then $R$ is Noetherian local with residue field $\overline{\mathbb F}_q$. We choose a section $\tilde b:Y\rightarrow S$ of the projection morphism $S\rightarrow K\backslash S$ (compare the first part of the proof of \cite[Thm.~5.6]{Hartl2011}) such that $\tilde b$ reduces to $b$ over the residue field of $R$. In the mixed characteristic case, note that $K\backslash S$ is a perfect $\overline{\mathbb F}_q$-scheme locally of perfect finite type. Then a similar argument as in the proof of Lemma \ref{lem:inducedperfcompl} shows that $\tilde b$ can be chosen to come from a point $\tilde b:\Spec (\sigma^{-n}R)\rightarrow S$ for a suitably large $n$. Replacing $R$ by this inverse image under $\sigma^n$, we may assume that $\tilde b$ is an $R$-valued point of $S$ also in the arithmetic case.

\begin{proposition}\label{prop:codimSandUnifDef}
In the above context, the codimensions of $Y_{[b]}\subseteq Y$ for the Newton stratification induced by $\tilde b$ and of $S_{[b],b}^{\wedge}\subseteq S_b^{\wedge}$ coincide. 
\end{proposition}

\begin{proof}
We use Remark \ref{rem:viehmann2015lemma}(b) which describes the codimensions as maximal lengths of chains of Newton strata specializing into another. Since $Y$ is a subscheme of $S_b^{\wedge}$, this implies immediately that $\rm{codim}(Y_{[b]}\subseteq Y)\leq \rm{codim}(S_{[b],b}^{\wedge}\subseteq S_b^{\wedge})$. For the opposite estimate we consider a chain $[b]=[b_0]<[b_1]<\dotsm<[b_m]$ that is realizable in $S_b^{\wedge}$. Thus there is a (perfect) local scheme $Z=\Spec R'$ resp.~$Z=\Spec (R')^{\perf}$ with $R'$ of finite type and a morphism $\beta:Z\rightarrow S_b^{\wedge}$ such that the chain is already realized over $Z$. By \cite[Cor.~6.32]{VW_2025} (whose proof also shows the analogous statement in equal characteristic) there is a $c$ such that the Newton stratification does not change if we modify the $Z$-valued point of $LG$ within its $K_c$-right coset. Here, $K_c\subset K$ is the kernel of the reduction modulo $\varepsilon^c$. Since $S/K_c$ is (perfectly) of finite type, the same argument as for $\tilde b$ above shows that we may assume that $\beta$ is in fact an $R'$-valued point. Then the same inductive argument as for the claim in the proof of \cite[Thm.~5.6]{Hartl2011} shows that $\beta$ is $K(W(R'))$-$\sigma$-conjugate to a point in $Y$. Since $K$-$\sigma$-conjugation does not change the Newton stratification, the above chain of length $m$ is then also realized in $Y$, which proves the remaining estimate.  
\end{proof}

\begin{theorem}\label{thmdimfoliat}
In the above context, let $h\in [b]\cap S$ and fix some $\dot g\in LG$, a lift of an element $g\in X_S(b)$ with $\dot g^{-1}b\sigma(\dot g)=h$. Then
$$\dim S_{[b],h}^{\wedge}=\dim X_S(b)_g^{\wedge}+\langle \nu_{[b]},2\rho\rangle.$$
\end{theorem}

\begin{proof}
This assertion has been shown under additional assumptions for many particular cases. In the function field case for split $G$, hyperspecial $K$ and $S=\overline{K\mu K}$ this is \cite[Cor.~6.8 to Thm.~6.5]{HartlViehmann_foliat}, which was inspired by earlier ideas for Newton strata in the Siegel moduli space due to Oort. It is generalized to unramified $G$, parahoric $K$, but $S$ still the closure of a double coset of a hyperspecial subgroup in \cite{Viehmann_Wu} and to the Iwahori case and $S=\overline{IxI}$ in \cite{Milicevic2020}. Each of these is shown by some adaptation of the proof of \cite[Thm.~6.5]{HartlViehmann_foliat} where a product structure up to finite morphism on the universal deformation of a bounded local $G$-shtuka is established. All of these adaptations together also show the above assertion in the function field case.

 For the arithmetic case, the situation is again a bit more subtle due to the fact that perfection and deformation theory do not interact well. For hyperspecial $K$ and $S=\overline{K\mu K}$, Takaya shows in \cite[4.4]{Takaya2022} that 
 \begin{equation}\label{eq:foliat_dim_1}\dim Y_{[b]}\leq \dim X_S(b)_g^{\wedge}+\langle \nu_{[b]},2\rho\rangle
 \end{equation} 
for a specific choice of $Y$ as above, but corresponding to $h=\dot g^{-1}b\sigma(\dot g)$. For this, he uses a modification of the proof of \cite[Thm.~6.5]{HartlViehmann_foliat} and his technique of deperfection. The necessary minor modifications made, Takaya's proof also shows \eqref{eq:foliat_dim_1} in the generality claimed in our theorem. By Proposition \ref{prop:codimSandUnifDef}, this proves $$\dim S_{[b],h}^{\wedge}\leq\dim X_S(b)_g^{\wedge}+\langle \nu_{[b]},2\rho\rangle.$$

 It remains to prove the opposite estimate in the arithmetic case. Let $g\in X_S(b)(k)$ for any algebraically closed field $k$ of characteristic $p$. Let $x\in \widetilde W$ be a straight element such that any representative in $LG$ lies in $[b]$. We may assume that $b$ is such a representative of $x$, denoted $x$ from now on. Let $M'\supseteq T$ be the smallest Levi subgroup of $G$ containing $x$ and let $M\supseteq M'$ be the centralizer of the $M'\cap B$-dominant Newton point $\nu_x$ of $x$. Let $P$ be the (semi-standard) parabolic subgroup associated with the rational cocharacter $\nu_x$. Then $P=MN$ where $N$ is the unipotent radical of $P$ and $M$ its Levi factor containing $T$. Let $\bar N$ be the unipotent radical of the opposite parabolic. Let $I_{N}=I\cap N$ and similarly for $M$ and $\bar N$. Since $x$ is straight, it is also $P$-fundamental. In particular, for $\phi_x:LG\rightarrow LG$ with $g\mapsto \sigma(xgx^{-1})$ we have $\phi_x^{-1}(I_{\bar N})\subseteq I_{\bar N}$ and $\phi_x^{-1}$ is topologically nilpotent on $I_{\bar N}$. For $n\in \mathbb N$ that will be chosen below, we define $I_{x,n}= \phi_x^{-n}(I_{\bar N})/\phi_x^{-n-1}(I_{\bar N})$, a subquotient of $I_{\bar N}$ which is (independently of $n$) isomorphic to the perfection of $\mathbb A^{\langle \nu_b,2\rho\rangle}$. Let $I_{x,n}^{\wedge}$ be the completion at the identity of the perfect scheme $I_{x,n}$. Let $Y$ and $Y_{[h]}$ be as in the construction before Proposition \ref{prop:codimSandUnifDef}, but for the element $h=\dot g^{-1}x\sigma(\dot g)\in [b]$. Then $\dim S_{[x],h}^{\wedge}=\dim Y_{[h]}$. It remains to show that $\dim Y_{[h]}\geq \dim X_S(x)_{g}^{\wedge}\hat \times I_{x,n}^{\wedge}.$
We first want to construct a morphism $$\hat\alpha: X_S(x)_{g}^{\wedge}\hat \times I_{x,n}^{\wedge}=:\widehat T\rightarrow Y_{[h]}$$ for suitable $n$, as in Step 4 of \cite[Proof of Thm.~6.5]{HartlViehmann_foliat}. We apply \cite[Cor. 2.15]{Takaya2022} to the universal element of $X_S(x)(\widehat T)$ and obtain a representative $\tilde g\in LG(\widehat T)$. We may assume that its reduction to the closed point $\Spec k\hookrightarrow \widehat T$ is equal to $\dot g$. Similarly, we choose a section $I_{x,n}\rightarrow \phi_x^{-n}(I_{\bar N})$ and let $\delta\in \phi_x^{-n}(I_{\bar N})(\widehat T)$ be the image of the universal element. Since $X_S(x)$ and $I_{x,n}$ are locally of finite type, $\widehat T=\Spec R^{\perf}$ where $R$ is the completion of a finitely generated $k$-algebra and $(\tilde g,\delta)$ can be chosen to be defined over $\sigma^{-m}(R)$ for some finite $m$. Replacing $R$ by this inverse image in $R^{\perf}$, we may assume that $(\tilde g,\delta)$ is even defined over $R$. Consider ${\tilde g}^{-1}x\delta \sigma(\tilde g)\in LG(\widehat T)$. At its closed point, it reduces to $\dot  g^{-1}x\sigma(\dot g)=h$, the closed point of $S_{[x],h}^{\wedge}$. Since $\tilde g$ is bounded and $\phi_x^{-1}$ topologically nilpotent on $I_{\bar N}$, we can choose $n$ so large that $\sigma(\tilde g)^{-1}\phi_x^{-n}(I_{\bar N})\sigma(\tilde g)\subseteq I$. Then $\tilde g^{-1}x\delta \sigma(\tilde g)=\left( \tilde g^{-1}x\sigma(\tilde g)\right)\left(\sigma(\tilde g)^{-1}\delta \sigma(\tilde g)\right)$ factors through $S$. By Lemma \ref{lem:inducedperfcompl} it thus defines a morphism 
$$\hat\beta:X_S(x)_{g}^{\wedge}\hat \times I_{x,n}^{\wedge}=:\widehat T\rightarrow S_{[x],h}^{\wedge},$$ 
more precisely defined by some $\hat\beta: \Spec R\rightarrow S_{[x],h}^{\wedge}$. 
The same iterative argument as in the proof of Proposition \ref{prop:codimSandUnifDef} shows that we can $K(R)$-$\sigma$-conjugate this $R$-valued point $\hat \beta$ into an $R$-valued point $\hat \alpha$ of $Y_{[h]}$ as desired.

By \cite[Prop. 1.17]{Takaya2022}, it is enough to prove that $\hat\alpha:\widehat T\rightarrow Y_{[h]}$ is adic. We roughly follow Step 7 of \cite[Proof of Thm.~6.5]{HartlViehmann_foliat}. Let $\mathcal F_0$ be the fiber of $\hat\alpha$ of the closed point of $Y_{[h]}$ in the following sense. If we write $X_S(x)_{g}^{\wedge}\hat \times I_{x,n}^{\wedge}=\Spec B$ and $Y_{[h]}=\Spec A$ for local adic perfect rings, then $\mathcal F_0$ is the spectrum of the quotient of $B$ by the closure of the ideal generated by the image of $A^{\circ\circ}$ in $B$. By  \cite[Lemma 1.18]{Takaya2022}, the map is adic if and only if $\mathcal F_0$ is the closed point of $X_S(x)_{g}^{\wedge}\hat \times I_{x,n}^{\wedge}$. Assume that $\mathcal F_0$ has another geometric point $y$, and let $(\tilde g_y,\delta_y)$ be the reduction to that point. Then the argument for \enquote{$h=h'$} in Step 7 of \cite[Proof of Thm.~6.5]{HartlViehmann_foliat} shows that $\tilde g_y$ is equal to $\dot g$. To show that $\delta$ is also constant over $\mathcal F_0$, one uses that the construction of $(\tilde g,\delta)$ takes place over $R$ instead of $R^{\perf}$, in particular over a ring whose maximal ideal is topologically nilpotent. Then one can employ the same argument as for the corresponding part of the proof of Step 7 of \cite[Proof of Thm.~6.5]{HartlViehmann_foliat}. Altogether, this shows that $\hat\alpha$ is adic, hence by \cite[Prop. 1.17]{Takaya2022} we obtain the desired estimate for the dimensions.
\end{proof}

\subsection{The ING property}

For this section, $G$ is again any connected reductive group over $F$. We fix a dominant cocharacter $\mu \in X_\ast(T)_{\Gamma_0}$ and a $\sigma$-stable parahoric subgroup $K\subseteq G(\breve F)$, containing a fixed $\sigma$-stable Iwahori subgroup $I\subseteq G(\breve F)$. Write $\Delta_K\subseteq \Delta_\af$ for the $\sigma$-stable set of affine roots corresponding to $K$. There is a unique largest $[b]\in B(G,\mu)$ such that $(G,\mu,[b])$ is Hodge-Newton indecomposable,  which we denote by $[b_{\mi}]$, cf.\ \cite[Cor.\ 7.6]{Viehmann2024} or \cite[Corollary~4.3]{He2022}. In the quasi-split case, this maximal indecomposable class is determined as follows.
\begin{lemma}\label{lem:max_indec_lambda}
    Let $G$ be quasi-split and $\mu\in X_\ast(T)_{\Gamma_0}$ be a dominant cocharacter that is not central on any $\sigma$-connected component of $\Delta$. Then the $\lambda$-invariant of $[b_{\mi}]$ is given by
    \begin{align*}
        \lambda(b_{\mi}) = \mu - \sum_{\alpha\in \Delta/\sigma} \alpha^\vee \in X_\ast(T)_\Gamma,
    \end{align*}
    i.e.\ we subtract one representative of each Frobenius orbit of simple coroots from $\mu$.
\end{lemma}
\begin{proof}
    Define $\lambda := \mu - \sum_{\alpha\in \Delta/\sigma} \alpha^\vee \in X_\ast(T)_\Gamma$. For any $\sigma$-conjugacy class $[b]\in B(G,\mu)$ with $\lambda(b)>\lambda$, we find a $\sigma$-stable subset $\emptyset \neq J\subseteq \Delta$ with $\lambda(b)\equiv \mu \pmod J$ in $X_\ast(T)_{\Gamma}$. Then $(G,\mu,[b])$ is Hodge-Newton decomposable. Hence $\lambda(b_{\mi})\leq \lambda$.

    We get \begin{align*}
        \langle \lambda(b_{\mi}),2\rho\rangle = \langle \mu,2\rho\rangle - 2\#(\Delta/m)= \langle \lambda,2\rho\rangle
    \end{align*} from \cite[Corollary~4.3]{He2022}.
\end{proof}

\begin{remark}
We have a decomposition of $K\Adm(\mu)K$ into orbits under the action of $K\times K_1$ where $K$ acts by $\sigma$-conjugation and where its pro-unipotent radical $K_1$ acts by left multiplication. The resulting locally closed strata are the EKOR strata of \cite{He2017}, so we also refer to them as EKOR strata. They are parametrized by the set $\Adm(\mu)^K$. The closure of an EKOR stratum for some $x$ consists of the union of EKOR strata for $x'\in\Adm(\mu)^K$ such that there is a $w\in W_K$ with $wx'\sigma(w)^{-1}\leq x$ in the Bruhat order (see \cite{He2017}, Thm. 6.15). We denote this partial order by $x'\leq_K x$.
\end{remark}
In this section, we study the Newton stratification of the set $K\Adm(\mu) K\subseteq G(\breve F)$. Inside this set, there is the Hodge-Newton indecomposable locus, which (famously) is a union of both Newton strata $[b]\cap (K\Adm(\mu) K)$ for $[b]\in B(G,\mu)$ Hodge-Newton indecomposable, and also a union of EKOR-strata $K\cdot_\sigma (x K_1)$ for certain $x\in \Adm(\mu)^K$, cf.\ \cite[Theorem~4.17]{Goertz2019}.

\begin{lemma}\label{lem:defing}
The following are equivalent:
\begin{enumerate}[(a)]
\item The closure of $[b_{\mi}]\cap K\Adm(\mu)K$ inside $K\Adm(\mu)K$ is given by
\begin{align*}
\overline{[b_{\mi}]\cap K\Adm(\mu)K} = \bigcup_{[b]} [b]\cap K\Adm(\mu) K,
\end{align*}
with the union taken over all $\sigma$-conjugacy classes $[b]\in B(G,\mu)$ such that $(G,\mu,[b])$ is Hodge-Newton indecomposable.
\item Letting
\begin{align*}
A := \{x\in \Adm(\mu)^K\mid (G,\mu,[b_{x,\max}])\text{ HN-indecomposable}\},
\end{align*}
and letting $x\in A$ be maximal with respect to $\leq_K$ inside $A$, we have $[b_{x,\max}] = [b_{\mi}]$.
\end{enumerate}
\end{lemma}

\begin{proof}
By \cite[Theorem~4.11]{Goertz2019}, the union of all Hodge-Newton decomposable Newton strata is a union of EKOR strata. It follows that also the union of all Hodge-Newton indecomposable Newton stata is such a union, which is thus indexed by the set $A$ above. Since by definition EKOR strata are irreducible, Condition (a) is equivalent to the condition that all generic EKOR strata in this union have generic $\sigma$-conjugacy class $[b_{\mi}]$. By the description of the closure relations above, the generic EKOR strata are indexed precisely by the $\leq_K$-maximal elements inside $A$.
\end{proof}
\begin{definition}\label{def:ing}
If the conditions of the above lemma are satisfied, we say that $(G,\mu,K)$ satisfies the \emph{indecomposable Newton genericity condition}, or (ING) for short.
\end{definition}
\begin{example}
If $(G,\mu)$ is fully Hodge-Newton decomposable, then $[b_{\mi}]$ is basic so (ING) is trivially satisfied for all $K$. 

If $(G,\mu)$ is arbitrary and $K$ is hyperspecial, then $K\Adm(\mu)K=\overline{K\mu K}$. In this case,  $$\overline{[b]\cap K\mu K}=\overline{\bigcup_{[b']\leq [b]} [b']\cap K\mu K}=\bigcup_{[b']\leq [b]} [b']\cap \overline{K\mu K}.$$ Indeed, in the function field case this is \cite{ViehmannICM} Thm 5.4. The arithmetic case can be shown along the same lines, using dimension estimates by Takaya \cite{Takaya2022}. The details of this case are also worked out in an ongoing master thesis in the group of the second author. In particular, (ING) is satisfied in all hyperspecial cases.
\end{example}

\begin{lemma}\label{lem:ING_level_change}
Let $(G,\mu,K)$ satisfy (ING) and let $K'\supseteq K$ be a $\sigma$-stable parahoric. Then $(G,\mu,K')$ satisfies (ING).
\end{lemma}
\begin{proof}
Let 
\begin{align*}
&A := \{x\in \Adm(\mu)^K\mid (G,\mu,[b_{x,\max}])\text{ HN-indecomposable}\},
\\&A' := \{x\in \Adm(\mu)^{K'}\mid (G,\mu,[b_{x,\max}])\text{ HN-indecomposable}\}.
\end{align*}
Let $x\in A'$ be maximal with respect to $\leq_{K'}$. Since $A'\subseteq A$, we find a $\leq_K$-maximal element $y\in A$ such that $x\leq_K y$. By assumption on $(G,\mu,K)$, we see that $[b_{y,\max}] = [b_{\mi}]$. Consider the following irreducible subsets of $K\Adm(\mu)K$:
\begin{align*}
&S_1 = K\cdot_\sigma (IxI) = \{k i_1 x i_2 \sigma(k^{-1})\mid i_1, i_2\in I,~k\in K\},
\\&S_2 = K\cdot_\sigma (IyI) = \{k i_1y i_2 \sigma(k^{-1})\mid i_1, i_2\in I,~k\in K\},
\\&S_1' = K'\cdot_\sigma (IxI) = K'\cdot_\sigma S_1,
\\&S_2' = K'\cdot_\sigma (IyI) = K'\cdot_\sigma S_2.
\end{align*}
By definition, $S_1$ is the EKOR-stratum of $K\Adm(\mu)K$ indexed by $x$. We have $S_1\subseteq \overline{S_2}$ since $x\leq_K y$. Thus $S_1'\subseteq K'\cdot_\sigma \overline{S_2}$. The latter set $K'\cdot_\sigma\overline{S_2}$ is irreducible (since $S_2$ is), contained in $K'\Adm(\mu)K'$ and closed under $(K',\sigma)$-conjugation. We conclude that $K'\cdot_\sigma \overline{S_2}$ is a union of EKOR-strata of $K'\Adm(\mu)K'$, containing precisely one generic stratum $K'\cdot_\sigma (Iy'I)$ for some $y'\in A'$.

Now the fact $S_1'\subseteq \overline{K'\cdot_\sigma (Iy'I)}$ implies $x\leq_{K'} y'$. By maximality of $x$, we get $x=y'$. Since $\overline{K'\cdot_\sigma(Iy'I)}$ contains $S_2$ by assumption, the generic $\sigma$-conjugacy class of $\overline{K'\cdot_\sigma(Iy'I)}$ (which is $[b_{x,\max}]$) is at least $[b_{y,\max}]$, which agrees with $[b_{\mi}]$ by assumption. Since we assumed $(G,\mu,[b_{x,\max}])$ to be Hodge-Newton indecomposable, this finishes the proof.
\end{proof}
The main reason to study the property (ING) is the following result.

\begin{proposition}\label{prop:ING_L1BC}
    Let $(G,\mu,K)$ satisfy (ING). Define
\begin{align*}
A := \{x\in \Adm(\mu)^K\mid (G,\mu,[b_{x,\max}])\text{ HN-indecomposable}\},
\end{align*}
and
\begin{align*}
    \mathcal A = \bigcup_{x\in A} KxK = (K\Adm(\mu)K)\cap \bigcup_{(G,\mu,[b])\text{ HN-indec.}} [b].
\end{align*}
Assume moreover that all Hodge-Newton indecomposable $[b]\in B(G,\mu)$ satisfy (L1BC).
\begin{enumerate}[(a)]
\item There is a one-to-one correspondence between Bruhat-maximal elements of $A$ and irreducible components of $\mathcal A$, sending $x$ to the closure of its associated EKOR stratum.
\item For every irreducible component $\mathcal C$ of $\mathcal A$ and any $\sigma$-conjugacy class $[b]$ such that $(G,\mu,[b])$ is Hodge-Newton indecomposable, the intersection $\mathcal C\cap [b]$ is equidimensional with
\begin{align*}
    \codim(\mathcal C\cap [b]~\subseteq \mathcal C) = \length([b], [b_{\mu,\indec}]).
\end{align*}
Furthermore, $$\overline{(K\Adm(\mu) K)\cap [b]} =\overline{\mathcal A\cap [b]} = \bigcup_{[b']\leq [b]} \mathcal A\cap [b'] = \bigcup_{[b']\leq [b]} (K\Adm(\mu) K)\cap [b'].
$$
\item For any $[b]\in B(G, \mu)$ such that $(G,\mu,[b])$ is Hodge-Newton indecomposable, the dimension of
\begin{align*}
    X(\mu,b,K) = \{g\in G(\breve F)/K\mid g^{-1}b\sigma(g)\in K\Adm(\mu)K\}
\end{align*}
is
\begin{align*}
    &\dim X(\mu,b,K) = (\max_{x\in A}\ell(x)) - \langle \nu(b),2\rho\rangle - \length([b], [b_{\mu,\indec}])
    \\&=(\max_{x\in A}\ell(x)) - \langle \nu(b)+\nu(b_{\mu,\indec}),\rho\rangle + \frac 12(\defect(b_{\mu,\indec}) - \defect(b))
    \\&=(\max_{x\in A}\ell(x))+\#(\Delta/\sigma) + \frac 12\left(- \langle \mu,2\rho\rangle - \langle \nu(b),2\rho\rangle - \defect(b)\right).
\end{align*}
If in this case all Bruhat-maximal elements of $A$ have the same length, then $X(\mu,b,K)$ is also equi-dimensional.
\end{enumerate}
\end{proposition}
\begin{proof}
    \begin{enumerate}[(a)]
    \item This is a usual property of EKOR strata and does not require the assumption of (L1BC).
    \item First let $[b]\in B(G,\mu)$ be any $\sigma$-conjugacy class such that $\mathcal C\cap [b]\neq\emptyset$. We would like to prove that
    \begin{align}
    \codim(\mathcal C\cap[b]\subseteq \mathcal C)\geq \length([b], [b_{\mu,\indec}]).\label{eq:L1BC_ING_codimInequality}
    \end{align}
    For this, we prove that for each $x\in A$ whose corresponding EKOR-stratum $\mathcal S$ lies in $\mathcal C$, we have
    \begin{align*}
        \codim(\mathcal S\cap[b]\subseteq \mathcal C)\geq \length([b], [b_{\mu,\indec}]).
    \end{align*}
    Note that $[b_{x,\max}]$ is the generic $\sigma$-conjugacy class inside $\mathcal S$. By Lemma~\ref{lem:CL1BC_cordial}, we get
    \begin{align*}
        \codim(\mathcal S\cap[b]\subseteq \mathcal S) = \length([b], [b_{x,\max}]).
    \end{align*}
    If $x$ is Bruhat maximal inside $A$, then $[b_{x,\max}] = [b_{\mu,\indec}]$ by assumption and we are done. Otherwise, we find a sequence of Bruhat covers in $\widetilde W^K$
    \begin{align*}
    x = x_0\lessdot x_1\lessdot\cdots \lessdot x_m
    \end{align*}
    with $x_m\in A$ being Bruhat maximal. Let $\mathcal S_1,\dotsc,\mathcal S_m$ denote the corresponding EKOR-strata. Then $\codim(\mathcal S_i\subseteq \overline{\mathcal S_{i+1}})\geq 1$, proving that
    \begin{align*}
    \codim(\mathcal S\subseteq \mathcal C)\geq \codim(\mathcal S\subseteq \overline{S_m}) \geq m.
    \end{align*}
    Using the fact that all generic $\sigma$-conjugacy classes of the $x_0,\dotsc,x_m$ satisfy (L1BC), we get $\length([b_{x_i,\max}], [b_{x_{i+1},\max}])\leq 1$ for all $i$. From the condition (ING), we get $[b_{x_m,\max}] = [b_{\mu,\indec}]$. Thus
    \begin{align*}
    &m \geq \length([b_{x,\max}], [b_{x_m,\max}]) = \length([b_{x,\max}], [b_{\mu,\indec}]).
    \end{align*}
    We conclude
    \begin{align*}
        \codim(\mathcal C\cap[b]\subseteq \mathcal C)&=\codim(\mathcal S\cap [b]\subseteq \mathcal S) +\codim(\mathcal S\subseteq \mathcal C)
        \\&\geq \length([b], [b_{x,\max}]) + m
        \\&\geq\length([b], [b_{x,\max}]) + \length([b_{x,\max}], [b_{\mu,\indec}]) \\&=\length([b], [b_{\mu,\indec}]).
    \end{align*}
    We have proved \eqref{eq:L1BC_ING_codimInequality}.

By Theorem \ref{thm:purity} and Remark \ref{rem:viehmann2015lemma} we obtain that $\mathcal C\cap [b]$ is equidimensional of codimension equal to $\length([b], [b_{\mu,\indec}])$. 	Observe that $A$ contains a unique length zero element $\tau\in A$, and that the EKOR-stratum associated with $\tau$ lies in $\mathcal C$. Hence $\mathcal C$ intersects the unique basic $\sigma$-conjugacy class in $B(G,\mu)$. This implies the claim.

Moreover, we get that any $[b']\in B(G,\mu)$ with $[b]\leq [b']\leq [b_{\mu,\indec}]$ satisfies $\mathcal C\cap [b']\neq\emptyset$ as well as the second claimed equality in the last line of (b). The first and third follow from the fact that for every Hodge-Newton indecomposable $[b]$ every $[b']\leq [b]$ is again Hodge-Newton indecomposable and hence $K\Adm(\mu) K\cap [b']=\mathcal A\cap [b']$.

	\item We use Theorem \ref{thmdimfoliat} together with the above formula for codimensions of Newton strata and obtain
	\begin{align*}
	\dim X(\mu,b,K) =& 
    \dim \mathcal A - \length([b], [b_{\mu,\indec}]) - \langle \nu(b),2\rho\rangle.
	\end{align*}
Then $\dim\mathcal A = \max_{x\in A}\ell(x)$ yields the first formula. In the same way, if $\mathcal A$ is equidimensional, then so is $X(\mu,b,K)$. The second formula then follows from Chai's length formula \cite{Chai2000}
	\begin{align*}
	\length([b], [b_{\mu,\indec}]) =\frac 12\left( \langle \nu(b_{\mu,\indec})-\nu(b),2\rho\rangle + \defect(b)-\defect(b_{\mu,\indec})\right).
	\end{align*}
	
	For the third formula, we observe that
	\begin{align*}
	\langle \nu(b_{\mu,\indec}),2\rho\rangle - \defect(b_{\mu,\indec}) = \langle \lambda(b_{\mu,\indec}),2\rho\rangle = \langle \mu-\sum_{\alpha\in \Delta/\sigma} \alpha^\vee,2\rho\rangle.
	\end{align*}
	This finishes the proof.
    \qedhere\end{enumerate}
\end{proof}
In Section \ref{sec54}, we prove that (ING) holds in many cases of $\depth(G,\mu)<2$. More precisely, we will show the following.
\begin{theorem}\label{thm:INGoverview}
Let $G'$ be absolutely quasi-simple and of Cartan type $A_m$ and let $G$ be a Weil restriction of scalars of $G'$ as in Section~\ref{sec:weilRestrict}. Write $\mu = (\mu_1,\dotsc,\mu_d)$ as in that section and let again $\mu' = \mu_1+\cdots+\mu_d$. In the following cases, the pair $(G,\mu,I)$ satisfies the property (ING).
   \begin{center} 
    \begin{tabular}{c|l}
    $m$&Conditions on $\mu$\\
    \midrule
    $m\geq 1$ & $\mu' = 2\omega_1^\vee+\omega_m^\vee$,\\
    $m\geq 2$ & $d=1$ and $\mu = 2\omega_1^\vee$,\\
    $m\geq 2$ & $\mu' = \omega_2^\vee+\omega_m^\vee$,\\
    $m\geq 4$ & $\mu' = \omega_2^\vee$.
\end{tabular}
\end{center}
\end{theorem}

The starting point for the proof is the following result of to He-Yu.
\begin{proposition}[{\cite[Proposition~3]{He2024}}]\label{prop:HeYu}
    Let $\mu$ be a dominant cocharacter and $K\subseteq G(\breve F)$ a semi-standard parahoric subgroup with Weyl group $\widetilde W_K\leq \widetilde W$. For any $x\in\Adm(\mu)$, the intersection
    \begin{align*}
        \Adm(\mu) \cap x \widetilde W_K
    \end{align*}
    contains a unique Bruhat maximal element.\rightqed
\end{proposition}

\subsection{A Hodge polygon cutoff}\label{sec:npCutoff}
For this subsection, we focus on the split group $G = \GL_n$. Let $B$ be the Borel subgroup of upper triangular matrices and $T$ the split torus of diagonal matrices. We identify $X_\ast(T) \cong \mathbb Z^n$ as usual. We fix a dominant cocharacter $\mu = (\mu_1,\dotsc,\mu_n)\in X_\ast(T)$, i.e.~$\mu_1\geq\cdots \geq \mu_n$. Then the admissible set $\Adm(\mu)\subseteq \widetilde W$ is stable under conjugation by length zero elements of $\widetilde W$, and contained in the union of all double cosets $W_0 \varepsilon^{\mu'} W_0$ where $\mu'\in X_\ast(T)$ ranges over all dominant cocharacters such that $\mu'\leq \mu$. In this section, we study the Hodge-Newton indecomposable locus of $\Adm(\mu)$ and how to cover it, up to length zero conjugation, by \enquote{fewer} subsets $W_0\varepsilon^{\mu'} W_0$.

\begin{definition}
Let $s\in\mathbb Q$.
\begin{enumerate}[(a)]
\item We say that a (rational) dominant coweight $\mu' = (\mu_1',\dotsc,\mu_n')\leq \mu$ is \emph{$s$-similar to $\mu$} if 
\begin{align*}
\sum_{i=1}^n \abs{\mu_i-s} = \sum_{i=1}^n \abs{\mu_i'-s}.
\end{align*}
\item We say that $\mu$ is \emph{$s$-trivial} if either all entries satisfy $\mu_i\geq s$, or if all entries satisfy $\mu_i\leq s$.
\item For any (rational) cocharacter $\lambda = (\lambda_1,\dotsc,\lambda_n)\in \mathbb Q^n$, we define
\begin{align*}
D_{>s}(\lambda) &= \{j\in\{1,\dotsc,n\}\mid \lambda_j>s\},\quad S_{>s}(\lambda) = \sum_{j\in D_{>s}(\lambda)} (\lambda_j - s)\\
D_{<s}(\lambda) &= \{j\in\{1,\dotsc,n\}\mid \lambda_j<s\},\quad S_{<s}(\lambda) = \sum_{j\in D_{<s}(\lambda)} (s-\lambda_j).
\end{align*}
\end{enumerate}
\end{definition}
\begin{remark}
Observe that $D_{>s}(w\lambda) = w^{-1}D_{>s}(\lambda)$ and $S_{>s}(w\lambda) = S_{>s}(\lambda)$ for all $\lambda\in\mathbb Q^n$ and all $w\in W_0$. Moreover, we have 
\begin{align*}S_{>s}(\lambda) - S_{<s}(\lambda) = -ns+\sum_{j=1}^n \lambda_j.\end{align*}
So if $\lambda\leq \mu$, then $\lambda_{\dom}$ is $s$-similar to $\mu$ iff $S_{>s}(\lambda_{\dom}) + S_{<s}(\lambda_{\dom}) = S_{>s}(\mu) + S_{<s}(\mu)$.

By the two observations above, this is equivalent to $S_{>s}(\lambda)+S_{<s}(\lambda) = S_{>s}(\mu)+S_{<s}(\mu)$, which is moreover equivalent to $S_{>s}(\lambda)= S_{>s}(\mu)$.
\end{remark}
\begin{remark}
Draw $\mu$ as an upper convex Hodge polygon. Consider the unique tangent line with slope $s$. Then $\mu$ is $s$-trivial if and only if this tangent line touches an endpoint of the polygon. Now $\mu'\leq \mu$ is $s$-similar to $\mu$ if and only the Hodge polygon of $\mu'$ touches this tangent line.

Consider for example $\mu = (2,0,0,-1)$ and $s=0$. We draw the Hodge polygon of $\mu$ in black and the tangent line of slope $s$ in solid red. Then $\mu'\leq \mu$ is $s$-similar to $\mu$ if and only if the Hodge polygon of $\mu'$ touches the red line in the picture below.
\\[1em]
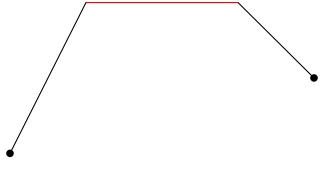
\begin{center}\begin{tikzpicture}
\draw (0,0) -- (1,2) -- (2,2) -- (3,2) -- (4,1);
\fill (0,0) circle[radius=0.05];
\fill (4,1) circle[radius=0.05];
\draw[red] (1,2) -- (3,2);
\end{tikzpicture}\end{center}
~
\end{remark}
The main result of this subsection is the following.
\begin{proposition}\label{prop:newtonCutoff}
Let $s\in\mathbb Z$ such that $\mu$ is not $s$-trivial, and $x\in\Adm(\mu)$ such $\nu(\dot x)\in X_\ast(T)_{\mathbb Q}$ is not $s$-similar to $\mu$. Then there exists a length zero element $\tau\in\Omega$ and a dominant cocharacter $\mu' \in X_\ast(T)$ such that  $x' = \tau^{-1} x \tau\in W_0\varepsilon^{\mu'} W_0$ and $\mu'\leq \mu $ is not $s$-similar to $\mu$.
\end{proposition}
\begin{proof}
Write $x = w_x \varepsilon^{\mu_x}$ with $w_x\in W_0$ and $\mu_x\in\mathbb Z^n$ a cocharacter. For $i=1,\dotsc,n$, let $\omega_i = (1,\dotsc,1,0,\dotsc,0)$ be the fundamental coweight with the first $i$ entries given by $1$ and the others equal to $0$. Denote the corresponding length zero element by $\tau_i$, so that $\Omega \cap \varepsilon^{\omega_i} W_0 = \{\tau_i\}$.

Define $\mu^{(i)} = \mu_x + \omega_i - w_x^{-1}\omega_i$, such that $\tau_i^{-1} x \tau_i\in W_0 \varepsilon^{\mu^{(i)}} W_0$. The condition $x\in\Adm(\mu)$ is equivalent to the \emph{permissible set condition} $(\mu^{(i)})_{\dom}\leq \mu$ for all $i$. The equality of admissible and permissible sets for $\GL_n$ has been known for a while, and a concise proof can be found in \cite[Corollary~4.14]{Schremmer2024_bruhat}. We have to show $S_{>s}(\mu^{(i)}) < S_{>s}(\mu)$ (or equivalently $S_{<s}(\mu^{(i)}) < S_{<s}(\mu)$) for some $i$. Thus assume, aiming for a contradiction, that for all $1\leq i\leq n$ we have
\begin{align}
 S_{>s}(\mu^{(i)}) = S_{>s}(\mu)\text{ and }S_{<s}(\mu^{(i)}) = S_{<s}(\mu).\label{eq:newtonCutoffProof}
\end{align}

Define the sets
\begin{align*}
D_{>s}(x) := \bigcup_{i=1}^n D_{>s}(\mu^{(i)})\subseteq \{1,\dotsc,n\},\qquad 
D_{<s}(x) := \bigcup_{i=1}^n D_{<s}(\mu^{(i)})\subseteq \{1,\dotsc,n\}.
\end{align*}
For any $i,i',j\in\{1,\dotsc,n\}$, observe that $\abs{\mu^{(i)}_j - \mu^{(i')}_j}\leq 1$.
Hence $D_{>s}(x)\cap D_{<s}(x)=\emptyset$.

By assumption, $\mu$ is not $s$-trivial. This, together with \eqref{eq:newtonCutoffProof}, then implies that $D_{>s}(x)\neq\emptyset\neq D_{<s}(x)$.

Next, we will prove that both $D_{>s}(x)$ and $D_{<s}(x)$ are stable under the action of $w_x\in W$. We only do this for $D_{<s}(x)$, as the proof for $D_{>s}(x)$ is similar. So pick $j\in D_{<s}(x)$, and assume without loss of generality that $j\neq w_x(j)$. Observe that $\mu^{(j-1)}_j\leq \mu^{(i)}_j$ for all $i\in\{1,\dotsc,n\}$ (where we put $\mu^{(0)} := \mu_x$). Hence $j\in D_{<s}(\mu^{(j-1)})$. We get \begin{align*}\mu^{(j)} = \mu^{(j-1)} + e_j - e_{w_x^{-1}(j)},\end{align*} where $e_j$ denotes the $j$-th basis vector $e_j = (0,\dotsc,0,1,0,\dotsc,0)$. It suffices to prove that $w_x^{-1}(j)\in D_{<s}(\mu^{(j)})$, since then $w_x^{-1}(j)\in D_{<s}(x)$ follows. Now if instead we had $\mu^{(j)}_{w_x^{-1}(j)}\geq s$, then also $\mu^{(j-1)}_{w_x^{-1}(j)}\geq s$. Hence
\begin{align*}
S_{<s}(\mu^{(j)}) - S_{<s}(\mu^{(j-1)}) = (\mu^{(j)}_j-s) - (\mu^{(j-1)}_j-s) = 1,
\end{align*}
contradicting \eqref{eq:newtonCutoffProof}.

The contradiction shows that $D_{<s}(x)$ is stable under $w_x^{-1}$ (hence also stable under $w_x$), and we use the same result for $D_{>s}(x)$. Let $v\in W$ be any permutation sending $D_{>s}(x)$ to the subset $\{1,\dotsc,\# D_{>s}(x)\}$ and sending $D_{<s}(x)$ to the subset $\{n-\# D_{<s}(x)+1,\dotsc,n\}$.

Let $N>0$ with $w_x^N=1$, so
\begin{align*}
\nu_x = \frac 1N\sum_{i=1}^N w_x^i \mu_x\in\mathbb Q^n
\end{align*}
is the non-dominant Newton point of $x$. Then
\begin{align*}
S_{>s}(\nu_x) &= \sum_{i\in D_{>s}(x)} ((\nu_x)_i-s).
\intertext{
Since $D_{>s}(x)$ is $w_x$-stable, this is}
&=\sum_{i\in D_{>s}(x)} ((\mu_x)_i-s) = S_{>s}(\mu_x) \underset{\text{\eqref{eq:newtonCutoffProof}}}= S_{>s}(\mu).
\end{align*}
This contradicts the assumption that $\nu(\dot x)$ is not $s$-similar to $\mu$.
\end{proof}
As an immediate application of Proposition~\ref{prop:newtonCutoff}, we present the following boundedness result for ADLV. It is of independent interest and not used in the remainder of this article.

\begin{proposition}
	Let $\mu\in X_\ast(T)$ a dominant cocharacter such that $[1]\in B(G,\mu)$.
	Set $d = \depth(G,\mu)\in\mathbb Z_{\geq 0}$ and $K = I W_0 K = \GL_n(\mathcal O_{\breve F})$.
	Then
	\begin{align*}
		X_\mu(1) := X(\mu,1,K)\subseteq \bigcup_\lambda \mathbb J_1(F) K\varepsilon^\lambda K,
	\end{align*}
	where the union is taken over all cocharacters $\lambda\in X_\ast(T)$ satisfying
	$\langle \lambda,\alpha \rangle \leq d$ for all roots $\alpha\in\Phi$.
\end{proposition}
\begin{proof}
	Induction on $d$. If $d=0$, we get $\mu=0$ and $X_\mu(1) = \mathbb J_1(F) K$. The claim follows from this.
	
	Let now $d>0$ and the claim be proved for $d-1$. Let $s=0$. Then
	\begin{align*}
		\depth(G,\mu) = \max_{\alpha\in\Delta}\langle \mu,\omega_\alpha\rangle = S_{>0}(\mu).
	\end{align*}
	
	We know that
	\begin{align*}
		X_\mu(1) = \bigsqcup_{x\in \Adm(\mu)^{\mathbb S_0}} \pi(X_x(1)),
	\end{align*}
	where $\pi : G(\breve F)/I\to G(\breve F)/K_0$ is the projection map. Hence it suffices to control the location of each $X_x(1)$ inside the affine flag variety.
	
	Let hence $x\in \Adm(\mu)^{\mathbb S_0}$. If $\mu$ and $[\dot x]$ are $(s=0)$-similar, then by the Hodge-Newton decomposition of \cite[Theorem~4.17]{Goertz2019}, or alternatively the nonemptiness criterion for the basic locus from \cite{Goertz2015_nonemptiness}, we get $X_x(1)=\emptyset$. Otherwise, Proposition~\ref{prop:newtonCutoff} gives us a length zero element $\tau\in \Omega$ such that $x' =\tau^{-1} x \tau$ is not $s$-similar to $\mu$, that is $x'\in W_0 \varepsilon^{\mu'} W_0$ for a dominant cocharacter $\mu'$ which is not $(s=0)$-similar to $\mu$. Note that
	\begin{align*}
		X_x(1) = X_{x'}(1)\tau,\text{ thus }\pi(X_x(1)) \subseteq \pi(X_{x'})(1) K_0 \tau K_0 \subseteq X_{\mu'}(1) K_0 \tau K_0.
	\end{align*}
	The condition that $\mu$ and $\mu'$ are not $0$-similar means precisely that
	\begin{align*}
		\depth(G,\mu') = S_{>0}(\mu')< S_{>0}(\mu) = d.
	\end{align*}
	By induction, we get 
	\begin{align*}
		X_{\mu'}(1)\subseteq \bigcup_{\substack{\lambda \in X_\ast(T)\\ \langle\lambda,\alpha\rangle \leq d-1\forall \alpha\in \Phi}} \mathbb J_1(F) K\varepsilon^\lambda K.
	\end{align*}
	If $\lambda\in X_\ast(T)$ satisfies $\langle \lambda,\alpha\rangle \leq d-1$ for all $\alpha\in \Phi$ and $\omega$ is in the $W_0$-orbit of a minuscule cocharacter, then $\langle \lambda+\omega,\alpha\rangle \leq d$ for all $\alpha\in \Phi$. This finishes the induction and the proof.
\end{proof}

\subsection{Some cases of (ING)}\label{sec54}
The goal of this section is to prove Theorem~\ref{thm:INGoverview}. We assume that $G'$ is the split group $G' = \GL_{m+1}$ over an extension $F^{(d)}$ as in Section~\ref{sec:weilRestrict}. Furthermore, let $T'$, $B'$, $\mu$ and $\mu'$ be as in section. As usual we denote the $i$-th fundamental coweight of $\GL_{m+1}$ by $\omega_i^{\vee} = (1,\dotsc,1,0,\dotsc,0)\in X_\ast(T')$ for $i=1,\dotsc,m$. The following lemma is helpful to determine the generic $\sigma$-conjugacy class $[b_{x,\max}]$ of Iwahori double cosets $IxI$ in this setting.
\begin{lemma}\label{lem:gnpWeilRestrict}
    Let $x = (x_1,\dotsc,x_d)\in\widetilde W$ and $x' = x_d\ast \cdots \ast x_1$, where $\ast$ denotes the Demazure product on $\widetilde W'$. Then under the canonical isomorphism $B(G)\xrightarrow\sim B(G')$, the generic $\sigma$-conjugacy class $[b_{x,\max}]$ gets mapped to $[b_{x',\max}]$.
\end{lemma}
\begin{proof}
    For $i=1,\dotsc,d$, put $y_i := (x_i,1,\dotsc,1)\in\widetilde W$ such that $x$ is the length additive product
    \begin{align*}
        x = y_1 \sigma(y_2)\sigma^2(y_3)\cdots \sigma^{d-1}(y_d).
    \end{align*}
    Since this is length additive, it is also a Demazure product, i.e.
    \begin{align*}
        x = y_1 \ast \sigma(y_2)\ast \sigma^2(y_3)\cdots \ast \sigma^{d-1}(y_d).
    \end{align*}
    Using \cite[Lemma~2.3]{He2024_demazure}, we get
    \begin{align*}
        [b_{x,\max}] = [b_{y_2 \ast \sigma(y_3)\ast\cdots \ast \sigma^{d-2}(y_d)\ast y_1,\max}].
    \end{align*}
    Observing where the different terms of the Demazure product lie in the decomposition $\widetilde W = (\widetilde W')^d$, we see that $\sigma(y_3)\ast\cdots \ast \sigma^{d-2}(y_d)\ast y_1 = y_1\ast \sigma(y_3)\ast\cdots \ast \sigma^{d-2}(y_d)$. So we get
    \begin{align*}
        [b_{x,\max}] = [b_{y_2 \ast y_1 \ast \sigma(y_3)\ast\cdots \ast \sigma^{d-2}(y_d),\max}].
    \end{align*}
Inductively, we obtain
    \begin{align*}
    [b_{x,\max}] = [b_{y_d\ast \cdots \ast y_1,\max}] = [b_{(x',1,\dotsc,1),\max}].
    \end{align*}
    By the definition of generic $\sigma$-conjugacy classes and the natural map $B(G)\to B(G')$, it is easy to see that $[b_{(x',1,\dotsc,1),\max}]$ gets mapped to $[b_{x',\max}]$.
\end{proof}

\begin{proposition}\label{prop:ing_2omega1pusomegam}
Let $\mu\in X_*(T)$ be such that $\mu' \in\{ (2,0,\dotsc,0,-1),(1,1,0,\dotsc,0,-1)\}\subseteq X_\ast(T')$ and set $\lambda  = (1,0,\dotsc,0)\in X_\ast(T').$
    
    Define
    \begin{align*}
        A := \{x\in \Adm(\mu)\mid \text{the image of $[b_{x,\max}]$ under $B(G)\xrightarrow\sim B(G')$ is }\leq [\varepsilon^\lambda]\}.
    \end{align*}
    Let $x\in A$ be Bruhat maximal in $A$.
    Then
    \begin{enumerate}[(a)]
    \item There exists a pure translation element $z\in \Adm(\mu)$ with $z\leq x$. It satisfies $\nu(\dot z) = \lambda \in X_\ast(T')$ under the canonical identification $B(G)\xrightarrow\sim B(G')$.
    \item We have $\nu(b_{x,\max}) = \lambda$. In particular, if $\mu' = (2,0,\dotsc,0,-1)$, then $(G,\mu,I)$ satisfies property (ING).
    \item If $\mu$ is minuscule, then $x=z$ and $\ell(x) = \langle \mu,2\rho\rangle$.
    \end{enumerate}
\end{proposition}
\begin{proof}
    Let us write $x = (x_1,\dotsc,x_d)$. For $i=1,\dotsc,d$, we define $x_{\ast,i} := x_{i-1} \ast x_{i-2}\ast \cdots \ast x_1 \ast x_d\ast\cdots \ast x_{i+1}\ast x_{i} \in \Adm_{G'}(\mu')$. It satisfies $[b_{x,\max}] = [b_{x_{\ast,i},\max}]$ under the identification $B(G,\mu)\cong B(G',\mu')$ by Lemma~\ref{lem:gnpWeilRestrict} and \cite[Lemma~2.3]{He2024_demazure}. We use Proposition~\ref{prop:newtonCutoff} where we choose $s=0$. We obtain a length zero element $\tau_i\in \Omega$ such that 
    \begin{align*}
        x_{\ast,i} \in \tau W_0 \varepsilon^\lambda W_0\tau^{-1} = \widetilde W_{K_i}\varepsilon^{\lambda} \widetilde W_{K_i},
    \end{align*}where $K_i = \tau \GL_{m+1}(\mathcal O_{\breve F}) \tau^{-1}$ is the hyperspecial subgroup associated to $\tau_i$. It follows for every $y\in x_{\ast,i} \widetilde W_{K_i}$ that $\nu(b_{y,\max})\leq \lambda$.
    
    By Proposition~\ref{prop:HeYu}, the intersection $\Adm(\mu_i)\cap x_i \widetilde W_{K_i}$ contains a unique Bruhat maximum, say $y_i$. We certainly get $x_i\leq y_i$. Hence $x\leq \tilde x := (x_1,\dotsc,x_{i-1}, y_i, x_{i+1},\dotsc,x_d)$. Since $\tilde x_{\ast,i} \in x_{\ast,i} \widetilde W_{K_i}$, we get $\tilde x\in A$. By Bruhat maximality, we get $x = \tilde x$, i.e.\ $x_i = y_i$.

    Since $K_i$ is hyperspecial, the coset $x_i \widetilde W_{K_i}$ contains a unique pure translation element, which we denote $z_i = \varepsilon^{\lambda_i}$. Since $\Adm(\mu_i)\ni x_i\in \widetilde W_{K_i} z_i \widetilde W_{K_i}$, we get $(\lambda_i)_{\dom}\leq \mu_i$ and hence $z_i\in \Adm(\mu_i)$. By choice of $y_i$, we get $z_i\leq y_i$, which agrees with $x_i$ as we saw before. Hence $z_i\leq x_i$ and $z := (z_1,\dotsc,z_d)\leq x$. In particular, $(G,\mu,[\dot z])$ is Hodge-Newton indecomposable and $[\dot z]$ is an unramified $\sigma$-conjugacy class. A quick inspection of the Newton polygons of $\mu'$ and $\lambda$ shows that this is only possible if $\nu(\dot z) = \lambda$. This proves (a).
    
    By \cite[Corollary~5.6]{Viehmann2014}, we get $\nu(b_{x,\max})\geq \lambda$. Inspecting Newton polygons as above, we see that $x\in A$ forces $\nu(b_{x,\max})\leq\lambda$, and we have proved (b).

    For (c), assume that $\mu$ is minuscule, i.e.\ each $\mu_i$ is either $(1,0,\dotsc,0)$ or $(0,\dotsc,0,-1)$. Then $(\lambda_i)_{\dom}\leq \mu_i$ implies already that $(\lambda_i)_{\dom} = \mu_i$, i.e.\ $z_i$ is Bruhat maximal inside $\Adm(\mu_i)$. Hence $x=z$ and we get the desired length formula.
\end{proof}

We next show that (ING) holds for $(\GL_{m+1},2\omega_1^\vee,I)$. Some preparation is needed for that.

\begin{lemma}\label{lem:ing_2omega1}
    Let $G = \GL_{m+1}$ and $\mu = (2,0,\dotsc,0)\in X_\ast(T)$. We also set $\omega := (1,1,0,\dotsc,0)\leq \mu$.
    \begin{enumerate}[(a)]
    \item We have $\lambda(b_{\mu,\indec}) = (1,0,\dotsc,0,1)$. For all $x\in \Adm(\mu)\setminus \Adm(\omega)$, the generic $\sigma$-conjugacy class satisfies $[b_{x,\max}]\geq [b_{\mu,\indec}]$.
    \item Every $x\in \Adm(\omega)$ can be uniquely written in the form $x = w_x\varepsilon^{v_x\omega}$ for $v_x\in W_0^{\{s_1,s_3,s_4,\dotsc,s_m\}}$ and $w_x\in W_0$. 
    \item Define $A := \{x\in \Adm(\omega)\mid [b_{x,\max}]\leq [b_{\mu,\indec}]\}$. Pick $v_x\in W_0^{\{s_1,s_3,s_4,\dotsc,s_m\}}$, seen as a permutation of $\{1,\dotsc,m+1\}$, and let $a=v_x(1)$. The set
    \begin{align*}
        \{w_x v_x\mid w_x\varepsilon^{v_x\omega} \in A\}
    \end{align*}
    is given by $\{w\in W_0\mid w\leq s_{a-1}\cdots s_1\}$ if $v_x(2) = m+1$, and empty otherwise. In the former case, we have
    \begin{align*}
        v_x = s_{a-1}\cdots s_1 s_m\cdots s_2 \in W_0.
    \end{align*}
    \end{enumerate}
\end{lemma}
\begin{proof}
    \begin{enumerate}[(a)]
    \item It is clear that $\lambda(b_{x,\max}) = \mu -\sum_{\alpha\in\Delta}\alpha^\vee$ for split group, hence the first claim.
    
    For the group $\GL_{m+1}$, the notion of admissible and permissible sets coincide. Hence $\Adm(\mu)\setminus \Adm(\omega)$ consists of all elements $x\in \Adm(\mu)$ such that there is $\tau\in \Omega$ with $\tau^{-1} x \tau \in W_0\varepsilon^{\mu} W_0$. So it suffices to show $[b_{x,\max}]\geq [b_{\mu,\indec}]$ for all $x\in W_0\varepsilon^\mu W_0$.

    Write $x = w_x\varepsilon^{v_x \mu}$ with $v_x\in W_0^{\{s_2,\dotsc,s_m\}}$. We get $\lambda(b_{x,\max})\geq \mu - \wt(v_x\Rightarrow w_x v_x) \geq \mu - \wt(w_0\Rightarrow 1)$, where $\wt$ denotes the weight function of the quantum Bruhat graph and $w_0\in W_0$ is the longest element. The weight $\wt(w_0\Rightarrow 1)$ is computed explicitly in \cite[Section~5]{Sadhukhan2022_qbg} to be
    \begin{align*}
        \wt(w_0\Rightarrow 1) = (1, 3, 5,\dotsc, -5, -3, -1) \in X_\ast(T).
    \end{align*}
    Hence $\mu - \wt(w_0\Rightarrow 1) = (1, -3, -5,\dotsc,5, 3, 1)$, showing that
    \begin{align*}
        \nu(b_{x,\max}) \geq (1, -1/m,\dotsc,-1/m).
    \end{align*}
    The claim follows.
    \item Since $\omega$ is minuscule, we get $\Adm(\omega)\subseteq W_0\varepsilon^\omega W_0$. Now the claim is clear.
    \item The condition $v_x\in W_0^{\{s_1,s_3,s_4,\dotsc,s_m\}}$ means, seeing $v_x$ as a permutation, that
    \begin{align*}
        v_x(1)<v_x(2),\qquad v_x(3)<v_x(4)<\cdots <v_x(m+1).
    \end{align*}
    Hence $v_x(2)=m+1$ or $v_x(m+1) = m+1$. Given $x = w_x\varepsilon^{v_x\omega}$, we get $x\in \Adm(\mu)\iff w_x v_x\leq v_x$ in the Bruhat order by \cite[Proposition~4.12]{Schremmer2024_bruhat}. Thus, $v_x(m+1)=m+1$ would imply that $v_x$ is in the subgroup generated by $s_1,\dotsc,s_{m-1}$, and then also $w_x v_x(m+1)=m+1$. Hence $v_x(m+1)$ is a fixed point of the permutation $w_x$, showing that the $v_x(m+1)$-coordinate of
    \begin{align*}
        \sum_{n=1}^N w_x^n(v_x \mu_x)
    \end{align*}
    is zero for all $N\geq 1$. Hence one coordinate of $\nu(\dot x)$ is zero, meaning $x\notin A$. This is a contradiction.

    We conclude that $v_x(2)=m+1$. With $v_x(1)$ and $v_x(2)$ fixed, the remaining values $v_x(3)<\cdots <v_x(m+1)$ are uniquely determined through
    \begin{align*}
    \{v_x(3),\dotsc,v_x(m+1)\} = \{1,2,\dotsc,a-1,a+1,\dotsc,m\}.
    \end{align*}
    Explicitly, that means
    \begin{align*}
    v_x(i) = \begin{cases}a,&i=1,\\m+1,&i=2,\\i-2,&3\leq i\leq a+1,\\
    i-1,&a+2\leq i\leq m+1.\end{cases}
    \end{align*}
    Then for $x = w_x\varepsilon^{v_x}$ to lie in $A$, it is certainly necessary that $w_x v_x\leq v_x$ and $w_x v_x(i)\neq v_x(i)$ for $i=3,\dotsc,m+1$ (repeating the above argument). The Bruhat order condition $v_x\leq w_x v_x$ can be made explicit using \cite[Theorem~2.1.5]{Bjorner2005} applied to $w_0 v_x w_0$ and $w_0 w_x v_x w_0$. This means for all $i,j\in\{0,\dotsc,m+1\}$ that
    \begin{align*}
    \#\{k\geq i\mid w_x v_x(k)\leq j\} \leq \#\{k\geq i\mid v_x(k)\leq j\}.
    \end{align*}
    Applies to $i\geq 3$ and $j = v_x(i)-1$, the right-hand side of the inequality is zero. Hence we get $w_x v_x(k) \geq v_x(i)$ for all $k\geq i$. Since also $w_x v_x(i)\neq i$, we conclude $w_x v_x(i)\geq v_x(i)+1$. For $i\geq a+2$, this means $w_x v_x(i)\geq i$, and then equality must hold by $w_x v_x$ being a bijection. For $3\leq i\leq a+1$, the same argument shows that at least $w_x v_x(i)\geq i-1$. The same inequality is trivially satisfied for $i=1,2$.

    We verify $w_x v_x\leq s_{a}\cdots s_1$. Indeed, the above shows that $w_x v_x$ is in the subgroup generated by $s_1,\dotsc,s_a$. Now for $2\leq i\leq a+1$ and $1\leq j\leq a+1$, the simple observation
    \begin{align*}
        &\#\{k\geq i\mid k\leq a+1\text{ and }w_x v_x(k)\leq j\} \\&\leq \#\{k\geq i\mid k\leq a+1\text{ and }k-1\leq j\} \\&= \#\{k\geq i\mid k\leq a+1\text{ and }s_a\cdots s_1(k)\leq j\}
    \end{align*}
    shows the claim, appealing to \cite[Theorem~2.5.1]{Bjorner2005} again. The claimed reduced word for $v_x$ comes from the above explicit calculation of its values.
    \qedhere\end{enumerate}
\end{proof}

\begin{proposition}\label{prop:ing_2omega1}
    Let $G = \GL_{m+1}$ and $\mu = (2,0,\dotsc,0)\in X_\ast(T)$. Then $(G,\mu,I)$ satisfies (ING).
\end{proposition}
\begin{proof}
    Consider
    \begin{align*}
        A := \{x\in \Adm(\mu)\mid [b_{x,\max}]\leq [b_{\mu,\indec}]\}.
    \end{align*}
    Let $x$ be Bruhat maximal inside $A$, and define $\omega$ as in Lemma~\ref{lem:ing_2omega1}. If $x\notin \Adm(\omega)$, then part (a) of this lemma shows $[b_{x,\max}]\geq [b_{\mu,\indec}]$, so we are done.

    Assume now that $x\in A\cap \Adm(\omega)$. Write $x = w_x\varepsilon^{v_x\omega}$ with $v_x\in W_0^{\{s_1,s_3,s_4,\dotsc,s_{m-1},s_m\}}$, and set $a = v_x(1)$. We would like to consider the elements
    \begin{align*}
        y = s_a \cdots s_1 v_x^{-1}\varepsilon^{v_x\omega}.
    \end{align*}
    We certainly get $x\leq y$. To see that $y\in \Adm(\mu)\setminus \Adm(\omega)$, we compute
    \begin{align*}
    \wt(s_a\cdots s_1 \Rightarrow v_x) \leq \wt(s_a\cdots s_1 \Rightarrow s_a\cdots s_2) = \alpha_1^\vee.
    \end{align*}
    
    Next, we claim that $y\in A$. For this, we observe that
    \begin{align*}
        \cl(y) = s_a\cdots s_1 v_x^{-1} = (s_a\cdots s_2) (s_1\cdots s_m) (s_1\cdots s_{a-1}) = s_1\cdots s_m,
    \end{align*}
    which is a Coxeter element. Hence $[\dot y]$ is basic. By Bruhat maximality of $x$ in $A$, we get $x=y\notin \Adm(\omega)$. This contradiction finishes the proof.
\end{proof}

\begin{lemma}\label{lem:L1BC_implies_ING}
    Let $d, F^{(d)}, G', T', B', \mu, \mu'$ be as in Section~\ref{sec:weilRestrict}. We assume that $G'$ is the split group $G' = \PGL_{m+1}$ with $m\geq 2$, and denote the $i$-th fundamental coweight by $\omega_i\in X_\ast(T')$ for $i=1,\dotsc,m$ (according to the usual conventions for type $A$ Dynkin diagrams). We require that
    \begin{align*}
        \mu' \in \{\omega_2^\vee,\omega_2^\vee+\omega_m^\vee\}.
    \end{align*}
    Then $(G,I,\mu)$ satisfies (ING).
\end{lemma}
\begin{proof}
     We find a simple root $\alpha\in X_\ast(T)$ such that $\mu+\alpha^\vee$ is dominant and the image of $\alpha$ in $X_\ast(T')$ is the simple root $\alpha_1$ (corresponding to the fundamental coweight $\omega_1^\vee$). We know that $(G,\mu+\alpha^\vee,I)$ satisfies (ING) by Proposition~\ref{prop:ing_2omega1} and Proposition~\ref{prop:ing_2omega1pusomegam}. Moreover, straightforward inspection of $B(G',\mu'+\alpha_1^\vee)$ shows that $[b_{\mu,\indec}]$ is the unique lower cover of $[b_{\mu+\alpha^\vee,\indec}]$ in $B(G,\mu+\alpha^\vee)$. By Theorem~\ref{prop:classification}, we see that $\depth(G,\mu+\alpha^\vee)<2$, so $[b_{\mu+\alpha^\vee,\indec}]$ satisfies (L1BC).

    Let $x\in \Adm_G(\mu)$ be Bruhat maximal such that $(G,\mu,[b_{x,\max}])$ is Hodge-Newton indecomposable. Then we find $y\in \Adm_G(\mu+\alpha^\vee)$ such that $x\lessdot y$ and $(G,\mu+\alpha^\vee,[b_{y,\max}])$ is Hodge-Newton indecomposable using $x\in \Adm_G(\mu+\alpha^\vee)$ and the property (ING).

    If $[b_{y,\max}] = [b_{\mu+\alpha^\vee,\indec}]$, then the property (L1BC) shows that $[b_{x,\max}]\lessdot [b_{y,\max}]$. As we saw above, this forces $[b_{x,\max}] = [b_{\mu,\indec}]$ and we are done.

    Let us assume now that $[b_{y,\max}] < [b_{\mu+\alpha^\vee,\indec}]$, so that $[b_{y,\max}]\leq [b_{\mu,\indec}]$ by the above discussion on lower covers. It suffices to show $y\in \Adm_G(\mu)$.
    \begin{itemize}
        \item Case $\mu' = \omega_2^\vee$: We are done by Lemma \ref{lem:ing_2omega1} (a).
        \item Case $\mu' = \omega_2^\vee+\omega_m^\vee$. The choice of $y$ as above might not work, but we can improve it. All that we have to show is that there exists $x\lessdot y\in \Adm(\mu)$ with $[b_{y,\max}]\leq [b_{\mu+\alpha^\vee,\indec}]$. This is immediate from Proposition~\ref{prop:ing_2omega1pusomegam}. Now we can argue as above; if $[b_{y,\max}] = [b_{\mu + \alpha^\vee,\indec}]$, then we get $[b_{x,\max}] = [b_{\mu,\indec}]$ by (L1BC), and otherwise $[b_{y,\max}]\leq [b_{\mu,\indec}]$ contradicting the choice of $x$.
    \qedhere\end{itemize}
\end{proof}
\begin{proof}[Proof of Theorem~\ref{thm:INGoverview}]
    The four cases follow from Proposition~\ref{prop:ing_2omega1pusomegam}, Proposition~\ref{prop:ing_2omega1} resp.\ Lemma~\ref{lem:L1BC_implies_ING}.
\end{proof}
In general, it is possible that $\depth(G,\mu)<2$ and the property (ING) is not satisfied for $(G,\mu,I)$. One such family is constructed in the example below. It shows that Theorem~\ref{thm:INGoverview} covers almost all cases of $\depth(G,\mu)<2$ where the property (ING) is satisfied.
\begin{example}\label{ex:ing_failure}
Suppose $m\geq 3$, $d=2$ and $\mu = (\omega_1^\vee,\omega_1^\vee)$, where we write $\omega_1^\vee = (1,0,\dotsc,0)$. We sketch how one can see that $(G,\mu,I)$ does not satisfy the property (ING).

Let $\tau = \varepsilon^{\omega_1^\vee} s_1 s_2 \cdots s_m\in \widetilde W_{G'}$ denote the length zero element in $\Adm(\omega_1^\vee)$. We define $w\in W_0'$ to be
\begin{align*}
    w = s_1 s_3 s_5 \cdots s_{m-\varepsilon},\qquad m' = \begin{cases} m-1,&m\text{ even},\\m,&m\text{ odd.}\end{cases}
\end{align*}
We see that $w$ is a product of $\lfloor \frac{m+1}2\rfloor$ pairwise commuting simple reflections.
Then certainly $x_2 := \tau w < \tau s_m s_{m-1}\cdots s_3 s_2 s_1 = \varepsilon^{\omega_1^\vee}\in \Adm(\omega_1^\vee)$. Similarly, $x_1 := \tau^{-1} x_2 \tau \in \Adm(\omega_1^\vee)$. Set $x = (x_1, x_2)\in \Adm(\mu)$. Then
\begin{align*}
    x' := x_2\ast x_1 = (\tau w) \ast (w\tau ) = \tau (w\ast w)\tau = \tau w \tau.
\end{align*}
For $i\in \{1,\dotsc,m-1\}$, we have $\tau s_i \tau^{-1} = s_{i+1}$. Using the Deligne-Lusztig reduction method, one can show $[b_{x',\max}] = [b_{\tau^2 s_{m'},\max}]$. Hence $\langle \nu(b_{x,\max}),2\rho\rangle = \langle \nu(b_{x',\max}),2\rho\rangle$ is $1$ for $m$ even and $0$ for $m$ odd. Using \cite[Theorem~2.23]{He2015}, one concludes $\dim X(\mu,b_{x,\max})\geq m+1$.

We compare this to the predicted dimension formula from Proposition \ref{prop:ING_L1BC} (c). Defining $A$ as in this proposition, we claim that $\max_{y\in A}\ell(y) \leq m+1$. Indeed, given $(y_1, y_2)\in A$, we get that $y' := y_2 y_1\in \Adm(\mu')$ satisfies that $(G,\mu',[\dot y'])$ is Hodge-Newton indecomposable. Then Lemma~\ref{lem:fixed_point_hndec} below shows $\ell_R(\cl(y'))\geq m-1$, hence $\ell_R(\cl(y))\geq m-1$, and we get $\ell(y)\leq m+1$ from Lemma~\ref{lem:adm_refl_length_ineq} below.

Note that $\langle \mu,\rho\rangle = m = \#(\Delta/\sigma)$. A quick inspection of Newton polygons shows that $\defect(b)\geq m-1$ for all $[b]\in B(G',\mu')\cong B(G,\mu)$ such that $(G',\mu',[b])$ is Hodge-Newton indecomposable. Hence we get
\begin{align*}
    &\max_{y\in A}\ell(y) + \#(\Delta/\sigma) +\frac 12\left(-\langle \mu,2\rho\rangle - \langle \nu(b_{x,\max}),2\rho\rangle - \defect(b_{x,\max})\right) \\\leq& (m+1) + m - \frac 12\left(-2m - (m-1)\right)
    = \frac {m+3}2.
\end{align*}
For $m\geq 3$, this is $<m+1 \leq \dim X(\mu,b_{x,\max})$. Hence not only does the property (ING) fail for $(G,\mu,I)$, this failure is inherent to the geometry of the Newton stratification in $I\Adm(\mu)I$.
\end{example}
\subsection{Applications}

We are, of course, most interested in those cases where both (ING) and, in suitable generality, (L1BC) are satisfied so that we can apply Proposition~\ref{prop:ING_L1BC}.

\begin{theorem}\label{thm:depth2GLn}
Let $m\geq 1$ and $G$ be the split group $G = \GL_{m+1}$. Pick a dominant cocharacter $\mu\in X_\ast(T)$, written according to $X_\ast(T)\cong \mathbb Z^{m+1}$. We also choose a standard parahoric subgroup $K = I \widetilde W_J I$ for a set $J\subsetneq \{\alpha_0,\alpha_1,\dotsc,\alpha_m\}$ of simple affine roots. We define the constant $D\in \frac 12 \mathbb Z$ in the following cases:
    \begin{enumerate}[(i)]
        \item If $\mu = (2,0,\dotsc,0)$, then $D:=m$.
        \item If $m\geq 3$, $\mu = (1,1,0,\dotsc,0)$ and $J\cap \{s_0,s_2\}=\emptyset$, then $D:=m$.
        \item If $m\geq 2$, $\mu = (2,0,\dotsc,0,-1)$ and $J\cap \{s_0, s_1\}=\emptyset$, then $D:=\frac{5m-2}2$.
        \item If $m\geq 3$, $\mu = (1,1,0,\dotsc,0,-1)$ and $J\cap \{s_0, s_m\}=\emptyset$, then $D:=\frac{5m-10}2$.
        \item If $m=4$, $\mu = (2,1,0,0,0)$ and $J=\emptyset$, then $D:=7/2$.
    \end{enumerate}
    Then in each of these five cases, the triple $(G,\mu,K)$ satisfies (ING). Moreover, every $[b]\in B(G,\mu)$ such that $(G,\mu,[b])$ is Hodge-Newton indecomposable satisfies (L1BC), and we have
    \begin{align*}
        \dim X(\mu,b,K) = D -\langle \nu(b),\rho\rangle - \frac 12\defect(b).
    \end{align*}
   In case (v), the ADLV $X(\mu,b,I)$ is equidimensional.
\end{theorem}
For the proof, we need the following lemmata.
\begin{lemma}\label{lem:adm_refl_length_ineq}
    For any quasi-split group $G$ with a dominant cocharacter $\mu$ and $x\in\Adm(\mu)$, we have
    \begin{align*}
        \ell(x) \leq \langle \mu,2\rho\rangle - \ell_R(\cl\, x),
    \end{align*}
    where $\cl\, x\in W_0$ denotes the classical part of $x$ and $\ell_R$ is the reflection length of $W_0$.
\end{lemma}
\begin{proof}
    Write $x = w_x\varepsilon^{v_x\mu_x}$ with $\mu_x$ dominant and $v_x$ of minimal length. Using \cite[Proposition~4.12]{Schremmer2024_bruhat} and \cite[Equation~(4.3)]{Milicevic2020}, we get
    \begin{align*}
        \ell(x)&=\langle \mu_x,2\rho\rangle - \ell(v_x)+\ell(w_x v_x)
        \\&\leq \langle \mu-\wt(w_x v_x\Rightarrow v_x),2\rho\rangle - \ell(v_x)+\ell(w_x v_x)
        \\&=\langle\mu,2\rho\rangle - d(w_x v_x\Rightarrow v_x)\\
        &\leq \langle \mu,2\rho\rangle - \ell_R(w_x).\qedhere
    \end{align*}
\end{proof}
\begin{lemma}\label{lem:fixed_point_hndec}
Let $G = \GL_{m+1}$, let $\mu\in X_\ast(T)$ a dominant cocharacter and $x = w_x \varepsilon^{v_x \mu_x}\in \Adm(\mu)$, where $\mu_x\in X_\ast(T)$ is dominant and $w_x, v_x\in W_0$ such that $v_x$ has minimal length in its right $\mathrm{Stab}(\mu_x)$-coset. Write $\mu = (\mu_1,\dotsc,\mu_{m+1})$ and $\mu_x = (\mu_{x,1},\dotsc,\mu_{x,m+1})\in X_\ast(T)\cong\mathbb Z^{m+1}$. Assume that $(G,\mu,[b_{x,\max}])$ is Hodge-Newton indecomposable. Then for $j\in \{1,m+1\}$, we have
\begin{align*}
    \ell_R(w_x) \geq \#\{i\in \{1,\dotsc,m+1\}\mid \mu_{x,i}=\mu_j\}.
\end{align*}
\end{lemma}
\begin{proof}
    Let
    \begin{align*}
        E &:= \{i\in \{1,\dotsc,m+1\}\mid \mu_{x,i}=\mu_j\}
    \end{align*}
    We also let $N>1$ such that the action of $w_x^N$ on $\{1,\dotsc,m+1\}$ is trivial, such that the Newton point of $[\dot x]\in B(G)$ is equal to the dominant representative of the $W_0$-orbit of
    \begin{align*}
    \nu := \frac 1N\sum_{k=1}^N w_x^k v_x \mu_x\in X_\ast(T)\otimes\mathbb Q.
    \end{align*}
    
    Suppose that $v_x^{-1}(E)\subseteq \{1,\dotsc,m+1\}$ contains an orbit $o$ of the $w_x$-action. Then for all $i\in o$, the $i$-th coefficient of $v_x \mu_x$ is equal to $\mu_j$ and moreover $w_x(i)\in o$. Hence $\nu_i = \mu_j$.

    If $j=1$, then $\mu_{i}\leq \mu_1$ and hence $\mu_{x,i}\leq \mu_1$ for all $i\in\{1,\dotsc,m+1\}$. We get $\nu_i\leq \mu_1$ for all $i$, and (by the above considerations) equality holds for some indices. Hence the dominant representative of $\nu$ has entries equal to $\mu_1$, contradicting Hodge-Newton indecomposability.

    Similarly, if $j=m+1$, then $\mu_{m+1} = \min\{\mu_i, \mu_{x,i}, \nu_i\mid i=1,\dotsc,m+1\}$ and one gets again a contradiction.

    The contradiction shows that $v_x^{-1}(E)$ does not contain any $w_x$-orbit. So if we write $w_x$ as a product of reflections $w_x = r_1\cdots r_{\ell}$, there must be for every $i\in E$ some reflection $r_k$ mapping $i$ to an element in $\{1,\dotsc,m+1\}\setminus E$. Hence $\ell \geq \# E$ as required.
\end{proof}
\begin{proof}[Proof of Theorem~\ref{thm:depth2GLn}]
    The properties (ING) and (L1BC) are about combinatorics of the Iwahori-Weyl group and hence invariant under group isogenies. We see that (L1BC) is satisfied in all cases of this theorem by Theorem~\ref{thm:depth2L1BC}. For $K = I W_J I$, we define
    \begin{align*}
        A := \{x\in \Adm_G(\mu)^K \mid (G,\mu,[b_{x,\max}])\text{ is HN-indecomposable}\}.
    \end{align*}

    If $\depth(G,\mu)\leq 1$, then the property (ING) is trivially satisfied as there is only one Hodge-Newton indecomposable $\sigma$-conjuagcy class in $B(G,\mu)$. The remaining cases are detailed by Theorem~\ref{prop:classification}. For the four infinite families, we obtain (ING) from Theorem~\ref{thm:INGoverview}. The last remaining case is $m=5$ and the image of $\mu$ in the adjoint quotient is $\omega_1^\vee+\omega_2^\vee$, e.g.\ the case $\mu = (2,1,0,0,0)$. This case is happily checked with a computer, we obtain $15$ Bruhat maximal elements in $A$ for $K=I$, all of them having length $7$.

    By Proposition~\ref{prop:ING_L1BC}, the constant
    \begin{align*}
        \tilde D := \max_{x\in A}\ell(x) +\#(\Delta/\sigma) - \langle \mu,\rho\rangle\in\frac 12\mathbb Z
    \end{align*}
    satisfies
        \begin{align*}
        \dim X(\mu,b,K) = \tilde D -\langle \nu(b),\rho\rangle - \frac 12\defect(b).
    \end{align*}
    for all Hodge-Newton indecomposable $[b]\in B(G,\mu)$.
    
    It remains to compute the constant $\tilde D$ in the above cases. By Lemmas \ref{lem:adm_refl_length_ineq} and \ref{lem:fixed_point_hndec}, we get for any $x = w_x\varepsilon^{\mu_x}\in A$ that
    \begin{align}\label{eqp528}
        \ell(x)\leq \langle \mu,2\rho\rangle - \# \{i\in\{1,\dotsc,m+1\}\mid \mu_{x,i} = \mu_{j}\}\quad\text{for } j=1\text{ and }m+1.
    \end{align}
    \begin{enumerate}[(i)]
        \item Let $\mu = (2,0,\dotsc,0)$. By Proposition~\ref{prop:ING_L1BC}, we get $\tilde D = \max_{x\in A}\ell(x)$. We first prove that $\ell(x)\leq m$ for all $x\in A$. If $\mu_x = (2,0,\dotsc,0)$, then this follows from \eqref{eqp528}. Similarly, if there is a length zero conjugate $\tau^{-1} x\tau$ of $x$ lying in $W_0 \varepsilon^{(2,0,\dotsc,0)} W_0$, then we get the claim as $A$ is closed under length zero conjugation.

        The only alternative to $(\mu_x)_{\dom} = (2,0,\dotsc,0)$ is $(\mu_x)_{\dom} = (1,1,0,\dotsc,0)$. So assume that $\tau^{-1} x \tau \in W_0 \varepsilon^{(1,1,0,\dotsc,0)} W_0$ for all $\tau\in \Omega$. This means that $x$ lies in the permissible set for $(1,1,0,\dotsc,0)$ by definition, which agrees with the admissible set for the general linear group. Thus, we get
        \begin{eqnarray*}
            \ell(x)&\underset{\text{L\ref{lem:adm_refl_length_ineq}}}\leq& \langle (1,1,0,\dotsc,0),2\rho\rangle - \ell_R(w_x)
            \\&=&\langle \mu,2\rho\rangle - 2-\ell_R(w_x)
            \\&\underset{\text{L\ref{lem:adm_refl_length_ineq}}}\leq &\langle \mu,2\rho\rangle -m-1.
        \end{eqnarray*}

        So we get the desired upper bound for $\ell(x)$ as $x\in A$. In order to construct a lower bound, we note that
        \begin{align*}
            \tau = \varepsilon^{(1,0,\dotsc,0)} s_1\cdots s_m\in\widetilde W
        \end{align*}
        has length zero. Therefore, 
        \begin{align*}
            x = \varepsilon^{(1,0,\dotsc,0)}\tau = \varepsilon^{\mu} s_1\cdots s_m\in\widetilde W
        \end{align*}
        has length $\langle (1,0,\dotsc,0),2\rho\rangle=m = \langle \mu,2\rho\rangle - m$. In particular, $x < \varepsilon^\mu\in \Adm(\mu)$. It is clear that $\cl\, x$ acts fixed point free on $\{1,\dotsc,m+1\}$, so $[\dot x]$ is basic. Lastly, $\varepsilon^{\mu}$ has minimal length in its right $W_{\{s_0, s_2,s_3,\dotsc,s_m\}}$-coset, so $x$ has minimal length in its right $W_0$-coset. Hence $x\in A$ whenever $s_0\notin J$. Up to length zero conjugation, this covers all possibilities for $J$.
        \item Let $m\geq 3, \mu = (1,1,0,\dotsc,0)$ and $J\cap \{s_0,s_2\}=\emptyset$. By Proposition~\ref{prop:ING_L1BC}, we get $\tilde D = 1+\max_{x\in A}\ell(x)$. By \eqref{eqp528}, we get $\max_{x\in A}\ell(x)\leq \langle \mu,2\rho\rangle -(m-1) = m-1$ as desired.

        Now for the converse, we consider the element
        \begin{align*}
            x = \varepsilon^{(1,0,1,0,\dotsc,0)} s_1 s_3 s_4\cdots s_m\in\widetilde W.
        \end{align*}
        As in the previous case, one can prove that $x\in W^{\{s_1, s_3,\dotsc,s_m\}}$. In particular, $x<\varepsilon^{(1,0,1,0,\dotsc,0)}\in \Adm(\mu)$. By assumption on $J$, we get $x\in A$.
        \item Let $m\geq 2, \mu = (2,0,\dotsc,0,-1)$ and $J\cap \{s_0, s_1\}=\emptyset$.  By Proposition~\ref{prop:ING_L1BC}, we get $\tilde D = -m/2 + \max_{x\in A}\ell(x)$. By \eqref{eqp528}, we get $\max_{x\in A}\ell(x)\leq \langle \mu,2\rho\rangle-1 = 3m-1$, and we wish to show equality. For this, we can simply choose $x = \varepsilon^{\mu} s_m$, which lies in $W^{\{s_2,\dotsc,s_m\}}\cap \Adm(\mu)$. We are done as above.
        \item Let $m\geq 2, \mu = (1,1,0,\dotsc,0,-1)$ and $J\cap \{s_0, s_m\}=\emptyset$. By Proposition~\ref{prop:ING_L1BC}, we get $\tilde D = -m/2-1 + \max_{x\in A}\ell(x)$. By \eqref{eqp528}, we get $\max_{x\in A}\ell(x)\leq \langle \mu,2\rho\rangle-2 = 3m-4$, and we wish to show equality. For this, we can choose $x = \varepsilon^\mu s_2 s_1$, which lies in $W^{\{s_1,s_2,\dotsc,s_{m-1}\}}$. We are done as above.
        \item Let $m=4, \mu = (2,1,0,0,0)$ and $J=\emptyset$. An explicit calculation verifies (ING) and $\max_{x\in A}\ell(x)=7$. The property (L1BC) is checked similarly. By Proposition~\ref{prop:ING_L1BC}, we get $\tilde D = -7/2 + \max_{x\in A}\ell(x) = 7/2$. We are done.
    \qedhere\end{enumerate}
\end{proof}
\begin{proposition}
    Let $F'/F$ be unramified of degree $3$, let $G'= \PGL_{m+1, F'}$ and let $\mu,\mu'$ be as in Section~\ref{sec:weilRestrict} such that $\mu' = 2\omega_1^\vee+\omega_m^\vee$ and $\mu$ is minuscule. Then for every $[b]\in B(G,\mu)$ such that $(G,\mu,[b])$ is Hodge-Newton indecomposable, the ADLV $X(\mu,b)$ is equidimensional of dimension
    \begin{align*}
    \dim X(\mu,b) = \frac{5m}2 - \langle \nu(b),\rho\rangle - \frac 12\defect(b).
    \end{align*}
\end{proposition}
\begin{proof}
    Similarly to the proof of Theorem~\ref{thm:depth2GLn}, we get property (ING) from Proposition~\ref{prop:ing_2omega1pusomegam}. Part (c) of it moreover shows that $\ell(x)=\langle \mu,2\rho\rangle$ if $x$ is Bruhat maximal in the Hodge-Newton indecomposable locus $A$. The property (L1BC) follows from Theorem~\ref{thm:depth2L1BC}. Now one can apply Proposition~\ref{prop:ING_L1BC} as above, computing
    \begin{align*}
        &\max_{x\in A}\ell(x) + \#(\Delta/\sigma) - \langle \mu,\rho\rangle = 3m + m - \frac{3m}2 = \frac {5m}2.\qedhere
    \end{align*}
\end{proof}

\begin{CJK*}{UTF8}{gbsn}
\printbibliography
\end{CJK*}
\appendix
\section{Computations}
\begin{refsection}
At a handful of occasions, our work relies on a combination of general arguments and exhaustive computations for finite numbers of exceptional cases. While all these calculations can be certainly done manually, it is granted to use computer algebra software to help with some of the longer or more repetitive calculations. We used Sagemath \cite{sagemath, sage-combinat} for this purpose. Part of the Sagemath software itself are calculations in extended affine Weyl groups of split groups, root and coroot spaces as well as fundamental Coxeter theoretic algorithms (e.g.\ for Bruhat order, descent sets etc.). We rely on these basic functions to implement the more specific algorithms for our subject, as explained below. We have prepared example implementations of the algorithms described below under \url{https://github.com/65537/beyond-minute}. Our example implementations are aimed to provide short, conceptual code with only minor optimizations. Hence, their practical use is tied to small examples (which covers the cases needed for our paper).
\subsection{Enumeration of cocharacters by depth}
The first computational task is, given an adjoint and relatively quasi-simple group $G$ as well as a real number $u\in \mathbb R$, to enumerate all dominant cocharacters $\mu\in X_\ast(T)_{\Gamma_0}$ such that $\depth(G,\mu)<u$. For the proof of Theorem~\ref{prop:classification}, we would set $u=2$ and then filter out all cocharacters of depth $\leq 1$ in a subsequent and straightforward step.

The algorithm is based on a queue (first in first out), which initially contains only one element, namely the cocharacter $0\in X_\ast(T)$. In an iterative fashion, we retrieve and remove the first element $\mu$ of the queue and compute its depth. If the depth is $\geq u$, we discard this cocharacter and continue the iteration. Otherwise, we output $\mu$. Moreover, we compute for each fundamental coweight $\omega$ the cocharacter $\mu+\omega\in X_\ast(T)$, and add it to the queue if it is not already part of it. Once this is done for all fundamental coweights, we continue the iteration. The iteration finishes when the queue is empty.
\subsection{Enumeration of admissible loci}
Given a set $K$ of simple affine roots and a dominant cocharacter $\mu$ for the (say, split) group $G$, we wish to enumerate the set $\Adm(\mu)^K$. For this, we first enumerate the set $\{\varepsilon^{w\mu}\mid w\in W_0\}\cap \widetilde W^K$. Then, we iteratively compute Bruhat lower covers for every element obtained so far, and check if they belong to $\widetilde W^K$. If so, this is an element of the admissible set, and we iteratively compute Bruhat lower covers of those elements again. This process eventually exhausts the full admissible set.
\subsection{Verification of positive Coxeter type and geometric Coxeter type}
We verify positive Coxeter type following the definition. We enumerate the length positive sets as outlined in \cite[Lemma~2.14]{Schremmer2022_newton}. Then, the definition of positive Coxeter pairs \cite[Definition~3.4]{Schremmer2023_coxeter} is straightforward to check.

For geometric Coxeter type, we follow the definition again. More precisely, we have a procedure, which, given any element $x\in \widetilde W$, either produces the symbol \textbf{False} to indicate that $x$ does \emph{not} have geometric Coxeter type, or the set $\{\nu(b)\mid [b]\in B(G)_x\}$ of all Newton points of non-empty ADLV if $x$ does have geometric Coxeter type. This procedure operates recursively, but recursive calls are only made to elements of smaller length in $\widetilde W$ (this guarantees that the recursion terminates).

The procedure operates based on a growing list of elements of $\widetilde W$. It is an invariant that at any point of the algorithm, in the notation of \cite[Section~2]{Nie2025}, every element $y$ of that list satisfies $y\approx_\sigma x$. In particular, $x$ has geometric Coxeter type if and only if $y$ has.

The list is initialized with the sole element $x$. We then iterate through this list, in a way where each iteration can add elements to the end of this list (which are then processed in order). 

At each iteration, we get an element $y$. We compute $\ell(y)-\ell(sy\sigma(s))$ for each simple affine reflection $s$. If this difference is $2$ at some point, we proceed as follows:
\begin{itemize}
	\item Recursively call the procedure on the element $y_2 = sy\sigma(s)$ of length $\ell(y)-2$. If this returns \textbf{False}, then we conclude that $y$ (and hence $x$) cannot have geometric Coxeter type. We also return \textbf{False}. Otherwise, store the result $N_2 := \{\nu(b)\mid [b]\in B(G)_{y_2}\}$.
	\item Recursively call the procedure on the element $y_1 = sy$ of length $\ell(y)-1$. If this returns \textbf{False}, then we conclude that $y$ (and hence $x$) cannot have geometric Coxeter type. We also return \textbf{False}. Otherwise, store the result $N_1 := \{\nu(b)\mid [b]\in B(G)_{y_1}\}$.
	\item If both $y_1$ and $y_2$ are of geometric Coxeter type and $N_1\cap N_2\neq\emptyset$, then $y$ (and hence $x$) cannot have geometric Coxeter type. We return \textbf{False} in this case.
	\item In the remaining case, we get that $y_1$ and $y_2$ have geometric Coxeter type and $N_1\cap N_2=\emptyset$. We conclude that $y$ (and hence $x$) has geometric Coxeter type, and return $N_1\cup N_2$.
\end{itemize}
This concludes the computation in case some simple affine reflection $s$ satisfies $\ell(sy\sigma(s)) = \ell(y)-2$. We may therefore assume from now on that this is not the case. Then, we enumerate all elements of the form $y' = sy\sigma(s)$ for simple affine reflections $s$ with $\ell(y') = \ell(y)$. For each such element $y'$, we add it to the list if it is not already contained therein. Once all $y'$ are processed, we resume the iteration of our list.

Once the list is full enumerated and the computation has not been concluded at any intermediate point, we know that $x$ must have minimal length in its $\sigma$-conjugacy class by \cite{He2014b}. Then, it remains to check whether or not $x$ is of \emph{minimal Coxeter type} according to \cite[Definition~5.5]{Nie2025}. This condition is most easily checked using \cite[Proposition~6.6]{Nie2025}. We compute $\nu(\dot x)$ by constructing a suitable $\sigma$-twisted power $x\sigma(x)\sigma^2(x)\cdots \sigma^N(x)$, and then compute the remaining notions from that proposition using the standard definitions. According to this criterion, we either return \textbf{False} or $\{\nu(\dot x)\}$. This finally completes the procedure.
\subsection{Verification of (L1BC) and (ING)}\label{app:a4}
Given an admissible set $\Adm(\mu)^K$, which we enumerate as outlined above, we moreover obtain the graph of all Bruhat covers $x\lessdot x'$ in $\Adm(\mu)^K$ from the former algorithm. We compute $\nu(\dot x)$ for each $x$ in this graph as outlined above, and note which elements satisfy $\ell(x) = \langle \nu(\dot x),2\rho\rangle$. These are the straight (or fundamental) elements. We moreover note which of these fundamental elements satisfy $\nu(\dot x) \in \nu(\varepsilon^\mu) + \sum_{\alpha\in \Delta} \mathbb Q_{<0} \alpha^\vee$, i.e.\ whether or not $(G,\mu,[\dot x])$ is Hodge-Newton indecomposable. By \cite[Corollary~5.6]{Viehmann2014} and He's partial conjugation method \cite{He2007}, we note for each $x\in \Adm(\mu)^K$ that
\begin{align*}
	\nu(b_{x,\max}) = \max\{\nu(\dot y)\mid y\in \Adm(\mu)^K\text{ fundamental and }y\leq x\}.
\end{align*}
We compute this for each $x\in \Adm(\mu)^K$.

The verification of (ING) is not straightforward. We can easily filter out all elements $x\in\Adm(\mu)^K$ such that $(G,\mu,[\dot x])$ is Hodge-Newton decomposable. Among the remaining elements, we compute the Bruhat maximal ones and check if they all share the same generic Newton point.

In order to check if a certain class $[b]\in B(G)$ satisfies (L1BC), we may assume that $K=I$ and choose $\mu$ such that every fundamental element $x\in\widetilde W$ with $[\dot x]= [b]$ lies in $\Adm(\mu)$. Then, the above gives us all the information we need to check (L1BC) according to the definition. Namely, for each fundamental element $x\in\Adm(\mu)$ with $[\dot x] = [b]$ and each $x'\lessdot x$, we already know $[b_{x',\max}]$ from the above computation. It remains to check if $\ell(b_{x',\max}, [b])$ is always equal to $1$.
\printbibliography
\end{refsection}
\end{document}